\documentclass[12pt]{article}
\usepackage{latexsym}
\usepackage{amsthm}
\usepackage{amsmath,amssymb,amsfonts,bm,amsbsy,bbm}
\usepackage{epsfig}
\usepackage{graphicx,rotating,lscape}
\usepackage{booktabs}
\usepackage{hyperref}
\usepackage{verbatim}
\usepackage{natbib}
\usepackage{multirow,color}
\usepackage{geometry}
\usepackage{setspace}
\usepackage{subfigure}
\usepackage{subcaption}

\usepackage{bbm}
\usepackage[flushleft]{threeparttable}
\usepackage{enumitem}

\newtheorem{Pro}{Proposition}

\newtheorem{Th}{\underline{\bf Theorem}}

\newtheorem{Rem}{\underline{\bf Remark}}

\newcommand{\indep}{\;\, \rule[0em]{.03em}{.67em} \hspace{-.25em}
\rule[0em]{.65em}{.03em}
\hspace{-.25em}\rule[0em]{.03em}{.67em}\;\,}

\def\wh{\widehat}
\def\wt{\widetilde}
\def\n{\nonumber}   
\def\var{\mbox{var}}
\def\cov{\mbox{cov}}
\def\corr{\mbox{corr}}

\def\log{\mbox{log}}
\def\trans{^{\top}}

\def\sumi{\sum_{i=1}^n}
\def\sumia{\sum_{i=1}^{n_1}}
\def\sumib{\sum_{i=n_1+1}^n}

\def\0{{\bf 0}}
\def\1{{\bf 1}}
\def\eff{_{\rm eff}}

\def\b{{\bf b}}

\def\X{{\bf X}}
\def\b0{{\bf 0}}

\def\Z{{\bf Z}}
\def\z{{\bf z}}

\def\bSigma{{\boldsymbol \Sigma}}

\def\calH{{\cal H}}

\def\bse{\begin{eqnarray*}}
\def\ese{\end{eqnarray*}}
\def\be{\begin{eqnarray}}
\def\ee{\end{eqnarray}}
\def\bsq{\begin{equation*}}
\def\esq{\end{equation*}}
\def\bq{\begin{equation}}
\def\eq{\end{equation}}

\definecolor{mygreen}{RGB}{34, 139, 34} % Custom green

\allowdisplaybreaks

\begin{document}
\begin{center}
{\Large{\bf Semiparametric Efficiency of Residual Correlation Testing under Gaussian Additive Noise Models}} \\
\vspace{0.5cm}
Yin Tang\textsuperscript{a}, Yanyuan Ma\textsuperscript{b}, Bing Li\textsuperscript{b}\\
\textsuperscript{a}\textit{Dr. Bing Zhang Department of Statistics, University of Kentucky, USA}\\
\textsuperscript{b}\textit{Department of Statistics, Pennsylvania State University, USA}
\end{center}

\baselineskip=18pt

\begin{abstract}
This paper studies conditional independence testing under the Gaussian additive noise model (GANM), where two variables are modeled as nonlinear functions of covariates with independent bivariate Gaussian regression errors. Under this framework, conditional independence can be characterized by the correlation coefficient of the regression errors, which motivates a test based on the Pearson correlation coefficient computed from the fitted residuals. Despite its simple form, the asymptotic behavior and statistical efficiency of the resulting test have not been well understood. In this paper, we develop the semiparametric efficiency theory under GANM and show, surprisingly, that the efficient estimator coincides exactly with the ordinary residual Pearson correlation estimator. We further establish the asymptotic properties of the proposed test and develop the corresponding inference procedure. Simulation studies demonstrate that the proposed method achieves near-oracle efficiency and competitive empirical power while maintaining valid Type I error control. We further apply the proposed test to conditional dependence analysis of U.S. stock returns.
\end{abstract}


\section{Introduction}

Consider the case where $X$ and $Y$ are random variables, and $\Z$ is a random vector. 
We want to test 
\be\label{eq:h0-ci}
H_0: \quad X \indep Y | \Z
\ee
against the alternative that $X$ and $Y$ are dependent conditioning on $\Z$. See \cite{dawid1979conditional} for detailed explanations of conditional independence.
Conditional independence test plays a central role in many statistical fields, including sufficient dimension reduction \citep{li2018sufficient,ma2013review}, statistical graphical models \citep{lauritzen1996graphical,koller2009probabilistic}, and causal inference \citep{ding2024course,pearl2016causal}. 

In the multivariate Gaussian case, conditional independence can be characterized by partial correlations \citep{dempster1972covariance,baba2004partial}, which can be interpreted through linear regression. Specifically, the partial correlation between $X$ and $Y$ given $\Z$ equals the ordinary Pearson correlation between the regression errors from the linear regressions of $X$ on $\Z$ and $Y$ on $\Z$. Consequently, in Gaussian linear models, testing \eqref{eq:h0-ci} reduces to testing whether these regression errors are uncorrelated, with the unknown errors replaced in practice by their fitted residuals. Such a characterization is often used in multivariate analysis \citep{anderson2003introduction} and graphical models \citep{lauritzen1996graphical}. 

In the partial correlation test, it is assumed that the mean functions of $X$ and $Y$ given $\Z$ are both linear. An extension is the additive noise model (ANM), in which the conditional mean functions are allowed to be nonlinear while the regression errors remain additive.
Specifically, $X$ and $Y$ are some deterministic functions of $\Z$ plus additive regression errors, i.e.,
\be\label{eq:anm}
X=m_x(\Z)+\epsilon_x,\qquad Y=m_y(\Z)+\epsilon_y,
\ee
where $\epsilon_x$ and $\epsilon_y$ are zero-mean regression errors independent of $\Z$. In particular, the ANM with Gaussian noise can be viewed as a nonlinear extension of the classical linear Gaussian model.
In practice, ANM is often applied to causal discovery
\citep{shimizu2006linear,hoyer2008nonlinear,peters2011causal,peters2014causal}.

As pointed out in Section 3.1.5 of \cite{li2019nonparametric}, testing conditional independence under the ANM framework can be reduced to testing unconditional independence between the corresponding regression errors \citep{zhang2017causal,zhang2019measuring}. In practice, this is typically carried out in two steps: (1) regress $X$ on $\Z$ and $Y$ on $\Z$ to estimate the regression functions; and (2) test for unconditional independence between the resulting fitted residuals
\bse
\wh{\epsilon}_x = X-\wh m_x(\Z), \qquad
\wh{\epsilon}_y = Y-\wh m_y(\Z).
\ese

When ANM is violated, some nonparametric tests for conditional independence have been proposed, including the discretization-based conditional independence test \citep{huang2010testing}, the metric-based approaches via suitable discrepancy criteria, \citep{su2008nonparametric,huang2016flexible,wang2015conditional},
the permutation-based kernel conditional independence test \citep{doran2014permutation}, and the transformation-based mutual independence framework \citep{cai2022distribution}. 
See \cite{li2019nonparametric} for a review. 
In addition, under the reproducing kernel Hilbert space (RKHS) framework, \cite{zhang2012kernelbased} proposed the kernel conditional independence test (KCIT) based on the conditional covariance operator of \cite{fukumizu2004dimensionality,fukumizu2007kernel}; see also \cite{sheng2023distance} and \cite{tang2026kernel} for further theoretical developments.
Furthermore, \cite{strobl2019approximate} introduced two tests based on random Fourier features, the randomized conditional independence test (RCIT) and the randomized conditional correlation test (RCoT).

In this paper, we focus on the Gaussian additive noise model (GANM). That is, we assume that $(X,Y,\Z)$ satisfy the model \eqref{eq:anm}, where the regression errors $(\epsilon_x,\epsilon_y)$ follow a bivariate Gaussian distribution. In this setting, following the idea of \cite{zhang2017causal}, conditional independence between $X$ and $Y$ given $\Z$ can be tested using Pearson's correlation coefficient computed from the fitted residuals $\wh\epsilon_x$ and $\wh\epsilon_y$.

However, despite this seemingly simple construction, several important theoretical questions remain largely unresolved. First, after replacing the unobserved regression errors $(\epsilon_x,\epsilon_y)$ by fitted residuals $(\wh\epsilon_x,\wh\epsilon_y)$ obtained from nonparametric regressions, the asymptotic distribution of the resulting Pearson correlation estimator is no longer immediate, especially when flexible machine learning methods are employed. Second, it remains unclear whether the resulting residual-correlation-based test retains statistical efficiency after the nuisance regression functions are estimated nonparametrically. In particular, it is important to understand whether the estimation error from the nonparametric regressions affects the first-order asymptotic behavior of the test statistic, and whether the procedure can still achieve the same asymptotic efficiency as the oracle procedure based on the true regression errors. Third, suitable convergence rate conditions on the nonparametric regression estimators are needed to guarantee the validity of the asymptotic inference. Addressing these issues is therefore essential for establishing a rigorous theoretical foundation for residual-correlation-based conditional independence testing under the GANM framework.

Surprisingly, under GANM, the semiparametrically efficient estimator induced by the efficient influence function coincides exactly with the ordinary Pearson correlation coefficient computed from the fitted residuals. Thus, despite its simple form, the residual correlation estimator remains asymptotically efficient even when the nuisance regression functions are estimated nonparametrically using flexible machine learning methods. Moreover, by combining sample splitting and cross-fitting, the resulting asymptotic theory only requires suitable convergence rate conditions on the nuisance regression estimators, without requiring explicit asymptotic expansions of the underlying machine learning procedures.
To establish these results, we develop the semiparametric efficiency theory under GANM and derive the asymptotic linearity and asymptotic normality of the resulting estimator.

Although our work is also based on residuals from nonparametric
regression, its scope is narrower from that of
\cite{chang2026testing}, which develops a general framework for
high-dimensional conditional independence testing. By 
  focusing on the more specific class of
  Gaussian additive noise models, we are able to recast the conditional
  independence test to the problem of testing zero residual correlation.
  This approach further allows us to establish the semiparametric efficiency
theory for the residual Pearson correlation estimator, which leads to
a most powerful test. In summary, using different tools, we study a
more specific problem more thoroughly, whereas 
\cite{chang2026testing} works in a more general framework without
studying optimality.

The rest of the paper is organized as follows. 
Section \ref{sec:test} illustrates the setting of GANM and the intuition of the residual correlation estimator. 
Section \ref{sec:efficiency} introduces the semiparametric efficiency theory of the estimator and gives its asymptotic properties. 
Section \ref{sec:simulation} conducts some simulation studies of the proposed estimator with some comparisons with some other conditional independence tests. 
Section \ref{sec:application} applies the proposed test to a real dataset on U.S. stocks.
To save space, all proofs and additional simulations tables and figures  are presented in Supplementary Materials.

\section{Model and Test Construction}\label{sec:test}

\subsection{Gaussian Additive Noise Model}
Consider the Gaussian additive noise model (GANM), where $X, Y, \Z$ satisfy the ANM in \eqref{eq:anm}, where the regression errors $(\epsilon_x,\epsilon_y)\trans\sim N(\0,\bSigma)$ where
\bse
\bSigma=\left(\begin{matrix}
    \sigma_x^2 & \rho\sigma_x\sigma_y \\
    \rho\sigma_x\sigma_y & \sigma_y^2
    \end{matrix}\right).
\ese
Here, $\sigma_x^2$ and $\sigma_y^2$ are the variances of $\epsilon_x$ and $\epsilon_y$, respectively, and $\rho$ is the correlation coefficient between $\epsilon_x$ and $\epsilon_y$, i.e., $\rho=\corr(\epsilon_x,\epsilon_y)$.

Clearly, under GANM, since $X$ and $Y$ depend on $\Z$ only through the conditional mean function $m_x(\Z)$ and $m_y(\Z)$, we know that $X \indep Y | \Z$ if and only if the regression errors are independent, i.e., $\epsilon_x \indep \epsilon_y$, which, under the joint Gaussian assumption for $(\epsilon_x,\epsilon_y)$, is further equivalent to $\corr(\epsilon_x,\epsilon_y)=0$, i.e., $\rho=0$. Therefore, to test whether \eqref{eq:h0-ci} holds, we can equivalently test 
\be\label{eq:hypothesis}
H_0: \rho=0\quad \text{vs.}\quad H_1:\rho\neq 0.
\ee
In this way, under GANM, testing conditional independence of $X$ and $Y$ given $\Z$ is equivalent to testing the uncorrelatedness of the regression errors $\epsilon_x$ and $\epsilon_y$.

\subsection{Residual Correlation Estimator}\label{sec:estimator}

We first consider the oracle case when the regression functions $m_x$ and $m_y$ are known. In this case, we could directly calculate the true regression errors
\bse
\epsilon_{xi}=X_i-m_x(\Z_i), \quad
\epsilon_{yi}=Y_i-m_y(\Z_i), \qquad
i=1,\dots,n,
\ese
Based on $(\epsilon_{x1},\epsilon_{y1}), \dots, (\epsilon_{xn},\epsilon_{yn})$, we then estimate $\rho$ by the Pearson correlation coefficient 
\bse
\wh\rho_{\rm{orc}}
=
\frac{\wh\sigma_{xy,\rm{orc}}}{\wh\sigma_{x,\rm{orc}}\wh\sigma_{y,\rm{orc}}}
=
\frac{
\sumi \epsilon_{xi}\epsilon_{yi}
}{
(\sumi \epsilon_{xi}^2)^{1/2}
(\sumi \epsilon_{yi}^2)^{1/2}
}.
\ese
Here, 
\bse
\wh\sigma_{xy,\rm{orc}}=n^{-1}\sumi \epsilon_{xi}\epsilon_{yi}, \quad
\wh\sigma_{x,\rm{orc}}^2=n^{-1}\sumi \epsilon_{xi}^2,
\quad
\wh\sigma_{y,\rm{orc}}^2=n^{-1}\sumi \epsilon_{yi}^2
\ese
are estimators of $\sigma_x^2 = \var(\epsilon_x)$, $\sigma_y^2 = \var(\epsilon_y)$, and $\sigma_{xy} = \cov(\epsilon_x,\epsilon_y)$, respectively.
Since $(\epsilon_{xi},\epsilon_{yi})\trans$ are i.i.d. bivariate Gaussian random vectors, by Theorem 5.1.6 of \cite{muirhead1982aspects}, 
the asymptotic distribution of $\wh\rho_{\rm{orc}}$ is
\be\label{eq:oracle-clt}
n^{1/2}(\wh\rho_{\rm{orc}}-\rho)
\overset{d}{\to}
N\{0,(1-\rho^2)^2\}.
\ee
In particular, under the null hypothesis $H_0:\rho=0$, the asymptotic null distribution of $\wh\rho$ is
\bse
n^{1/2}\wh\rho_{\rm{orc}}
\overset{d}{\to}
N(0,1).
\ese

In practice, the regression functions $m_x$ and $m_y$ are unknown, and we define $\wh m_x$ and $\wh m_y$ as nonparametric estimators of $m_x$ and $m_y$, respectively. Based on $\wh m_x$ and $\wh m_y$, We calculate the fitted residuals as
\be\label{eq:hat-eps}
\wh\epsilon_{xi}=X_i-\wh m_x(\Z_i),
\quad
\wh\epsilon_{yi}=Y_i-\wh m_y(\Z_i),
\qquad
i=1,\dots,n,
\ee
and the residual correlation estimator is given by the Pearson correlation coefficient based on $(\wh\epsilon_{x1},\wh\epsilon_{y1}), \dots, (\wh\epsilon_{xn},\wh\epsilon_{yn})$ is 
\be\label{eq:rho-hat-pearson}
\wh\rho
=
\frac{\wh\sigma_{xy}}{\wh\sigma_x\wh\sigma_y}
=
\frac{
\sumi\wh\epsilon_{xi}\wh\epsilon_{yi}
}{
(\sumi \wh\epsilon_{xi}^2)^{1/2}
(\sumi \wh\epsilon_{yi}^2)^{1/2}
}.
\ee
Here, 
\be\label{eq:sigma-hat-full}
\wh\sigma_{xy}=n^{-1}\sumi \wh\epsilon_{xi}\wh\epsilon_{yi}, \quad
\wh\sigma_{x}^2=n^{-1}\sumi \wh\epsilon_{xi}^2,
\quad
\wh\sigma_{y}^2=n^{-1}\sumi \wh\epsilon_{yi}^2
\ee
are corresponding estimators of $\sigma_{xy}$, $\sigma_x^2$ and $\sigma_y^2$ when $m_x$ and $m_y$ are estimated.

To make the asymptotic theory applicable to a broad class of machine learning methods for estimating $m_x$ and $m_y$, without requiring explicit asymptotic expansions of the nuisance regression estimators, we employ sample splitting and cross-fitting. Specifically, one part of the data is used to estimate the regression functions, while the other part is used to construct the residual correlation estimator $\wh\rho$. The two resulting estimators are then averaged to obtain the final estimator. More details are given in Section \ref{sec:splitting}.

\section{Semiparametric Efficiency Theory}\label{sec:efficiency}

\subsection{Semiparametric Model and Likelihood}

Since $(\epsilon_x,\epsilon_y) \indep \Z$, the conditional density of $X,Y|\Z$ can be written in terms of the joint density of the regression errors $(\epsilon_x,\epsilon_y)$. That is, 
\bse
f_{X,Y|\Z}(x,y,\z) 
= f_{\epsilon_x,\epsilon_y}\{x-m_x(\z),y-m_y(\z)\}
= f_{\epsilon_x,\epsilon_y}(\epsilon_x,\epsilon_y),
\ese
where 
\be\label{eq:exey}
\epsilon_x=x-m_x(\z), \quad \epsilon_y=y-m_y(\z).
\ee
Then, for a realization $(x,y,\z)$, the likelihood function can be written as
\bse
f(x,y,\z)
=f_{X,Y|\Z}(x,y,\z)f_\Z(\z)
=f_{\epsilon_x,\epsilon_y}(\epsilon_x,\epsilon_y)f_\Z(\z),
\ese
where $f_\Z$ is the marginal distribution of $\Z$. 
Furthermore, since $(\epsilon_x,\epsilon_y)\trans\sim N(\0,\bSigma)$, we plug in the Gaussian density function to get
\be\label{eq:likelihood}
f(x,y,\z)
=\frac{1}{2\pi\sigma_x\sigma_y\sqrt{1-\rho^2}}\exp\left\{-\frac{1}{2(1-\rho^2)}\left(\frac{\epsilon_x^2}{\sigma_x^2}-2\rho\frac{\epsilon_x \epsilon_y}{\sigma_x\sigma_y}+\frac{\epsilon_y^2}{\sigma_y^2}\right)\right\}f_\Z(\z),
\ee
where $\epsilon_x$ and $\epsilon_y$ are given by \eqref{eq:exey}.

Note that, in the likelihood, our parameter of interest is $\rho$, and all other parameters, including $\eta=\{\sigma_x,\sigma_y,m_x(\cdot),m_y(\cdot),f_\Z(\cdot)\}$, can be viewed as nuisance parameters. As we see, in the nuisance parameters, $\sigma_x$ and $\sigma_y$ are parametric, and the rest ones $m_x(\cdot)$, $m_y(\cdot)$ and $f_\Z(\cdot)$ are nonparametric.

The log likelihood is given by
\bse
\log f(x,y,\z) &=& \log(2\pi) - \log(\sigma_x) - \log(\sigma_y) - \frac{1}{2}\log(1-\rho^2)\\
&& -\frac{1}{2(1-\rho^2)}\left(\frac{\epsilon_x^2}{\sigma_x^2}-2\rho\frac{\epsilon_x \epsilon_y}{\sigma_x\sigma_y}+\frac{\epsilon_y^2}{\sigma_y^2}\right) +\log\{f_\Z(\z)\}.
\ese

\subsection{Nuisance Tangent Spaces}

To derive the efficient score for $\rho$, we first calculate the nuisance tangent space associated with each of the the nuisance parameter in $\eta=\{\sigma_x,\sigma_y,m_x(\cdot),m_y(\cdot),f_\Z(\cdot)\}$. We then derive its orthogonal decomposition and its orthogonal complement, on which we will need to project the score function with respect to $\rho$ to find the efficient score \citep{bcrw,tsiatis}.

Denote $\calH = \{f(x,y,\z): E\{f(X,Y,\Z)\}=0, \var\{f(X,Y,\Z)\}<\infty \}$ the Hilbert space of all possible influence functions. We first give the nuisance tangent space as the next proposition.

\begin{Pro}\label{prop:space}
Under model \eqref{eq:likelihood} with parameter of interest $\rho$ and nuisance parameters $\eta=\{\sigma_x,\sigma_y,m_x(\cdot),m_y(\cdot),f_\Z(\cdot)\}$, let $\Lambda_j$, $j=1,\dots,5$, denote the nuisance tangent space associated with the corresponding nuisance parameter component:
\bse
\Lambda_1:\sigma_x,\qquad
\Lambda_2:\sigma_y,\qquad
\Lambda_3:m_x,\qquad
\Lambda_4:m_y,\qquad
\Lambda_5:f_Z.
\ese
Then the nuisance tangent space of model \eqref{eq:likelihood} is
\bse
\Lambda=\Lambda_1+\Lambda_2+\Lambda_3+\Lambda_4+\Lambda_5,
\ese
where
\bse
\Lambda_1 &=& \left\{c_1\left(-\frac{1}{\sigma_x}+\frac{1}{1-\rho^2}\frac{\epsilon_x^2}{\sigma_x^3}-\frac{\rho}{1-\rho^2}\frac{\epsilon_x\epsilon_y}{\sigma_x^2\sigma_y}\right)\right\},\\
\Lambda_2 &=& \left\{c_2\left(-\frac{1}{\sigma_y}+\frac{1}{1-\rho^2}\frac{\epsilon_y^2}{\sigma_y^3}-\frac{\rho}{1-\rho^2}\frac{\epsilon_x\epsilon_y}{\sigma_x\sigma_y^2}\right)\right\},\\
\Lambda_3 &=& \left\{\frac{1}{1-\rho^2}\left(\frac{\epsilon_x}{\sigma_x^2}-\rho\frac{\epsilon_y}{\sigma_x\sigma_y}\right)B_1(\z)\right\},\\
\Lambda_4 &=& \left\{\frac{1}{1-\rho^2}\left(\frac{\epsilon_y}{\sigma_y^2}-\rho\frac{\epsilon_x}{\sigma_x\sigma_y}\right)B_2(\z)\right\},\\
\Lambda_5 &=& \{a(\z):E(a)=0\}.
\ese
\end{Pro}

Based on the nuisance tangent space in Proposition \ref{prop:space}, we derive the orthogonal complement of the nuisance tangent space as the following proposition.

\begin{Pro}\label{prop:ortho-comp}
Let $\Lambda^\perp$ denote the orthogonal complement of the nuisance tangent space
$\Lambda$ in $\calH$, where
$\Lambda$ is given in Proposition \ref{prop:space}.
Then
\bse
\Lambda^\perp 
&=& \left\{b(x,y,\z): E\left\{\left(\wt\epsilon_x^2-\rho\wt\epsilon_x\wt\epsilon_y\right)b(X,Y,\Z)\right\}=0,E\left\{\left(\wt\epsilon_y^2-\rho\wt\epsilon_x\wt\epsilon_y\right)b(X,Y,\Z)\right\}=0,\right.\\
&& \left.E\left\{\wt\epsilon_x b(X,Y,\Z)\Big|\Z\right\}=0, E\left\{\wt\epsilon_y b(X,Y,\Z)\Big|\Z\right\}=0,E\left\{b(X,Y,\Z)|\Z\right\}=0\right\},
\ese
where $\wt\epsilon_x=\epsilon_x/\sigma_x$ and $\wt\epsilon_y=\epsilon_y/\sigma_y$.
\end{Pro}

\subsection{Efficient Score and Influence Function}

In the next theorem, we give the efficient score function, which is defined as the projection of the score function with respect to $\rho$ onto $\Lambda^\perp$, which is given in Proposition \ref{prop:ortho-comp}. 
The efficient score is written as $S\eff = \Pi(S_\rho|\Lambda^\perp)$, where $\Pi$ is the projection operator. Based on the efficient score, we can further derive the efficient Fisher information and the efficient influence function as the next theorem.

\begin{Pro}\label{prop:eif}
Under model \eqref{eq:likelihood}, the efficient score for $\rho$ is
\be\label{eq:seff}
S\eff
= -\frac{1}{2(1-\rho^2)^2}\left(\rho\wt\epsilon_x^2-2\wt\epsilon_x\wt\epsilon_y+\rho\wt\epsilon_y^2\right).
\ee
Furthermore, the semiparametric efficiency bound for estimating $\rho$ is
\bse
E(S\eff^2)^{-1} = (1-\rho^2)^2,
\ese
and the efficient influence function for $\rho$ is
\be\label{eq:eif}
\phi\eff(x,y,\z)=-\frac{1}{2}\left(\rho\wt\epsilon_x^2-2\wt\epsilon_x\wt\epsilon_y+\rho\wt\epsilon_y^2\right).
\ee
\end{Pro}
As shown in \eqref{eq:seff}, the efficient score function has a simple quadratic form involving only the standardized regression errors, and the corresponding efficient estimating equation leads to a simple explicit estimator in Section \ref{sec:eff-ee}.

\subsection{Efficient Estimating Equation}\label{sec:eff-ee}

Based on semiparametric theory, the efficient estimator can be obtained through implementing
\bse
\sumi S\eff(X_i,Y_i,\Z_i) = 0.
\ese
Plugging in the efficient score function \eqref{eq:seff}, we have
\be\label{eq:est-eq}
\sumi \left(\rho \frac{\epsilon_{xi}^2}{\sigma_x^2} - 2 \frac{\epsilon_{xi}\epsilon_{yi}}{\sigma_x\sigma_y} + \rho \frac{\epsilon_{yi}^2}{\sigma_y^2} \right) =0 .
\ee
The solution to the estimating equation \eqref{eq:est-eq} can be explicitly written as
\be\label{eq:rho-hat-ee}
\wh\rho = \frac{2 \sigma_x\sigma_y\sumi \epsilon_{xi}\epsilon_{yi}}{\sigma_y^2\sumi\epsilon_{xi}^2+\sigma_x^2\sumi\epsilon_{yi}^2}.
\ee

In practice, all nuisance parameters need to be estimated. Firstly, $m_x$ and $m_y$ are estimated by $\wh m_x$ and $\wh m_y$. Then, based on $\wh m_x$ and $\wh m_y$, we need to calculate the residuals $\wh\epsilon_{xi}$ and $\wh\epsilon_{yi}$ in \eqref{eq:hat-eps} as the estimated versions of regression errors $\epsilon_{xi}$ and $\epsilon_{yi}$. Next, based on the residuals $\wh\epsilon_{xi}$ and $\wh\epsilon_{yi}$, we can estimate the unknown parameters $\sigma_x$ and $\sigma_y$ by $\wh\sigma_x$ and $\wh\sigma_y$ as defined in \eqref{eq:sigma-hat-full}.

Surprisingly, after replacing all nuisance parameters in \eqref{eq:rho-hat-ee} by their empirical estimators, we obtain
\bse
\wh\rho = \frac{2 (n^{-1}\sumi\wh\epsilon_{xi}^2)^{1/2}(n^{-1}\sumi\wh\epsilon_{yi}^2)^{1/2}\sumi \wh\epsilon_{xi}\wh\epsilon_{yi}}{n^{-1}\sumi\wh\epsilon_{yi}^2\sumi\wh\epsilon_{xi}^2+n^{-1}\sumi\wh\epsilon_{xi}^2\sumi\wh\epsilon_{yi}^2}
=
\frac{
\sumi\wh\epsilon_{xi}\wh\epsilon_{yi}
}{
(\sumi \wh\epsilon_{xi}^2)^{1/2}
(\sumi \wh\epsilon_{yi}^2)^{1/2}
},
\ese
which coincides with residual correlation estimator as in \eqref{eq:rho-hat-pearson}. Therefore, the residual correlation estimator in \eqref{eq:rho-hat-pearson} can be equivalently interpreted as the estimator induced by the efficient influence function under GANM.

\subsection{Sample Splitting and Cross-Fitting}\label{sec:splitting}

Note that the nuisance regression functions $m_x$ and $m_y$ need to be estimated nonparametrically from the data. To facilitate the asymptotic analysis, we employ sample splitting and cross-fitting as introduced in Section \ref{sec:estimator}. This construction separates the nuisance estimation step from the estimation step for $\rho$, which avoids requiring explicit asymptotic expansions of the nonparametric regression estimators and allows the asymptotic theory to rely primarily on suitable convergence rate conditions.

Specifically, we use the first $n_1$ data to estimate $m_x$ and $m_y$, denote the corresponding estimators by $\wh m_{x1}$ and $\wh m_{y1}$, and use the remaining $n_2$ data to estimate $\sigma_x^2$, $\sigma_y^2$, $\sigma_{xy}$ and $\rho$. That is,
\bse
\wh \sigma_{xy2} = n_2^{-1} \sumib\wh\epsilon_{xi1}\wh\epsilon_{yi1},\quad
\wh \sigma_{x2}^2 = n_2^{-1} \sumib\wh\epsilon_{xi1}^2, \quad
\wh \sigma_{y2}^2 = n_2^{-1} \sumib\wh\epsilon_{yi1}^2, 
\ese
where
\bse
\wh\epsilon_{xi1} = X_i- \wh m_{x1}(\Z_i), \quad 
\wh\epsilon_{yi1} = Y_i- \wh m_{y1}(\Z_i), \qquad
i=n_1+1,\dots,n.
\ese
and the final estimator of $\rho$ is
\be\label{eq:hat-rho2}
\wh\rho_2
= \frac{\wh\sigma_{xy2}}{\wh\sigma_{x2}\wh\sigma_{y2}}
= \frac{\sumib\wh\epsilon_{xi1}\wh\epsilon_{yi1}}{(\sumib\wh\epsilon_{xi1}^2)^{1/2} (\sumib\wh\epsilon_{yi1}^2)^{1/2}}
.
\ee
We can also switch the roles of the two parts of data, using the later $n_2$ data to estimate $m_x$ and $m_y$, which are denoted similarly by $\wh m_{x2}$ and $\wh m_{y2}$. We then use the first $n_1$ data to estimate $\sigma_x^2$, $\sigma_y^2$ and $\sigma_{xy}$, which are denoted similarly by $\wh\sigma_{x1}^2$, $\wh\sigma_{y1}^2$ and $\wh\sigma_{xy1}$. Then we construct the estimator $\wh\rho_1$ as
\bse
\wh\rho_1
= \frac{\wh\sigma_{xy1}}{\wh\sigma_{x1}\wh\sigma_{y1}}
= \frac{\sumia\wh\epsilon_{xi2}\wh\epsilon_{yi2}}{(\sumia\wh\epsilon_{xi2}^2)^{1/2}(\sumia\wh\epsilon_{yi2}^2)^{1/2}},
\ese
where
\bse
\wh \sigma_{xy1} = n_1^{-1} \sumia\wh\epsilon_{xi2}\wh\epsilon_{yi2},\quad
\wh \sigma_{x1}^2 = n_1^{-1} \sumia\wh\epsilon_{xi2}^2, \quad
\wh \sigma_{y1}^2 = n_1^{-1} \sumia\wh\epsilon_{yi2}^2, 
\ese
with
\bse
\wh\epsilon_{xi2} = X_i- \wh m_{x2}(\Z_i), \quad 
\wh\epsilon_{yi2} = Y_i- \wh m_{y2}(\Z_i), \qquad
i=1,\dots,n_1.
\ese
Taking $n_1=n_2=n/2$ for simplicity, we can average the two estimators $\wh\rho_1$ and $\wh\rho_2$ to construct the final estimator 
\be\label{eq:rho-hat-final}
\wh\rho = (\wh\rho_1+\wh\rho_2)/2.
\ee

\subsection{Asymptotic Expansion and Efficiency}

In the next theorem, we give the asymptotic properties of $\wh\rho_2$. The analogous properties for $\wh\rho_1$ can be similarly derived.

\begin{Th}\label{thm:rho2-if}
Under model \eqref{eq:likelihood}, suppose that the nonparametric regression estimators $\wh m_{xk}$ and $\wh m_{yk}$ satisfy the convergence rate assumption:
\be\label{eq:rate-cond-1}
\|\wh m_{xk} - m_x\|_2 = o_p(n_1^{-1/4}), \qquad \|\wh m_{yk} - m_y\|_2 = o_p(n_1^{-1/4}), \qquad k=1,2,
\ee
where $\|\cdot\|_2$ denotes the $L_2(P_\Z)$-norm, where $P_\Z$ denotes the marginal distribution of $\Z$.
If $n_1 = n_2 = n/2$, we have
\be\label{eq:rho-if}
n^{1/2}(\wh\rho - \rho) 
= n^{-1/2} \sumi \phi\eff(X_i,Y_i,\Z_i)
+ O_p(n^{-1/2}),
\ee
where $\phi\eff$ is given by \eqref{eq:eif}.
\end{Th}

Note that \eqref{eq:rho-if} indicates that the estimator $\wh\rho$ is indeed efficient under the convergence rate assumptions in \eqref{eq:rate-cond-1}. 

\begin{Rem}[Effect of regression bias]\label{rem:reg-bias}
In the case where two nonparametric regression estimators have nonvanishing asymptotic bias, the corresponding estimator of residual correlation coefficient will also have some bias. Take $\wh\rho_2$ as an example. Suppose that the regression estimators $\wh m_{x1}$ and $\wh m_{y1}$ converge to the functions $m_x^*$ and $m_y^*$, which may be different from $m_x$ and $m_y$. Under similar or weaker convergence assumptions like \eqref{eq:rate-cond-1} where $m_x$ and $m_y$ are replaced by $m_x^*$ and $m_y^*$, we can similarly show that
\bse
\wh\sigma_{x2}^2\xrightarrow{P}\sigma_x^2+\var\{\delta_x(\Z)\}, \quad
\wh\sigma_{y2}^2\xrightarrow{P}\sigma_y^2+\var\{\delta_y(\Z)\}, \quad
\wh\sigma_{xy2}\xrightarrow{P}\sigma_{xy}+\cov\{\delta_x(\Z),\delta_y(\Z)\},
\ese
where $\delta_x = m_x^* - m_x$ and $\delta_y = m_y^* - m_y$ are corresponding biases. Thus, by Slutsky's theorem, we know that
\bse
\wh\rho_2 = \frac{\wh\sigma_{xy2}}{\wh\sigma_{x2}\wh\sigma_{y2}}
\xrightarrow{P}\frac{\sigma_{xy}+\cov\{\delta_x(\Z),\delta_y(\Z)\}}{[\sigma_x^2+\var\{\delta_x(\Z)\}]^{1/2}[\sigma_y^2+\var\{\delta_y(\Z)\}]^{1/2}}.
\ese
Note that the variance terms in the denominator are always inflated by the regression bias, while the numerator can either increase or decrease according to the sign of the covariance term. Therefore, if the true correlation $\rho$ is nonzero, the regression bias is more likely to shrink the residual correlation estimator toward zero, especially when the covariance between the two regression bias terms is relatively small. On the other hand, if $\rho=0$, the regression bias may as well induce some bias due to the $\cov\{\delta_x(\Z),\delta_y(\Z)\}$ term. The same phenomenon also holds for $\wh\rho_1$ and $\wh\rho$.
\end{Rem}

\subsection{Asymptotic Distribution and Statistical inference}

Note that the variance of $\phi\eff(X,Y,\Z)$ is
\bse
E \left\{\phi\eff^2(X,Y,\Z)\right\} = E(S\eff^2)^{-1} = (1-\rho^2)^2.
\ese
The next theorem gives the the asymptotic distribution of $\wh\rho$.
\begin{Th}\label{thm:asymp-dist}
Under model \eqref{eq:likelihood}, suppose that the nonparametric regression estimators $\wh m_{xk}$ and $\wh m_{yk}$ satisfy \eqref{eq:rate-cond-1}, for $k=1,2$. Let $\wh\rho$ be defined by \eqref{eq:rho-hat-final}. If $n_1 = n_2 = n/2$, then we have
\be\label{eq:asymp-dist}
n^{1/2}(\wh\rho-\rho)\xrightarrow{d}N\left\{0,(1-\rho^2)^2\right\}.
\ee
\end{Th}
The asymptotic variance in \eqref{eq:asymp-dist} coincides with that in the oracle result \eqref{eq:oracle-clt}, indicating that the convergence rate conditions in \eqref{eq:rate-cond-1} are sufficient to make the nonparametric regression estimation errors negligible in terms of first-order asymptotics.

Under the null hypothesis $\rho=0$, the asymptotic distribution of $\wh\rho$ becomes
\bse
n^{1/2}\wh\rho\xrightarrow{d}N(0,1).
\ese
Therefore, we can construct a Wald test for \eqref{eq:hypothesis} \citep{lehmann1999elements}. At significance level $\alpha$, our decision rule is to reject $H_0$ if $n^{1/2}|\wh\rho|>z_{1-\alpha/2}$, where $z_{1-\alpha/2}$ denotes the $(1-\alpha/2)$-quantile of the standard normal distribution. Equivalently, the corresponding p-value is given by
$2\Phi(-n^{1/2}|\wh{\rho}|)$,
where $\Phi$ is the cumulative distribution function of the standard normal distribution.

In general, to construct a confidence interval for $\rho$, since the asymptotic variance $(1-\rho^2)^2$ in \eqref{eq:asymp-dist} involves the unknown parameter $\rho$, we use the plug-in estimator as $(1-\wh\rho^2)^2$. Therefore, by Slutsky’s theorem, an asymptotic $100(1-\alpha)\%$ confidence interval for $\rho$ is given by
\bse
\wh\rho \pm n^{-1/2} z_{1-\alpha/2} (1-\wh\rho^2).
\ese

% \section{Non-Gaussian noise case}
% Suppose we only know that $E(\epsilon_x)=0$, $E(\epsilon_y)=0$, $\var(\epsilon_x)=\sigma_x^2<\infty$, $\var(\epsilon_y)=\sigma_y^2<\infty$. Then, we may use relative entropy or distance covariance to test whether $\epsilon_x$ and $\epsilon_y$ are independent. 

\section{Simulations}\label{sec:simulation}

\subsection{Simulation Settings}\label{sec:setting}

In the simulation studies, we consider two models:
\bse
\text{Model 1:} 
&& \Z = (Z_1, \dots, Z_6) \trans, \quad \text{where} \ Z_1,\dots,Z_6 \overset{\mathrm{iid}}{\sim} \mathrm{Uniform}(0,1), \\
&& X = 0.4 Z_1+0.6\sqrt{Z_3}+\epsilon_x, \quad
 Y = 0.6 Z_4+0.4\sqrt{Z_5}+\epsilon_y, \\
\text{Model 2:} 
&& \Z = (Z_1, \dots, Z_{10}) \trans, \quad \text{where} \ Z_1,\dots,Z_{10} \overset{\mathrm{iid}}{\sim} \mathrm{Uniform}(0,1), \\
&& X = (Z_1+\dots+Z_{10})/10+\epsilon_x, \quad
 Y = \log(1+Z_3+Z_6)+\epsilon_y,
\ese
where in both models, the regression errors $(\epsilon_x, \epsilon_y)\trans \indep \Z$ and
$(\epsilon_x, \epsilon_y)\trans \sim N(\0,\bSigma)$ with
\bse
\bSigma =
\left(
\begin{matrix}
    \sigma^2 & \rho \sigma^2 \\ \rho \sigma^2 & \sigma^2
\end{matrix} \right).
\ese
We set $\sigma^2=0.2$. We take sample sizes to be $n=100,200,500,1000,2000,5000$. Under the null hypothesis, we simply set $\rho=0$. Under the alternative hypothesis, we consider two sets of $\rho$: (1) $\rho=\pm 0.25, \pm 0.5, \pm 0.75, \pm 1$; (2) $\rho=\pm 0.025, \pm 0.05, \pm 0.075, \pm 0.1$. The former set of $\rho$ indicates a strong dependence signal, which represents deviation from the null hypothesis; the later indicates a weak dependence signal which is hard to detect. For each model, we conduct 500 independent experiments, and the significance level is set as $\alpha = 0.05$.

In the simulation studies, we implement our method using residual correlation estimator with sample splitting (RPCS). As is mentioned in Section \ref{sec:splitting}, the nonparametric regression estimators $\wh m_x$ and $\wh m_y$ are fitted on the first part of samples, while the residual correlation estimator $\wh\rho$ is computed on the second part. We further switch the roles of two parts and construct another estimate, and finally average over the two estimates to get the final result.
We also consider a full-data version, referred to as residual correlation estimator with full data (RPCF), in which both the estimation of $\wh m_x, \wh m_y$ and the computation of $\wh\rho$ are performed using the entire dataset without sample splitting.
As an oracle benchmark, we further include a test based on the Pearson correlation of the true regression errors $\epsilon_x$ and $\epsilon_y$, which we refer to as residual correlation estimator in the oracle setting (RPCO). 

On performing the nonparametric regression for $m_x$ and $m_y$, we use Super Learner \citep{van2007super,polley2011super,superlearner}, which combines a collection of candidate regression algorithms including both parametric and nonparametric methods. In our implementation, we include the mean estimator (\texttt{SL.mean}), generalized linear models (\texttt{SL.glm}), penalized linear regression (\texttt{SL.glmnet}), random forests (\texttt{SL.ranger}), gradient boosting trees (\texttt{SL.xgboost}), and neural networks (\texttt{SL.nnet}). The Super Learner is implemented under the Gaussian loss using nonnegative least squares based on the Lawson–Hanson algorithm (\texttt{method.NNLS}) to estimate the ensemble weights for combining the candidate learners. In addition, 5-fold cross-validation is used to evaluate and combine the individual learners.

We compare our method with several existing approaches. The first is the partial correlation test (PaCo), which can be viewed as a special case of our framework in which the regression functions $m_x(\z)$ and $m_y(\z)$ are restricted to be linear in $\z$. Partial correlation plays a fundamental role in multivariate analysis \citep{anderson2003introduction} and graphical models \citep{lauritzen1996graphical}. Moreover, its asymptotic distribution coincides with \eqref{eq:asymp-dist}; see Section 5.3 of \cite{muirhead1982aspects}.
The second method is based on residual Hilbert–Schmidt independence criterion (RHSIC), which measures the dependence between the residuals $\wh\epsilon_x$ and $\wh\epsilon_y$ using a kernel-based independence criterion; see, for example, \cite{gretton2005measuring,gretton2007kernel}. The implementation of RHSIC utilizes the R package \texttt{dHSIC} \citep{dhsic}.
The next three methods are the residual Randomized Independence Test (RRIT), the Randomized Conditional Independence Test (RCIT), and the Randomized conditional Correlation Test (RCoT). Here, RCIT and RCoT were proposed by \cite{strobl2019approximate}, while RRIT denotes the unconditional version of RCIT (or RCoT) applied to the residuals $\wh\epsilon_x$ and $\wh\epsilon_y$. As shown by \cite{strobl2019approximate}, RCIT and RCoT achieve comparable or better empirical performance than the Kernel Conditional Independence Test (KCIT) of \cite{zhang2012kernelbased}, including similar power and Type I error control, while being substantially more computationally efficient. Therefore, in our simulation studies, we include only RCIT and RCoT as representatives of nonparametric conditional independence tests. We implement RRIT, RCIT and RCoT by the R package \texttt{RCIT} \citep{rcit}.

\subsection{Type I Error Control}

Under the null hypothesis ($\rho = 0$), the empirical levels for Model 1 are reported in Table \ref{tab:ep-level-1a}, while the corresponding boxplots of p-values are shown in Figure \ref{fig:boxplot-null-1a}. The analogous results for Model 2 are provided in Table \ref{tab:ep-level-2a} and Figure \ref{fig:boxplot-null-2a} in the supplement. Overall, most methods achieve empirical levels close to the nominal level of 0.05 when the sample size is sufficiently large. For smaller sample sizes, the residual-based tests still maintain levels reasonably close to 0.05. In contrast, the two nonparametric methods, RCIT and RCoT, appear unable to well control the Type I error rate, particularly when $n = 100$ or $200$. A possible explanation is that these methods do not explicitly consider the additive noise structure and may therefore be more sensitive to spurious dependence in small-sample settings.

\begin{table}[htb]
    \centering
    \footnotesize
\begin{tabular}{r|rrrrrrrr}
\hline
 $n$ & RPCO & RPCS & RPCF & PaCo & RHSIC & RRIT & RCIT & RCoT\\
\hline
\hline
 100 & 0.074 & 0.090 & 0.070 & 0.080 & 0.070 & 0.058 & 1.000 & 1.000\\
 200 & 0.062 & 0.080 & 0.066 & 0.064 & 0.062 & 0.054 & 0.224 & 0.180\\
 500 & 0.050 & 0.058 & 0.050 & 0.052 & 0.048 & 0.052 & 0.076 & 0.088\\
 1000 & 0.034 & 0.050 & 0.036 & 0.038 & 0.050 & 0.046 & 0.068 & 0.084\\
 2000 & 0.046 & 0.044 & 0.042 & 0.042 & 0.036 & 0.050 & 0.054 & 0.050\\
 5000 & 0.044 & 0.034 & 0.044 & 0.038 & 0.044 & 0.050 & 0.050 & 0.030\\
\hline
\end{tabular}
    \caption{Empirical levels of eight tests under Model 1 based on 500 experiments.}
    \label{tab:ep-level-1a}
\end{table}

\begin{figure}[htb]
    \centering
    \includegraphics[width=0.75\linewidth,page=1]{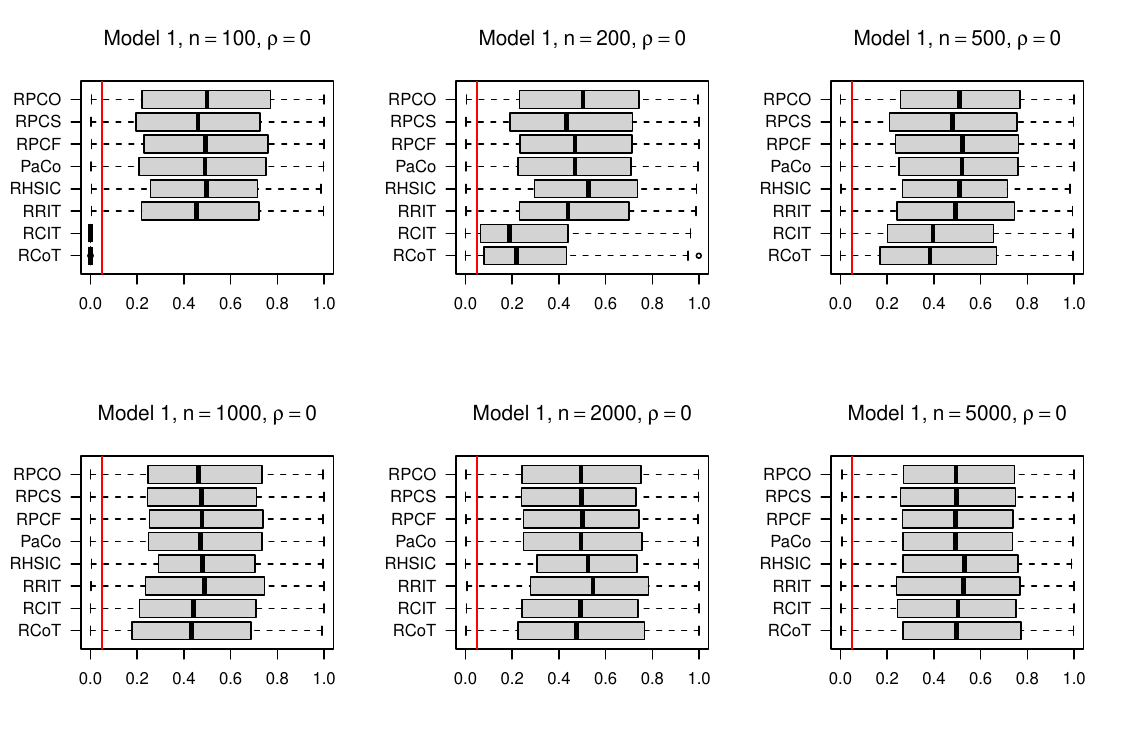}
    \caption{Boxplots of p-values of eight tests for Model 1 under the null hypothesis, with sample sizes $n=100,200,500,1000,2000,5000$. The red line represents 0.05.}
    \label{fig:boxplot-null-1a}
\end{figure}

\subsection{Power}

Under the alternative hypothesis, we report in Table \ref{tab:ep-1a-1} the empirical powers in Model 1 when $\rho$ is positive, and the results when $\rho$ is negative are symmetric and are reported in Table \ref{tab:ep-1a-2} of the supplement. We also present the boxplots of p-values for Model 1 in Figures \ref{fig:boxplot-1a-1}--\ref{fig:boxplot-1a-4} in the supplement. Analogous results for Model 2 are reported in Tables \ref{tab:ep-2a-1}--\ref{tab:ep-2a-2} and Figures \ref{fig:boxplot-2a-1}--\ref{fig:boxplot-2a-4} in the supplement.

\begin{table}[htbp]
    \centering
    \footnotesize
    \vspace{-0.5cm}
\begin{tabular}{rr|rrrrrrrr}
\hline
$\rho$ & $n$ & RPCO & RPCS & RPCF & PaCo & RHSIC & RRIT & RCIT & RCoT\\
\hline
\hline
0.25 & 100 & 0.730 & 0.644 & 0.692 & 0.726 & 0.328 & 0.466 & 1.000 & 1.000\\
0.25 & 200 & 0.954 & 0.912 & 0.940 & 0.946 & 0.610 & 0.804 & 0.746 & 0.784\\
0.25 & 500 & 1.000 & 1.000 & 1.000 & 1.000 & 0.986 & 0.982 & 0.944 & 0.994\\
0.25 & 1000 & 1.000 & 1.000 & 1.000 & 1.000 & 1.000 & 0.996 & 0.990 & 0.996\\
0.25 & 2000 & 1.000 & 1.000 & 1.000 & 1.000 & 1.000 & 1.000 & 1.000 & 1.000\\
0.25 & 5000 & 1.000 & 1.000 & 1.000 & 1.000 & 1.000 & 1.000 & 1.000 & 1.000\\
\hline
0.50 & 100 & 1.000 & 0.992 & 1.000 & 1.000 & 0.954 & 0.974 & 1.000 & 1.000\\
0.50 & 200 & 1.000 & 1.000 & 1.000 & 1.000 & 0.998 & 1.000 & 0.982 & 0.994\\
0.50 & 500 & 1.000 & 1.000 & 1.000 & 1.000 & 1.000 & 1.000 & 1.000 & 0.998\\
0.50 & 1000 & 1.000 & 1.000 & 1.000 & 1.000 & 1.000 & 1.000 & 1.000 & 1.000\\
0.50 & 2000 & 1.000 & 1.000 & 1.000 & 1.000 & 1.000 & 1.000 & 1.000 & 1.000\\
0.50 & 5000 & 1.000 & 1.000 & 1.000 & 1.000 & 1.000 & 1.000 & 1.000 & 1.000\\
\hline
0.75 & 100 & 1.000 & 1.000 & 1.000 & 1.000 & 1.000 & 1.000 & 1.000 & 1.000\\
0.75 & 200 & 1.000 & 1.000 & 1.000 & 1.000 & 1.000 & 1.000 & 0.998 & 1.000\\
0.75 & 500 & 1.000 & 1.000 & 1.000 & 1.000 & 1.000 & 1.000 & 1.000 & 1.000\\
0.75 & 1000 & 1.000 & 1.000 & 1.000 & 1.000 & 1.000 & 1.000 & 1.000 & 1.000\\
0.75 & 2000 & 1.000 & 1.000 & 1.000 & 1.000 & 1.000 & 1.000 & 1.000 & 1.000\\
0.75 & 5000 & 1.000 & 1.000 & 1.000 & 1.000 & 1.000 & 1.000 & 1.000 & 1.000\\
\hline
1.00 & 100 & 1.000 & 1.000 & 1.000 & 1.000 & 1.000 & 1.000 & 1.000 & 1.000\\
1.00 & 200 & 1.000 & 1.000 & 1.000 & 1.000 & 1.000 & 1.000 & 0.998 & 1.000\\
1.00 & 500 & 1.000 & 1.000 & 1.000 & 1.000 & 1.000 & 1.000 & 1.000 & 1.000\\
1.00 & 1000 & 1.000 & 1.000 & 1.000 & 1.000 & 1.000 & 1.000 & 1.000 & 1.000\\
1.00 & 2000 & 1.000 & 1.000 & 1.000 & 1.000 & 1.000 & 1.000 & 1.000 & 1.000\\
1.00 & 5000 & 1.000 & 1.000 & 1.000 & 1.000 & 1.000 & 1.000 & 1.000 & 1.000\\
\hline
\hline
0.025 & 100 & 0.064 & 0.080 & 0.054 & 0.066 & 0.054 & 0.050 & 1.000 & 1.000\\
0.025 & 200 & 0.052 & 0.064 & 0.050 & 0.050 & 0.052 & 0.064 & 0.230 & 0.250\\
0.025 & 500 & 0.078 & 0.092 & 0.076 & 0.078 & 0.064 & 0.074 & 0.092 & 0.080\\
0.025 & 1000 & 0.084 & 0.086 & 0.086 & 0.092 & 0.074 & 0.078 & 0.088 & 0.084\\
0.025 & 2000 & 0.174 & 0.170 & 0.172 & 0.176 & 0.094 & 0.128 & 0.152 & 0.160\\
0.025 & 5000 & 0.432 & 0.416 & 0.418 & 0.420 & 0.196 & 0.286 & 0.226 & 0.260\\
\hline
0.050 & 100 & 0.068 & 0.102 & 0.070 & 0.074 & 0.080 & 0.058 & 1.000 & 1.000\\
0.050 & 200 & 0.120 & 0.144 & 0.120 & 0.118 & 0.064 & 0.082 & 0.206 & 0.222\\
0.050 & 500 & 0.172 & 0.180 & 0.176 & 0.178 & 0.074 & 0.118 & 0.144 & 0.152\\
0.050 & 1000 & 0.364 & 0.356 & 0.358 & 0.358 & 0.144 & 0.250 & 0.256 & 0.258\\
0.050 & 2000 & 0.614 & 0.616 & 0.614 & 0.618 & 0.282 & 0.448 & 0.366 & 0.432\\
0.050 & 5000 & 0.944 & 0.938 & 0.942 & 0.944 & 0.578 & 0.820 & 0.704 & 0.782\\
\hline
0.075 & 100 & 0.120 & 0.132 & 0.128 & 0.130 & 0.074 & 0.090 & 1.000 & 1.000\\
0.075 & 200 & 0.202 & 0.198 & 0.200 & 0.196 & 0.108 & 0.120 & 0.278 & 0.282\\
0.075 & 500 & 0.372 & 0.354 & 0.358 & 0.364 & 0.168 & 0.232 & 0.236 & 0.282\\
0.075 & 1000 & 0.654 & 0.644 & 0.660 & 0.646 & 0.278 & 0.468 & 0.412 & 0.460\\
0.075 & 2000 & 0.930 & 0.924 & 0.926 & 0.926 & 0.564 & 0.772 & 0.680 & 0.724\\
0.075 & 5000 & 1.000 & 1.000 & 1.000 & 1.000 & 0.936 & 0.970 & 0.934 & 0.970\\
\hline
0.100 & 100 & 0.168 & 0.158 & 0.164 & 0.172 & 0.080 & 0.120 & 1.000 & 1.000\\
0.100 & 200 & 0.292 & 0.288 & 0.294 & 0.304 & 0.132 & 0.196 & 0.364 & 0.336\\
0.100 & 500 & 0.640 & 0.634 & 0.654 & 0.654 & 0.272 & 0.452 & 0.412 & 0.464\\
0.100 & 1000 & 0.866 & 0.864 & 0.864 & 0.868 & 0.472 & 0.692 & 0.630 & 0.686\\
0.100 & 2000 & 0.996 & 0.996 & 0.996 & 0.996 & 0.838 & 0.928 & 0.884 & 0.902\\
0.100 & 5000 & 1.000 & 1.000 & 1.000 & 1.000 & 1.000 & 0.996 & 0.978 & 0.998\\
\hline
\end{tabular}
    \caption{Empirical powers of eight tests under Model 1 based on 500 experiments, where the alternative distributions include (1) $\rho=0.25, 0.5, 0.75, 1$ and (2) $\rho= 0.025,  0.05,  0.075,  0.1$.}
    \label{tab:ep-1a-1}
\end{table}

Overall, the proposed methods RPCS and RPCF achieve strong empirical power in most settings, especially when the dependence signal or the sample size is not too small. In most cases, RPCS and RPCF perform similarly to the oracle procedure RPCO, showing that the proposed residual-correlation-based approach performs nearly as well as the ideal procedure using the true regression errors.

For relatively strong dependence signals, i.e., $\rho$ in set (1), all residual-correlation-based methods achieve power close to 1 when the sample size is sufficiently large. In particular, RPCS and RPCF remain very close to the oracle benchmark RPCO in most settings, suggesting that estimating the regression functions causes only a small loss of efficiency. As the sample size increases, the differences among the three procedures quickly become negligible.
In contrast, although RHSIC and RRIT also achieve high power for large sample sizes, their performance is noticeably worse when the sample size is small. 

For relatively weak dependence signals, i.e., $\rho$ in set (2), RPCS and RPCF still show increasing power as either the signal strength or the sample size increases, and they continue to perform similarly to RPCO in most settings. For example, under Model 1 with $\rho=0.05$ and moderate or large sample sizes, the powers of RPCS and RPCF are already very close to those of RPCO. These results suggest that the proposed methods remain effective even when the conditional dependence is weak.
Compared with RHSIC and RRIT, the proposed methods generally achieve higher power under weak dependence signals, particularly when the sample size is small or moderate. This suggests that directly using the residual correlation structure under GANM improves the ability to detect weak dependence.

In addition, in both settings with weak and strong signals, RPCS and RPCF perform very similarly across all settings, indicating that sample splitting causes only a small loss of efficiency in finite samples.

\subsection{Estimation Efficiency}

Table \ref{tab:est-model1-p1} and Tables \ref{tab:est-model1-p2}--\ref{tab:est-model2-p2} in the supplement further compare the estimation efficiency of the proposed residual correlation estimators, including RPCS, RPCF, RPCO and PaCo. We consider the cases where $\rho$ is in set (1) without the setting $\rho=\pm 1$, because in this case, the true asymptotic variance is $(1-\rho^2)^2=0$. In all the tables, $\wh {\rm SD}$ is computed based on the average of the estimated asymptotic standard deviations, which is $n^{-1/2}(1-\wh\rho^2)$, and 95\% cvg is the empirical coverage of the 95\% asymptotic confidence intervals given by $\wh\rho \pm n^{-1/2} z_{1-\alpha/2} (1-\wh\rho^2)$.

\subsection{Strong Nonlinear Effects}

In the previous simulation settings, the nonlinear regression structures are relatively smooth and can still be reasonably approximated by linear relationships. As a result, the partial-correlation-based method PaCo remains competitive in many settings. To further investigate the effect of nonlinear nuisance regression, we additionally consider a more challenging null setting in which the conditional mean functions are strongly nonlinear while the true correlation of the regression errors remains $\rho = 0$.

Specifically, we consider
\bse
\text{Model 3:} 
&& \Z = (Z_1, Z_2) \trans, \quad \text{where} \ Z_1,Z_2 \overset{\mathrm{iid}}{\sim} \mathrm{Uniform}(0,1), \\
&& X = \sin(2\pi Z_1)+4(Z_2-0.5)^2+\epsilon_x, \quad
 Y = \cos(4\pi Z_1)+4(Z_2-0.5)^2+\epsilon_y, 
\ese
where the distribution of $(\epsilon_x,\epsilon_y)\trans$ is same as in Section \ref{sec:setting}, and $\rho=0$.
We report the empirical levels and boxplots of p-values for the eight estimators under this setting as Table \ref{tab:ep-level-3} and Figure \ref{fig:boxplot-null-3}, as well as the estimation results in Table \ref{tab:est-model3} in the supplement.

\begin{table}[htb]
    \centering
    \footnotesize
\begin{tabular}{rrrrrrrrrr}
\hline
$\rho$ & $n$ & RPCO & RPCS & RPCF & PaCo & RHSIC & RRIT & RCIT & RCoT\\
\hline
\hline
0 & 100 & 0.046 & 0.112 & 0.054 & 0.318 & 0.056 & 0.032 & 0.126 & 0.110\\
0 & 200 & 0.040 & 0.092 & 0.072 & 0.452 & 0.048 & 0.052 & 0.108 & 0.080\\
0 & 500 & 0.044 & 0.056 & 0.054 & 0.902 & 0.060 & 0.054 & 0.072 & 0.072\\
0 & 1000 & 0.048 & 0.110 & 0.076 & 1.000 & 0.088 & 0.072 & 0.076 & 0.064\\
0 & 2000 & 0.060 & 0.108 & 0.082 & 1.000 & 0.054 & 0.064 & 0.078 & 0.076\\
0 & 5000 & 0.044 & 0.106 & 0.060 & 1.000 & 0.060 & 0.064 & 0.170 & 0.174\\
\hline
\end{tabular}
    \caption{Empirical levels of eight tests under Model 3 based on 500 experiments.}
    \label{tab:ep-level-3}
\end{table}
\begin{figure}[htb]
    \centering
    \includegraphics[width=0.75\linewidth,page=1]{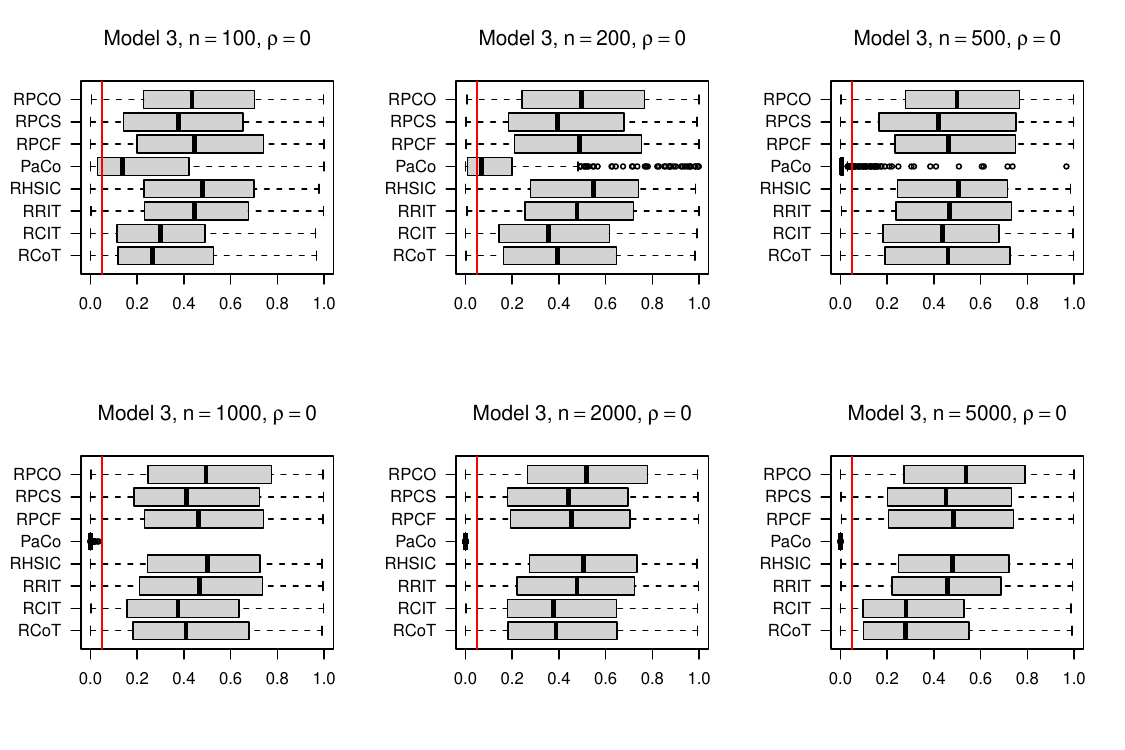}
    \caption{Boxplots of p-values of eight tests for Model 3 under the null hypothesis, with sample sizes $n=100,200,500,1000,2000,5000$. The red line represents 0.05.}
    \label{fig:boxplot-null-3}
\end{figure}

As shown in Table \ref{tab:ep-level-3} and Figure \ref{fig:boxplot-null-3}, all tests except PaCo continue to maintain reasonably accurate empirical levels under the strong nonlinear setting. In contrast, PaCo has substantially inflated Type I error rates, and most of its p-values are concentrated near 0, especially when the sample size is large. Furthermore, Table \ref{tab:est-model3} shows that PaCo also produces large biases in the estimation of $\wh\rho$. These results indicate that flexible nonparametric regression is essential in this setting. In particular, linear regression fails to adequately capture the nonlinear effects of $\Z$ on $X$ and $Y$, thereby leaving substantial residual dependence and causing PaCo to reject the null hypothesis much more frequently than the nominal level. The observed estimation bias is also consistent with Remark \ref{rem:reg-bias}, where regression bias may induce additional residual correlation even under the null hypothesis $\rho = 0$.

\section{Data Application}\label{sec:application}

We collected daily adjusted closing prices for 12 representative U.S. stocks, including Apple (AAPL), Microsoft (MSFT), NVIDIA (NVDA), Alphabet (GOOGL), Amazon (AMZN), JPMorgan Chase (JPM), Bank of America (BAC), Goldman Sachs (GS), ExxonMobil (XOM), Chevron (CVX), Walmart (WMT), and Costco (COST), together with several observed market factors, including the S\&P 500 ETF (SPY), the volatility index (VIX), long-term treasury bonds (TLT), oil prices (USO), and the U.S. dollar index (UUP), from Yahoo Finance using the R package \texttt{quantmod} \citep{quantmod} over the period from January 1, 2018 to January 1, 2024. Adjusted prices were used to account for stock splits and dividends. Daily log returns were computed as differences of the logarithms of consecutive adjusted prices. Observations containing missing values were removed to ensure complete alignment across all variables and factors. The resulting cleaned dataset was then separated into the stock return variables of interest $\X$ and the observed market factor variables $\Z$. Here, $\X$ represents the daily log returns of the 12 individual stocks, while $\Z$ contains 5 market-wide factors intended to capture common variation shared across stock returns and serves as the conditioning variables in the subsequent conditional dependence analysis. After the cleaning process, the sample size is $n=1453$.

We further conduct pairwise conditional independence tests among the 12 stocks based on the proposed residual correlation test under the additive noise model (ANM). Specifically, for each pair of stocks $(X_i, X_j)$, for $1 \le i < j \le 12$, we test whether
\bse
X_i \indep X_j | \Z,
\ese
where $\Z$ consists of the observed market factors represented by SPY, VIX, TLT, USO, and UUP. Under the ANM framework, each stock return is modeled as a nonparametric function of $\Z$ plus an independent additive noise term, and the proposed test is then applied to the residuals to determine whether any remaining conditional dependence exists after conditioning on these observed market factors. This analysis allows us to investigate the dependence structure among stocks beyond the effects explained by the common market factors.

\begin{figure}[htbp]
\centering
\begin{minipage}{0.48\textwidth}
    \centering
    \includegraphics[width=\linewidth]{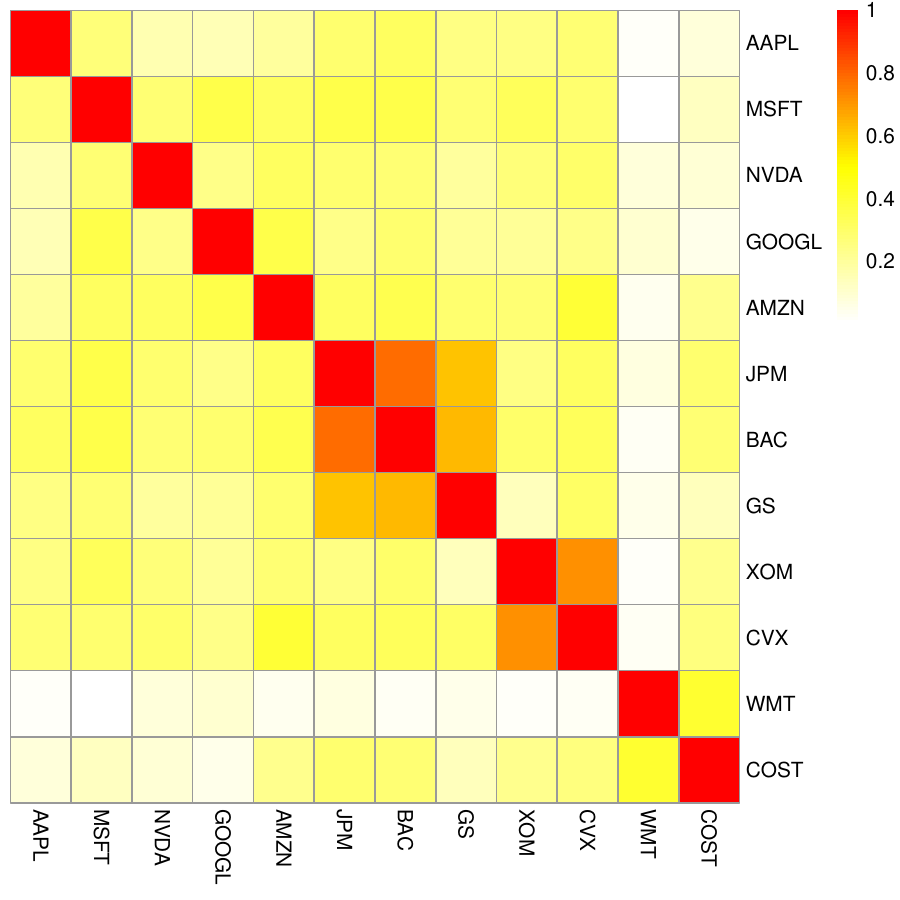}
    \caption{Heatmap of pairwise estimated Pearson correlations of residuals.}
    \label{fig:heatmap}
\end{minipage}
\hfill
\begin{minipage}{0.48\textwidth}
    \centering
    \includegraphics[width=\linewidth]{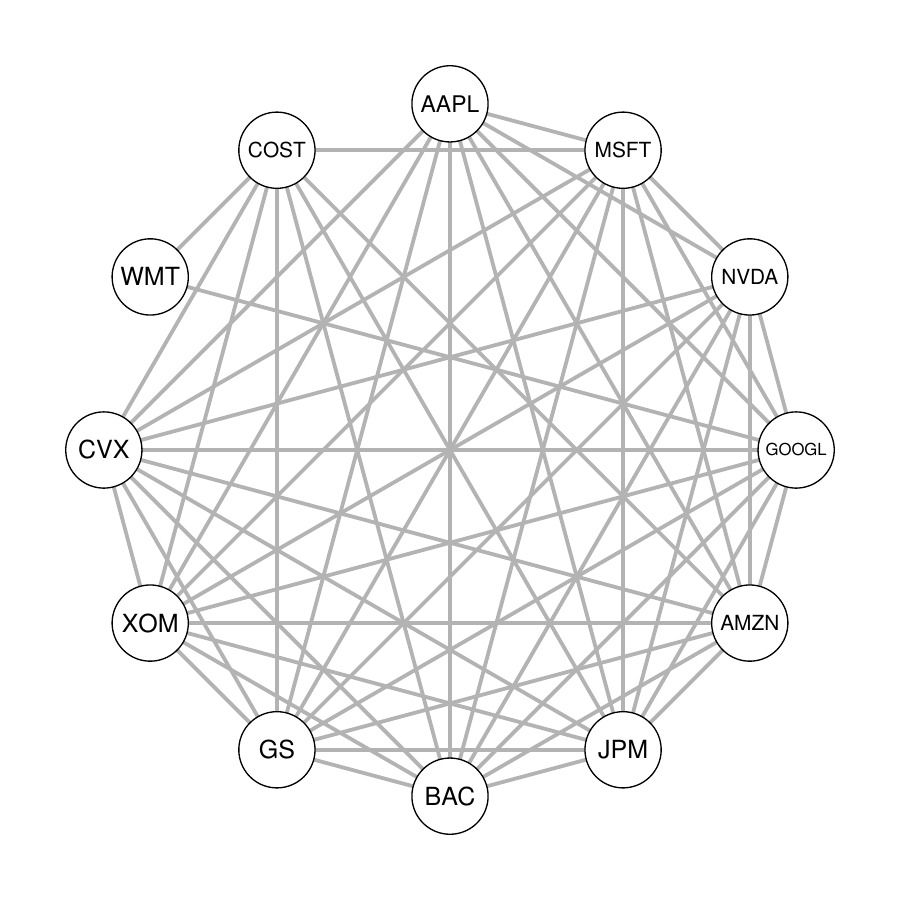}
    \caption{Estimated graph based on the conditional independence test.}
    \label{fig:graph}
\end{minipage}
\end{figure}

Figure \ref{fig:heatmap} presents the heatmap of the absolute values of estimated pairwise residual correlations after conditioning on the observed market factors $\Z$. The residual correlations are estimated by RPCS, while the nonparametric regressions are fitted using Super Learner \citep{van2007super,polley2011super,superlearner}, under the same settings as in Section \ref{sec:setting}. Larger absolute values in the heatmap indicate stronger remaining conditional dependence between the corresponding pairs of stocks after removing the common effects explained by SPY, VIX, TLT, USO, and UUP. Several sector-related dependence patterns can still be observed after conditioning on the market factors. In particular, relatively strong residual dependence appears among the financial stocks JPM, BAC, and GS, as well as between the energy stocks XOM and CVX. Although some cross-sector pairs exhibit weaker dependence after conditioning on the market factors, many stock pairs still retain noticeable residual dependence.

Figure \ref{fig:graph} further visualizes the estimated conditional dependence structure through a graph constructed from the pairwise conditional independence tests based on RPCS. In the graph, each node represents a stock, and an edge is included when the null hypothesis of conditional independence is rejected for the corresponding pair of stocks. The resulting network exhibits many connections across the stocks, indicating that substantial residual dependence remains even after conditioning on the observed market factors. These findings suggest that, although the observed market factors explain part of the common market variation, important residual relationships among individual stocks still persist.

\bibliographystyle{agsm}
\bibliography{ref}

@inproceedings{hoyer2008nonlinear,
 author = {Hoyer, Patrik and Janzing, Dominik and Mooij, Joris M and Peters, Jonas and Sch\"{o}lkopf, Bernhard},
 booktitle = {Advances in Neural Information Processing Systems},
 editor = {D. Koller and D. Schuurmans and Y. Bengio and L. Bottou},
 pages = {},
 publisher = {Curran Associates, Inc.},
 title = {Nonlinear causal discovery with additive noise models},
 volume = {21},
 year = {2008}
}

@article{zhang2017causal, 
title={Causal Discovery Using Regression-Based Conditional Independence Tests},
volume={31}, 
DOI={10.1609/aaai.v31i1.10698}, 
number={1}, 
journal={Proceedings of the AAAI Conference on Artificial Intelligence},
author={Zhang, Hao and Zhou, Shuigeng and Zhang, Kun and Guan, Jihong},
year={2017}, 
month={Feb.} 
}

@article{li2019nonparametric,
author = {Li, Chun and Fan, Xiaodan},
title = {On nonparametric conditional independence tests for continuous variables},
journal = {WIREs Computational Statistics},
volume = {12},
number = {3},
pages = {e1489},
doi = {https://doi.org/10.1002/wics.1489},
eprint = {https://wires.onlinelibrary.wiley.com/doi/pdf/10.1002/wics.1489},
year = {2020}
}

@book{anderson2003introduction,
  title={An Introduction to Multivariate Statistical Analysis},
  author={Anderson, T.W.},
  year={2003},
  publisher={John Wiley \& Suns, Inc. Huboken, New Jersey}
}

@book{lauritzen1996graphical,
  title={Graphical Models},
  author={Lauritzen, S.L.},
  series={Oxford Statistical Science Series},
  year={1996},
  publisher={Clarendon Press}
}

@book{muirhead1982aspects,
  title={Aspects of Multivariate Statistical Theory},
  author={Muirhead, R.J.},
  series={Wiley Series in Probability and Statistics},
  year={1982},
  publisher={Wiley}
}

@InProceedings{gretton2005measuring,
author="Gretton, Arthur
and Bousquet, Olivier
and Smola, Alex
and Sch{\"o}lkopf, Bernhard",
editor="Jain, Sanjay
and Simon, Hans Ulrich
and Tomita, Etsuji",
title="Measuring Statistical Dependence with Hilbert-Schmidt Norms",
booktitle="Algorithmic Learning Theory",
year="2005",
publisher="Springer Berlin Heidelberg",
address="Berlin, Heidelberg",
pages="63--77",
isbn="978-3-540-31696-1"
}

@inproceedings{gretton2007kernel,
 author = {Gretton, Arthur and Fukumizu, Kenji and Teo, Choon and Song, Le and Sch\"{o}lkopf, Bernhard and Smola, Alex},
 booktitle = {Advances in Neural Information Processing Systems},
 editor = {J. Platt and D. Koller and Y. Singer and S. Roweis},
 pages = {},
 publisher = {Curran Associates, Inc.},
 title = {A Kernel Statistical Test of Independence},
 volume = {20},
 year = {2007}
}

@article{strobl2019approximate,
title = {Approximate Kernel-Based Conditional Independence Tests for Fast Non-Parametric Causal Discovery},
author = {Eric V. Strobl and Kun Zhang and Shyam Visweswaran},
pages = {20180017},
volume = {7},
number = {1},
journal = {Journal of Causal Inference},
doi = {doi:10.1515/jci-2018-0017},
year = {2019},
lastchecked = {2026-05-24}
}

@inproceedings{zhang2012kernelbased,
author = {Zhang, Kun and Peters, Jonas and Janzing, Dominik and Sch\"{o}lkopf, Bernhard},
title = {Kernel-based conditional independence test and application in causal discovery},
year = {2011},
isbn = {9780974903972},
publisher = {AUAI Press},
address = {Arlington, Virginia, USA},
booktitle = {Proceedings of the Twenty-Seventh Conference on Uncertainty in Artificial Intelligence},
pages = {804–813},
numpages = {10},
location = {Barcelona, Spain},
series = {UAI'11}
}

@Manual{dhsic,
    title = {dHSIC: Independence Testing via Hilbert Schmidt Independence Criterion},
    author = {Niklas Pfister and Jonas Peters},
    year = {2026},
    note = {R package version 2.2},
  }

@Manual{rcit,
    title = {RCIT: The Randomized Conditional Independence Test (RCIT) and the
Randomized conditional Correlation Test (RCoT)},
    author = {Eric V. Strobl},
    year = {2026},
    note = {R package version 0.1.0, commit 7a7fb2be17ca6d8c657c750d435032824d1111a9},
    url = {https://github.com/ericstrobl/RCIT},
  }

@Manual{quantmod,
    title = {quantmod: Quantitative Financial Modelling Framework},
    author = {Jeffrey A. Ryan and Joshua M. Ulrich},
    year = {2025},
    note = {R package version 0.4.28, 
https://github.com/joshuaulrich/quantmod},
    url = {https://www.quantmod.com/},
  }

@article{peters2014causal,
  author  = {Jonas Peters and Joris M. Mooij and Dominik Janzing and Bernhard Sch{{\"o}}lkopf},
  title   = {Causal Discovery with Continuous Additive Noise Models},
  journal = {Journal of Machine Learning Research},
  year    = {2014},
  volume  = {15},
  number  = {58},
  pages   = {2009--2053}
}

@book{li2018sufficient,
  title={Sufficient Dimension Reduction: Methods and Applications with R},
  author={Li, B.},
  isbn={9781498704489},
  series={Chapman \& Hall/CRC Monographs on Statistics and Applied Probability},
  year={2018},
  publisher={CRC Press}
}

@article{ma2013review,
author = {Ma, Yanyuan and Zhu, Liping},
title = {A Review on Dimension Reduction},
journal = {International Statistical Review},
volume = {81},
number = {1},
pages = {134-150},
doi = {https://doi.org/10.1111/j.1751-5823.2012.00182.x},
year = {2013}
}

@book{koller2009probabilistic,
  title={Probabilistic Graphical Models: Principles and Techniques},
  author={Koller, D. and Friedman, N.},
  isbn={9780262258357},
  series={Adaptive Computation and Machine Learning series},
  year={2009},
  publisher={MIT Press}
}

@book{ding2024course,
      title={A First Course in Causal Inference}, 
      author={Peng Ding},
      year={2024},
      publisher={Chapman and Hall/CRC},
      doi={https://doi.org/10.1201/9781003484080}
}

@book{pearl2016causal,
  title={Causal Inference in Statistics: A Primer},
  author={Pearl, J. and Glymour, M. and Jewell, N.P.},
  isbn={9781119186847},
  lccn={2015037219},
  year={2016},
  publisher={Wiley}
}

@article{dawid1979conditional,
 ISSN = {00359246},
 author = {A. P. Dawid},
 journal = {Journal of the Royal Statistical Society. Series B (Methodological)},
 number = {1},
 pages = {1--31},
 publisher = {[Royal Statistical Society, Oxford University Press]},
 title = {Conditional Independence in Statistical Theory},
 urldate = {2026-05-26},
 volume = {41},
 year = {1979}
}

@Manual{superlearner,
    title = {SuperLearner: Super Learner Prediction},
    author = {Eric Polley and Erin LeDell and Chris Kennedy and Mark {van der Laan}},
    year = {2025},
    note = {R package version 2.0-40},
    url = {https://github.com/ecpolley/SuperLearner},
  }

@article{van2007super,
  title={Super learner},
  author={Van der Laan, Mark J and Polley, Eric C and Hubbard, Alan E},
  year={2007},
  journal = {Statistical Applications in Genetics and Molecular Biology},
  volume  = {6},
  pages   = {Article 25},
  year    = {2007},
  doi     = {10.2202/1544-6115.1309}
}

@Inbook{polley2011super,
author="Polley, Eric C.
and Rose, Sherri
and van der Laan, Mark J.",
title="Super Learning",
bookTitle="Targeted Learning: Causal Inference for Observational and Experimental Data",
year="2011",
publisher="Springer New York",
address="New York, NY",
pages="43--66",
isbn="978-1-4419-9782-1",
doi="10.1007/978-1-4419-9782-1_3"
}

@article{dempster1972covariance,
 ISSN = {0006341X, 15410420},
 author = {A. P. Dempster},
 journal = {Biometrics},
 number = {1},
 pages = {157--175},
 publisher = {International Biometric Society},
 title = {Covariance Selection},
 urldate = {2026-05-26},
 volume = {28},
 year = {1972}
}

@article{baba2004partial,
author = {Baba, Kunihiro and Shibata, Ritei and Sibuya, Masaaki},
title = {PARTIAL CORRELATION AND CONDITIONAL CORRELATION AS MEASURES OF CONDITIONAL INDEPENDENCE},
journal = {Australian \& New Zealand Journal of Statistics},
volume = {46},
number = {4},
pages = {657-664},
doi = {https://doi.org/10.1111/j.1467-842X.2004.00360.x},
year = {2004}
}

@ARTICLE{peters2011causal,
  author={Peters, Jonas and Janzing, Dominik and Scholkopf, Bernhard},
  journal={IEEE Transactions on Pattern Analysis and Machine Intelligence}, 
  title={Causal Inference on Discrete Data Using Additive Noise Models}, 
  year={2011},
  volume={33},
  number={12},
  pages={2436-2450},
  doi={10.1109/TPAMI.2011.71}}

@article{shimizu2006linear,
  author  = {Shohei Shimizu and Patrik O. Hoyer and Aapo Hyv\"arinen and Antti Kerminen},
  title   = {A Linear Non-Gaussian Acyclic Model for Causal Discovery},
  journal = {Journal of Machine Learning Research},
  year    = {2006},
  volume  = {7},
  number  = {72},
  pages   = {2003--2030}
}

@article{zhang2019measuring,
author = {Zhang, Hao and Zhou, Shuigeng and Guan, Jihong and Huan, Jun (Luke)},
title = {Measuring Conditional Independence by Independent Residuals for Causal Discovery},
year = {2019},
issue_date = {September 2019},
publisher = {Association for Computing Machinery},
address = {New York, NY, USA},
volume = {10},
number = {5},
issn = {2157-6904},
doi = {10.1145/3325708},
journal = {ACM Trans. Intell. Syst. Technol.},
month = sep,
articleno = {50},
numpages = {19}
}

@article{huang2010testing,
author = {Tzee-Ming Huang},
title = {{Testing conditional independence using maximal nonlinear conditional correlation}},
volume = {38},
journal = {The Annals of Statistics},
number = {4},
publisher = {Institute of Mathematical Statistics},
pages = {2047 -- 2091},
year = {2010}
}

@article{su2008nonparametric, 
title={A Nonparametric Hellinger Metric Test for Conditional Independence}, 
volume={24}, 
number={4}, journal={Econometric Theory}, publisher={Cambridge University Press}, author={Su, Liangjun and White, Halbert}, year={2008}, pages={829–864}}

@article{huang2016flexible,
title={A Flexible Nonparametric Test for Conditional Independence}, volume={32}, 
number={6}, journal={Econometric Theory}, publisher={Cambridge University Press}, author={Huang, Meng and Sun, Yixiao and White, Halbert}, year={2016}, pages={1434–1482}}

@inproceedings{doran2014permutation,
  title = {A Permutation-Based Kernel Conditional Independence Test},
  author = {Doran, G. and Muandet, K. and Zhang, K. and Sch{\"o}lkopf, B.},
  booktitle = {Proceedings of the 30th Conference on Uncertainty in Artificial Intelligence (UAI2014)},
  pages = {132--141},
  editors = {Nevin L. Zhang and Jin Tian},
  publisher = {AUAI Press Corvallis},
  address = {Oregon},
  year = {2014}
}

@article{wang2015conditional,
author = {Xueqin Wang and Wenliang Pan and Wenhao Hu and Yuan Tian and Heping Zhang},
title = {Conditional Distance Correlation},
journal = {Journal of the American Statistical Association},
volume = {110},
number = {512},
pages = {1726--1734},
year = {2015},
publisher = {Taylor \& Francis},
doi = {10.1080/01621459.2014.993081},

    note ={PMID: 26877569}
}

@article{cai2022distribution,
  author  = {Zhanrui Cai and Runze Li and Yaowu Zhang},
  title   = {A Distribution Free Conditional Independence Test with Applications to Causal Discovery},
  journal = {Journal of Machine Learning Research},
  year    = {2022},
  volume  = {23},
  number  = {85},
  pages   = {1--41}
}

@article{fukumizu2004dimensionality,
author = {Fukumizu, Kenji and Bach, Francis R. and Jordan, Michael I.},
title = {Dimensionality Reduction for Supervised Learning with Reproducing Kernel Hilbert Spaces},
year = {2004},
issue_date = {12/1/2004},
publisher = {JMLR.org},
volume = {5},
issn = {1532-4435},
journal = {J. Mach. Learn. Res.},
month = {dec},
pages = {73–99},
numpages = {27}
}

@inproceedings{fukumizu2007kernel,
 author = {Fukumizu, Kenji and Gretton, Arthur and Sun, Xiaohai and Sch\"{o}lkopf, Bernhard},
 booktitle = {Advances in Neural Information Processing Systems},
 editor = {J. Platt and D. Koller and Y. Singer and S. Roweis},
 pages = {},
 publisher = {Curran Associates, Inc.},
 title = {Kernel Measures of Conditional Dependence},
 volume = {20},
 year = {2007}
}

@article{sheng2023distance,
  author  = {Tianhong Sheng and Bharath K. Sriperumbudur},
  title   = {On Distance and Kernel Measures of Conditional Dependence},
  journal = {Journal of Machine Learning Research},
  year    = {2023},
  volume  = {24},
  number  = {7},
  pages   = {1--16}
}

@misc{tang2026kernel,
      title={A Kernel-Based Nonparametric Test for Conditional Independence of Functional Data}, 
      author={Yin Tang and Bing Li},
      year={2026},
      eprint={2603.13704},
      archivePrefix={arXiv},
      primaryClass={stat.ME}
}

@book{lehmann1999elements,
author="Lehmann, E. L.",
title="Elements of Large-Sample Theory",
year="1999",
publisher="Springer New York",
address="New York, NY"
}

@book{bcrw,
  title={Efficient and Adaptive Estimation for Semiparametric Models},
  author={Bickel, Peter J and Klaassen, J and Ritov, YA'Acov and Wellner, Jon A},
  year={1993},
  publisher={Johns Hopkins University Press Baltimore}
}

@book{tsiatis,
  title={Semiparametric Theory and Missing Data},
  author={Tsiatis, Anastasios A},
  year={2006},
  publisher={New York: Springer}
}

@article{chang2026testing,
author = {Jinyuan Chang and Yue Du and Jing He and Qiwei Yao},
title = {Testing Independence and Conditional Independence in High Dimensions via Coordinatewise Gaussianization},
journal = {Journal of the American Statistical Association},
volume = {0},
number = {0},
pages = {1--15},
year = {2026},
publisher = {Taylor \& Francis},
doi = {10.1080/01621459.2026.2637891},


URL = { 
    
        https://doi.org/10.1080/01621459.2026.2637891
    
    

},
eprint = { 
    
        https://doi.org/10.1080/01621459.2026.2637891
    
    

}

}

\clearpage

\pagenumbering{arabic}
\setcounter{page}{1} %%% Set page number to 1
\setcounter{equation}{0}\renewcommand{\theequation}{S.\arabic{equation}}
\setcounter{section}{0}\renewcommand{\thesection}{S.\arabic{section}}
\setcounter{subsection}{0}\renewcommand{\thesubsection}{S.\arabic{section}.\arabic{subsection}}
\setcounter{subsubsection}{0}\renewcommand{\thesubsubsection}{S.\arabic{section}.\arabic{subsection}.\arabic{subsubsection}}
\setcounter{Th}{0}\renewcommand{\theTh}{S.\arabic{Th}}
\setcounter{Lem}{0}\renewcommand{\theLem}{S.\arabic{Lem}}
\setcounter{Rem}{0}\renewcommand{\theRem}{S.\arabic{Rem}}
\setcounter{Cor}{0}\renewcommand{\theCor}{S.\arabic{Cor}}
\setcounter{table}{0}\renewcommand{\thetable}{S.\arabic{table}}
\setcounter{figure}{0}\renewcommand{\thefigure}{S.\arabic{figure}}

\section*{Supplementary Materials}

\section{Proofs}

\begin{proof}[Proof of Proposition \ref{prop:space}]
Note that 
\bse
S_{\sigma_x}(x,y,\z)=-\frac{1}{\sigma_x}+\frac{1}{1-\rho^2}\frac{\epsilon_x^2}{\sigma_x^3}-\frac{\rho}{1-\rho^2}\frac{\epsilon_x\epsilon_y}{\sigma_x^2\sigma_y}.
\ese
Thus, the nuisance tangent space for $\sigma_x$ is
\bse
\Lambda_1=\{c_1S_{\sigma_x}(x,y,\z)\}=\left\{c_1\left(-\frac{1}{\sigma_x}+\frac{1}{1-\rho^2}\frac{\epsilon_x^2}{\sigma_x^3}-\frac{\rho}{1-\rho^2}\frac{\epsilon_x\epsilon_y}{\sigma_x^2\sigma_y}\right)\right\}.
\ese
Similarly, the nuisance tangent space for $\sigma_y$ is
\bse
S_{\sigma_y}(x,y,\z)=-\frac{1}{\sigma_y}+\frac{1}{1-\rho^2}\frac{\epsilon_y^2}{\sigma_y^3}-\frac{\rho}{1-\rho^2}\frac{\epsilon_x\epsilon_y}{\sigma_x\sigma_y^2}.
\ese
And
\bse
\Lambda_2=\{c_2S_{\sigma_y}(x,y,\z)\}=\left\{c_2\left(-\frac{1}{\sigma_y}+\frac{1}{1-\rho^2}\frac{\epsilon_y^2}{\sigma_y^3}-\frac{\rho}{1-\rho^2}\frac{\epsilon_x\epsilon_y}{\sigma_x\sigma_y^2}\right)\right\}.
\ese

Then, let $\Lambda_3$ be the nuisance tangent space for $m_x$. Take $m_x(\z)=m_0(\z)+\gamma_xB_1(\z)$ as a parametric submodel of $m_x(\z)$. Thus,
\bse
S_{\gamma_x}|_{\gamma_x=0}
&=&-\frac{1}{2(1-\rho^2)}\left[\frac{2\epsilon_x\{-B_1(\z)\}}{\sigma_x^2}-2\rho\frac{\epsilon_y\{-B_1(\z)\}}{\sigma_x\sigma_y}\right]\\
&=&\frac{1}{1-\rho^2}\left(\frac{\epsilon_x}{\sigma_x^2}-\rho\frac{\epsilon_y}{\sigma_x\sigma_y}\right)B_1(\z).
\ese
Denote 
\bse
A=\left\{\frac{1}{1-\rho^2}\left(\frac{\epsilon_x}{\sigma_x^2}-\rho\frac{\epsilon_y}{\sigma_x\sigma_y}\right)B_1(\z)\right\}.
\ese
Note that $m_x(\z)=m_0(\z)+\gamma_xB_1(\z)$ is a parametric submodel of $m_x(\z)$, so $A\subset \Lambda_3$. On the other hand, for any parametric submodel $m_x(\z,\gamma)$ with $\gamma=\gamma_0$ leading to the truth, we have
\bse
\frac{\partial \log f}{\partial \gamma}\Big|_{\gamma=\gamma_0} &=& -\frac{1}{2(1-\rho^2)}\left[\frac{2\epsilon_x}{\sigma_x^2}\left\{-\frac{\partial m_x(\z,\gamma)}{\partial\gamma}\Big|_{\gamma=\gamma_0}\right\}-2\rho \frac{\epsilon_y}{\sigma_x\sigma_y}\left\{-\frac{\partial m_x(\z,\gamma)}{\partial\gamma}\Big|_{\gamma=\gamma_0}\right\}\right]\\
&=& \frac{1}{1-\rho^2}\left(\frac{\epsilon_x}{\sigma_x^2}-\rho \frac{\epsilon_y}{\sigma_x\sigma_y}\right)\left\{\frac{\partial m_x(\z,\gamma)}{\partial\gamma}\Big|_{\gamma=\gamma_0}\right\} \in A,
\ese
so $\Lambda_3\subset A$. Since both $\Lambda_3$ and $A$ are closed, we have $\Lambda_3=A$, i.e.,
\bse
\Lambda_3=\left\{\frac{1}{1-\rho^2}\left(\frac{\epsilon_x}{\sigma_x^2}-\rho\frac{\epsilon_y}{\sigma_x\sigma_y}\right)B_1(\z)\right\}.
\ese
Similarly, the nuisance tangent space for $m_y$ is
\bse
\Lambda_4=\left\{\frac{1}{1-\rho^2}\left(\frac{\epsilon_y}{\sigma_y^2}-\rho\frac{\epsilon_x}{\sigma_x\sigma_y}\right)B_2(\z)\right\}.
\ese
Obviously, the nuisance tangent space for $f_\Z$ is
\bse
\Lambda_5=\{a(\z):E(a)=0\}.
\ese
Summarizing the above results gives 
\bse
\Lambda = \Lambda_1+ \Lambda_2+ \Lambda_3+ \Lambda_4+ \Lambda_5.
\ese
\end{proof}

\begin{proof}[Proof of Proposition \ref{prop:ortho-comp}]
We then check the orthogonality of the five spaces above. For any $B_1(\z)$, we have
\bse
&&E\left\{\left(-\frac{1}{\sigma_x}+\frac{1}{1-\rho^2}\frac{\epsilon_x^2}{\sigma_x^3}-\frac{\rho}{1-\rho^2}\frac{\epsilon_x\epsilon_y}{\sigma_x^2\sigma_y}\right)\frac{1}{1-\rho^2}\left(\frac{\epsilon_x}{\sigma_x^2}-\frac{\epsilon_y}{\sigma_x\sigma_y}\right)B_1(\Z)\right\}\\
&=&E\left\{\left(-\frac{1}{\sigma_x}+\frac{1}{1-\rho^2}\frac{\epsilon_x^2}{\sigma_x^3}-\frac{\rho}{1-\rho^2}\frac{\epsilon_x\epsilon_y}{\sigma_x^2\sigma_y}\right)\frac{1}{1-\rho^2}\left(\frac{\epsilon_x}{\sigma_x^2}-\frac{\epsilon_y}{\sigma_x\sigma_y}\right)\right\}E\left\{B_1(\Z)\right\}\\
&=& 0,
\ese
since this only includes the first and third moments of a multivariate normal distribution. Thus, $\Lambda_1\perp\Lambda_3$. Similarly, $\Lambda_1\perp\Lambda_4$, $\Lambda_2\perp\Lambda_3$, $\Lambda_2\perp\Lambda_4$.

For $a(\z)\in\Lambda_5$, we have
\bse
&&E\left\{\left(-\frac{1}{\sigma_x}+\frac{1}{1-\rho^2}\frac{\epsilon_x^2}{\sigma_x^3}-\frac{\rho}{1-\rho^2}\frac{\epsilon_x\epsilon_y}{\sigma_x^2\sigma_y}\right)a(\Z)\right\}\\
&=&E\left(-\frac{1}{\sigma_x}+\frac{1}{1-\rho^2}\frac{\epsilon_x^2}{\sigma_x^3}-\frac{\rho}{1-\rho^2}\frac{\epsilon_x\epsilon_y}{\sigma_x^2\sigma_y}\right)E\left\{a(\Z)\right\}\\
&=& 0.
\ese
Thus, $\Lambda_1\perp\Lambda_5$. Similarly, $\Lambda_2\perp\Lambda_5$.

For $a(\z)\in\Lambda_5$ and any $B_1(\z)$, 
\bse
E\left[\frac{1}{1-\rho^2}\left\{\frac{\epsilon_x}{\sigma_x^2}-\rho\frac{\epsilon_y}{\sigma_x\sigma_y}\right\}B_1(\Z)a(\Z)\right]
=\frac{1}{1-\rho^2}E\left\{\frac{\epsilon_x}{\sigma_x^2}-\rho\frac{\epsilon_y}{\sigma_x\sigma_y}\right\}E\left[B_1(\Z)a(\Z)\right]
= 0.
\ese
Thus, $\Lambda_3\perp\Lambda_5$. Similarly, $\Lambda_4\perp\Lambda_5$.

However, 
\bse
&&E(S_{\sigma_y}S_{\sigma_x})\\
&=&E\left\{\left(-\frac{1}{\sigma_x}+\frac{1}{1-\rho^2}\frac{\epsilon_x^2}{\sigma_x^3}-\frac{\rho}{1-\rho^2}\frac{\epsilon_x\epsilon_y}{\sigma_x^2\sigma_y}\right)\left(-\frac{1}{\sigma_y}+\frac{1}{1-\rho^2}\frac{\epsilon_y^2}{\sigma_y^3}-\frac{\rho}{1-\rho^2}\frac{\epsilon_x\epsilon_y}{\sigma_x\sigma_y^2}\right)\right\}\\
&=& \frac{1}{\sigma_x\sigma_y}\left\{1-\frac{1}{1-\rho^2}+\frac{\rho^2}{1-\rho^2}-\frac{1}{1-\rho^2}+\frac{1+2\rho^2}{(1-\rho^2)^2}-\frac{3\rho^2}{(1-\rho^2)^2}+\frac{\rho^2}{1-\rho^2} \right. \\
&& \left. \qquad \quad -\frac{3\rho^2}{(1-\rho^2)^2}+\frac{\rho^2(1+2\rho^2)}{(1-\rho^2)^2}\right\}\\
&=&-\frac{1}{\sigma_x\sigma_y} \frac{\rho^2}{1-\rho^2} \\
&\neq& 0,
\ese
so $\Lambda_1\not\perp\Lambda_2$. Also,
\bse
&&E\left\{\frac{1}{1-\rho^2}\left(\frac{\epsilon_x}{\sigma_x^2}-\rho\frac{\epsilon_y}{\sigma_x\sigma_y}\right)B_1(\Z)\frac{1}{1-\rho^2}\left(\frac{\epsilon_y}{\sigma_y^2}-\rho\frac{\epsilon_x}{\sigma_x\sigma_y}\right)B_2(\Z)\right\}\\
&=&\frac{1}{(1-\rho^2)^2}E\left\{\left(\frac{\epsilon_x}{\sigma_x^2}-\rho\frac{\epsilon_y}{\sigma_x\sigma_y}\right)\left(\frac{\epsilon_y}{\sigma_y^2}-\rho\frac{\epsilon_x}{\sigma_x\sigma_y}\right)\right\}E\left\{B_1(\Z)B_2(\Z)\right\}\\
&=&\frac{1}{(1-\rho^2)^2}\frac{1}{\sigma_x\sigma_y}\left(\rho-\rho-\rho+\rho^3\right)E\left\{B_1(\Z)B_2(\Z)\right\}\\
&=&-\frac{\rho}{1-\rho^2}\frac{1}{\sigma_x\sigma_y}E\left\{B_1(\Z)B_2(\Z)\right\}\\
&\neq& 0,
\ese
so $\Lambda_3\not\perp\Lambda_4$.

We now do orthogonalization on the above two pairs. First, we find $c$ such that 
\bse
E\left\{(S_{\sigma_y}-cS_{\sigma_x})S_{\sigma_x}\right\}=0.
\ese
Notice that 
\bse
E\left(S_{\sigma_x}^2\right)
&=&E\left\{\left(-\frac{1}{\sigma_x}+\frac{1}{1-\rho^2}\frac{\epsilon_x^2}{\sigma_x^3}-\frac{\rho}{1-\rho^2}\frac{\epsilon_x\epsilon_y}{\sigma_x^2\sigma_y}\right)^2\right\}\\
&=&\frac{1}{\sigma_x^2}\left\{1+\frac{3}{(1-\rho^2)^2}+\frac{\rho^2(1+2\rho^2)}{(1-\rho^2)^2}-\frac{2}{1-\rho^2}+\frac{2\rho^2}{1-\rho^2}-\frac{6\rho^2}{(1-\rho^2)^2}\right\}\\
&=&\frac{1}{\sigma_x^2}\frac{2-\rho^2}{1-\rho^2}.
\ese
Therefore,
\bse
c=\frac{E(S_{\sigma_y}S_{\sigma_x})}{E(S_{\sigma_x}^2)}=\frac{-\frac{1}{\sigma_x\sigma_y} \frac{\rho^2}{1-\rho^2}}{\frac{1}{\sigma_x^2}\frac{2-\rho^2}{1-\rho^2}}=-\frac{\sigma_x}{\sigma_y}\frac{\rho^2}{2-\rho^2}.
\ese
We set 
\bse
\wt\Lambda_2
&=&\left\{c_2\left(-\frac{1}{\sigma_y}+\frac{1}{1-\rho^2}\frac{\epsilon_y^2}{\sigma_y^3}-\frac{\rho}{1-\rho^2}\frac{\epsilon_x\epsilon_y}{\sigma_x\sigma_y^2}\right) \right. \\
&& \left. \quad +c_2\frac{\sigma_x}{\sigma_y}\frac{\rho^2}{2-\rho^2}\left(-\frac{1}{\sigma_x}+\frac{1}{1-\rho^2}\frac{\epsilon_x^2}{\sigma_x^3}-\frac{\rho}{1-\rho^2}\frac{\epsilon_x\epsilon_y}{\sigma_x^2\sigma_y}\right)\right\}\\
&=&\left\{c_2\left\{-\frac{2}{2-\rho^2}\frac{1}{\sigma_y}+\frac{1}{1-\rho^2}\frac{\epsilon_y^2}{\sigma_y^3}+\frac{\rho^2}{(1-\rho^2)(2-\rho^2)}\frac{\epsilon_x^2}{\sigma_x^2\sigma_y}-\frac{2\rho}{(1-\rho^2)(2-\rho^2)}\frac{\epsilon_x\epsilon_y}{\sigma_x\sigma_y^2}\right\}\right\}.
\ese

Then, for an arbitrary $B_2(\z)$, we want to find $B(\z)$ such that 
\bse
E\left[\left\{\left(\frac{\epsilon_y}{\sigma_y^2}-\rho\frac{\epsilon_x}{\sigma_x\sigma_y}\right)B_2(\Z)-\left(\frac{\epsilon_x}{\sigma_x^2}-\rho\frac{\epsilon_y}{\sigma_x\sigma_y}\right)B(\Z)\right\}\left\{\left(\frac{\epsilon_x}{\sigma_x^2}-\rho\frac{\epsilon_y}{\sigma_x\sigma_y}\right)B_1(\Z)\right\}\right]=0
\ese
for all $B_1(\z)$. Since
\bse
&&E\left[\left\{\left(\frac{\epsilon_y}{\sigma_y^2}-\rho\frac{\epsilon_x}{\sigma_x\sigma_y}\right)B_2(\Z)-\left(\frac{\epsilon_x}{\sigma_x^2}-\rho\frac{\epsilon_y}{\sigma_x\sigma_y}\right)B(\Z)\right\}\left\{\left(\frac{\epsilon_x}{\sigma_x^2}-\rho\frac{\epsilon_y}{\sigma_x\sigma_y}\right)B_1(\Z)\right\}\right]\\
&=&E\left[\left\{\left(\frac{\epsilon_y}{\sigma_y^2}-\rho\frac{\epsilon_x}{\sigma_x\sigma_y}\right)\left(\frac{\epsilon_x}{\sigma_x^2}-\rho\frac{\epsilon_y}{\sigma_x\sigma_y}\right)\right\}B_2(\Z)B_1(\Z)-\left(\frac{\epsilon_x}{\sigma_x^2}-\rho\frac{\epsilon_y}{\sigma_x\sigma_y}\right)^2B(\Z)B_1(\Z)\right]\\
&=& E\left\{\frac{1}{\sigma_x\sigma_y}(\rho-\rho-\rho+\rho^3)B_2(\Z)B_1(\Z)-\frac{1}{\sigma_x^2}(1+\rho^2-2\rho^2)B(\Z)B_1(\Z)\right\}\\
&=& E\left[\left\{\frac{1}{\sigma_x\sigma_y}(\rho^3-\rho)B_2(\Z)-\frac{1}{\sigma_x^2}(1-\rho^2)B(\Z)\right\}B_1(\Z)\right],
\ese
then we take
\bse
B(\z)=-\frac{\sigma_x}{\sigma_y}\rho B_2(\z).
\ese
We set 
\bse
\wt\Lambda_4 &=& \left\{\frac{1}{1-\rho^2}\left(\frac{\epsilon_y}{\sigma_y^2}-\rho\frac{\epsilon_x}{\sigma_x\sigma_y}\right)B_2(\z)+\frac{1}{1-\rho^2}\left(\frac{\epsilon_x}{\sigma_x^2}-\rho\frac{\epsilon_y}{\sigma_x\sigma_y}\right)\frac{\sigma_x}{\sigma_y}\rho B_2(\z)\right\}\\
&=& \left\{\frac{\epsilon_y}{\sigma_y^2} B_2(\z)\right\}.
\ese

Therefore,
\be\label{eq:lambda-ortho}
\Lambda=\Lambda_1\oplus\wt\Lambda_2\oplus\Lambda_3\oplus\wt\Lambda_4\oplus\Lambda_5,
\ee
where
\bse
\Lambda_1 &=& \left\{c_1\left(-\frac{1}{\sigma_x}+\frac{1}{1-\rho^2}\frac{\epsilon_x^2}{\sigma_x^3}-\frac{\rho}{1-\rho^2}\frac{\epsilon_x\epsilon_y}{\sigma_x^2\sigma_y}\right)\right\},\\
\wt\Lambda_2 &=& \left\{c_2\left\{-\frac{2}{2-\rho^2}\frac{1}{\sigma_y}+\frac{1}{1-\rho^2}\frac{\epsilon_y^2}{\sigma_y^3}+\frac{\rho^2}{(1-\rho^2)(2-\rho^2)}\frac{\epsilon_x^2}{\sigma_x^2\sigma_y}-\frac{2\rho}{(1-\rho^2)(2-\rho^2)}\frac{\epsilon_x\epsilon_y}{\sigma_x\sigma_y^2}\right\}\right\},\\
\Lambda_3 &=& \left\{\frac{1}{1-\rho^2}\left(\frac{\epsilon_x}{\sigma_x^2}-\rho\frac{\epsilon_y}{\sigma_x\sigma_y}\right)B_1(\z)\right\},\\
\wt\Lambda_4 &=& \left\{\frac{\epsilon_y}{\sigma_y^2} B_2(\z)\right\},\\
\Lambda_5 &=& \{a(\z):E(a)=0\}.
\ese

By \eqref{eq:lambda-ortho}, we know that 
\bse
\Lambda^\perp = \Lambda_1^\perp\cap\wt\Lambda_2^\perp\cap\Lambda_3^\perp\cap\wt\Lambda_4^\perp\cap\Lambda_5^\perp.
\ese

Denote $\wt\epsilon_x=\epsilon_x/\sigma_x$ and $\wt\epsilon_y=\epsilon_y/\sigma_y$. Then,
\bse
\Lambda_1^\perp &=& \left\{b(x,y,\z):E\left\{\left(-\frac{1}{\sigma_x}+\frac{1}{1-\rho^2}\frac{\epsilon_x^2}{\sigma_x^3}-\frac{\rho}{1-\rho^2}\frac{\epsilon_x\epsilon_y}{\sigma_x^2\sigma_y}\right)b(X,Y,\Z)\right\}=0\right\}\\
&=& \left\{b(x,y,\z):E\left[\left\{-(1-\rho^2)+\wt\epsilon_x^2-\rho\wt\epsilon_x\wt\epsilon_y\right\}b(X,Y,\Z)\right]=0\right\},
\ese
and
\bse
\wt\Lambda_2^\perp &=& \left\{b(x,y,\z):E\left[\left\{-\frac{2}{2-\rho^2}\frac{1}{\sigma_y}+\frac{1}{1-\rho^2}\frac{\epsilon_y^2}{\sigma_y^3}+\frac{\rho^2}{(1-\rho^2)(2-\rho^2)}\frac{\epsilon_x^2}{\sigma_x^2\sigma_y}\right.\right.\right.\\
&&\qquad\qquad\qquad\qquad\left.\left.\left.-\frac{2\rho}{(1-\rho^2)(2-\rho^2)}\frac{\epsilon_x\epsilon_y}{\sigma_x\sigma_y^2}\right\}b(X,Y,\Z)\right]=0\right\}\\
&=& \left\{b(x,y,\z):E\left[\left\{-2(1-\rho^2)+(2-\rho^2)\wt\epsilon_y^2+\rho^2\wt\epsilon_x^2-2\rho\wt\epsilon_x\wt\epsilon_y\right\}b(X,Y,\Z)\right]=0\right\}.
\ese
Also,
\bse
\Lambda_3^\perp &=& \left\{b(x,y,\z):E\left\{\frac{1}{1-\rho^2}\left(\frac{\epsilon_x}{\sigma_x^2}-\rho\frac{\epsilon_y}{\sigma_x\sigma_y}\right)B_1(\Z)b(X,Y,\Z)\right\}=0,\ \forall B_1(\z)\right\}\\
&=& \left\{b(x,y,\z):E\left\{\frac{1}{1-\rho^2}\left(\frac{\epsilon_x}{\sigma_x^2}-\rho\frac{\epsilon_y}{\sigma_x\sigma_y}\right)b(X,Y,\Z)\Big|\Z\right\}=0\right\}\\
&=& \left\{b(x,y,\z):E\left\{\left(\wt\epsilon_x-\rho\wt\epsilon_y\right)b(X,Y,\Z)\Big|\Z\right\}=0\right\}
\ese
and
\bse
\wt\Lambda_4^\perp &=& \left\{b(x,y,\z):E\left\{\frac{\epsilon_y}{\sigma_y^2} B_2(Z)b(X,Y,\Z)\right\}=0,\ \forall B_2(\z)\right\}\\
&=& \left\{b(x,y,\z):E\left\{\wt\epsilon_y b(X,Y,\Z)\Big|\Z\right\}=0\right\}.
\ese
Also obviously,
\bse
\Lambda_5^\perp=\left\{b(x,y,\z):E\left\{b(X,Y,\Z)|\Z\right\}=0\right\}.
\ese

Therefore,
\bse
\Lambda^\perp &=& \left\{b(x,y,\z): E\left[\left\{-(1-\rho^2)+\wt\epsilon_x^2-\rho\wt\epsilon_x\wt\epsilon_y\right\}b(X,Y,\Z)\right]=0,\right.\\
&& \left.E\left[\left\{-2(1-\rho^2)+(2-\rho^2)\wt\epsilon_y^2+\rho^2\wt\epsilon_x^2-2\rho\wt\epsilon_x\wt\epsilon_y\right\}b(X,Y,\Z)\right]=0,\right.\\
&& \left.E\left\{\left(\wt\epsilon_x-\rho\wt\epsilon_y\right)b(X,Y,\Z)\Big|Z\right\}=0, E\left\{\wt\epsilon_y b(X,Y,\Z)\Big|\Z\right\}=0,E\left\{b(X,Y,\Z)|\Z\right\}=0\right\}\\
&=& \left\{b(x,y,\z): E\left\{\left(\wt\epsilon_x^2-\rho\wt\epsilon_x\wt\epsilon_y\right)b(X,Y,\Z)\right\}=0,E\left\{\left(\wt\epsilon_y^2-\rho\wt\epsilon_x\wt\epsilon_y\right)b(X,Y,\Z)\right\}=0,\right.\\
&& \left.E\left\{\wt\epsilon_x b(X,Y,\Z)\Big|\Z\right\}=0, E\left\{\wt\epsilon_y b(X,Y,\Z)\Big|\Z\right\}=0,E\left\{b(X,Y,\Z)|\Z\right\}=0\right\}.
\ese
\end{proof}

\begin{proof}[Proof of Proposition \ref{prop:eif}]
Note that
\bse
S_\rho 
&=& -\frac{1}{2}\frac{-2\rho}{1-\rho^2}+\frac{-2\rho}{2(1-\rho^2)^2}\left(\frac{\epsilon_x^2}{\sigma_x^2}-2\rho\frac{\epsilon_x\epsilon_y}{\sigma_x\sigma_y}+\frac{\epsilon_y^2}{\sigma_y^2}\right)-\frac{1}{2(1-\rho^2)}\left(-2\frac{\epsilon_x\epsilon_y}{\sigma_x\sigma_y}\right)\\
&=& \frac{\rho}{1-\rho^2}-\frac{\rho}{(1-\rho^2)^2}\left(\wt\epsilon_x^2-2\rho\wt\epsilon_x\wt\epsilon_y+\wt\epsilon_y^2\right)+\frac{1}{1-\rho^2}\wt\epsilon_x\wt\epsilon_y.
\ese

Thus,
\bse
E\left(S_\rho | \Z\right)
&=& E\left\{\frac{\rho}{1-\rho^2}-\frac{\rho}{(1-\rho^2)^2}\left(\wt\epsilon_x^2-2\rho\wt\epsilon_x\wt\epsilon_y+\wt\epsilon_y^2\right)+\frac{1}{1-\rho^2}\wt\epsilon_x\wt\epsilon_y\right\}\\
&=& \frac{\rho}{1-\rho^2}-\frac{\rho}{(1-\rho^2)^2}\left(1-2\rho^2+1\right)+\frac{1}{1-\rho^2}\rho\\
&=& \frac{\rho}{1-\rho^2}-\frac{2\rho}{1-\rho^2}+\frac{1}{1-\rho^2}\rho\\
&=& 0.
\ese
Also,
\bse
E\left(\wt\epsilon_x S_\rho|\Z\right)
= E\left\{\frac{\rho}{1-\rho^2}\wt\epsilon_x -\frac{\rho}{(1-\rho^2)^2}\left(\wt\epsilon_x^3-2\rho\wt\epsilon_x^2\wt\epsilon_y+\wt\epsilon_x \wt\epsilon_y^2\right)+\frac{1}{1-\rho^2}\wt\epsilon_x^2\wt\epsilon_y\right\}
= 0
\ese
since it only uses the first and third moments of $\epsilon_x$ and $\epsilon_y$. Similarly,
\bse
E\left(\wt\epsilon_y S_\rho|\Z\right)
= E\left\{\frac{\rho}{1-\rho^2}\wt\epsilon_y -\frac{\rho}{(1-\rho^2)^2}\left(\wt\epsilon_x^2\wt\epsilon_y-2\rho\wt\epsilon_x\wt\epsilon_y^2+\wt\epsilon_y^3\right)+\frac{1}{1-\rho^2}\wt\epsilon_x\wt\epsilon_y^2\right\}
= 0.
\ese

Then,
\bse
&&E\left\{\left(\wt\epsilon_x^2-\rho\wt\epsilon_x\wt\epsilon_y\right)S_\rho\right\}\\
&=& E\left[\left(\wt\epsilon_x^2-\rho\wt\epsilon_x\wt\epsilon_y\right) \left\{\frac{\rho}{1-\rho^2}-\frac{\rho}{(1-\rho^2)^2}\left(\wt\epsilon_x^2-2\rho\wt\epsilon_x\wt\epsilon_y+\wt\epsilon_y^2\right)+\frac{1}{1-\rho^2}\wt\epsilon_x\wt\epsilon_y\right\}\right]\\
&=& \frac{\rho}{1-\rho^2}\left(1-\rho^2\right)-\frac{\rho}{(1-\rho^2)^2}\left\{3-6\rho^2+1+2\rho^2-3\rho^2+2\rho^2(1+2\rho^2)-3\rho^2\right\}\\
&&+\frac{1}{1-\rho^2}\left\{3\rho-\rho(1+2\rho^2)\right\}\\
&=& \rho-\frac{\rho}{(1-\rho^2)^2}\left(4-8\rho^2+4\rho^4\right)+2\rho\\
&=& -\rho,
\ese
and by symmetry, 
\bse
E\left\{\left(\wt\epsilon_y^2-\rho\wt\epsilon_x\wt\epsilon_y\right)S_\rho\right\}= -\rho.
\ese

Let
\bse
\wt S_\rho = S_\rho + \rho \frac{\wt\epsilon_x^2-2\rho\wt\epsilon_x\wt\epsilon_y+\wt\epsilon_y^2-2(1-\rho^2)}{2(1-\rho^2)^2}
\ese
Note that 
\bse
E\left\{\frac{\wt\epsilon_x^2-2\rho\wt\epsilon_x\wt\epsilon_y+\wt\epsilon_y^2-2(1-\rho^2)}{2(1-\rho^2)^2}\Big|\Z\right\}
= \frac{1-2\rho^2+1-2(1-\rho^2)}{2(1-\rho^2)^2}
= 0,
\ese
and 
\bse
E\left\{\wt\epsilon_x\frac{\wt\epsilon_x^2-2\rho\wt\epsilon_x\wt\epsilon_y+\wt\epsilon_y^2-2(1-\rho^2)}{2(1-\rho^2)^2}\Big|\Z\right\}=E\left\{\wt\epsilon_y\frac{\wt\epsilon_x^2-2\rho\wt\epsilon_x\wt\epsilon_y+\wt\epsilon_y^2-2(1-\rho^2)}{2(1-\rho^2)^2}\Big|\Z\right\}=0
\ese
since they both only use the first and third moments of a multivariate normal distribution. Thus, we have
\bse
E(\wt S_\rho|\Z) = E(\wt\epsilon_x\wt S_\rho|\Z) = E(\wt\epsilon_y\wt S_\rho|\Z) =0.
\ese
Furthermore, 
\bse
&&E\left[\left(\wt\epsilon_x^2-\rho\wt\epsilon_x\wt\epsilon_y\right)\left\{\frac{\wt\epsilon_x^2-2\rho\wt\epsilon_x\wt\epsilon_y+\wt\epsilon_y^2-2(1-\rho^2)}{2(1-\rho^2)^2}\right\}\right]\\
&=&E\left[\frac{1}{2(1-\rho^2)^2}\left\{3-6\rho^2+1+2\rho^2-3\rho^2+2\rho^2(1+2\rho^2)-3\rho^2-2(1-\rho^2)^2\right\}\right]\\
&=&E\left[\frac{1}{2(1-\rho^2)^2}\left\{4-8\rho^2+4\rho^4-2(1-\rho^2)^2\right\}\right]\\
&=&E\left[\frac{1}{2(1-\rho^2)^2}\left\{4(1-\rho^2)^2-2(1-\rho^2)^2\right\}\right]\\
&=& 1.
\ese
Thus, we have
\bse
E\left\{\left(\wt\epsilon_x^2-\rho\wt\epsilon_x\wt\epsilon_y\right)\wt S_\rho\right\}= -\rho + \rho =0.
\ese
By symmetry, 
\bse
E\left\{\left(\wt\epsilon_x^2-\rho\wt\epsilon_x\wt\epsilon_y\right)\wt S_\rho\right\}= -\rho + \rho =0.
\ese

Therefore, $\wt S_\rho\in\Lambda^\perp$. Also, notice that 
\bse
E\left\{\frac{\wt\epsilon_x^2-2\rho\wt\epsilon_x\wt\epsilon_y+\wt\epsilon_y^2-2(1-\rho^2)}{2(1-\rho^2)^2}\right\}
=\frac{1-2\rho^2+1-2(1-\rho^2)}{2(1-\rho^2)^2}
= 0,
\ese
so 
\bse
\frac{\wt\epsilon_x^2-2\rho\wt\epsilon_x\wt\epsilon_y+\wt\epsilon_y^2-2(1-\rho^2)}{2(1-\rho^2)^2}\in\Lambda_5\subset\Lambda.
\ese
Therefore, $\wt S_\rho = \Pi (S_\rho|\Lambda^\perp)$. Thus,
\bse
S\eff
= \wt S_\rho 
&=& \frac{\rho}{1-\rho^2}-\frac{\rho}{(1-\rho^2)^2}\left(\wt\epsilon_x^2-2\rho\wt\epsilon_x\wt\epsilon_y+\wt\epsilon_y^2\right)+\frac{1}{1-\rho^2}\wt\epsilon_x\wt\epsilon_y \\
&& + \rho \frac{\wt\epsilon_x^2-2\rho\wt\epsilon_x\wt\epsilon_y+\wt\epsilon_y^2-2(1-\rho^2)}{2(1-\rho^2)^2}\\
&=& -\frac{\rho}{2(1-\rho^2)^2}\left(\wt\epsilon_x^2-2\rho\wt\epsilon_x\wt\epsilon_y+\wt\epsilon_y^2\right)+\frac{1}{1-\rho^2}\wt\epsilon_x\wt\epsilon_y \\
&=& -\frac{1}{2(1-\rho^2)^2}\left(\rho\wt\epsilon_x^2-2\wt\epsilon_x\wt\epsilon_y+\rho\wt\epsilon_y^2\right).
\ese

Based on the form of $S\eff$, we have 
\bse
E(S\eff^2) 
&=& \frac{1}{4(1-\rho^2)^4} E\left\{\left(\rho\wt\epsilon_x^2-2\wt\epsilon_x\wt\epsilon_y+\rho\wt\epsilon_y^2\right)^2\right\}\\
&=& \frac{1}{4(1-\rho^2)^4} \left\{3\rho^2 + 2\rho^2(1+2\rho^2) +3\rho^2 - 12\rho^2 +2\rho^2(1+2\rho^2) -12\rho^2\right\}\\
&=& \frac{1}{4(1-\rho^2)^4} \left\{4-8\rho^2+4\rho^4\right\}\\
&=& \frac{1}{4(1-\rho^2)^4} \left\{4(1-\rho^2)^2\right\}\\
&=& \frac{1}{(1-\rho^2)^2}.
\ese
Thus, the semiparametric efficiency bound is
\bse
E(S\eff^2)^{-1}=(1-\rho^2)^2.
\ese
The efficient influence function is
\bse
\phi\eff(x,y,\z)=[E(S\eff^2)]^{-1}S\eff
=-\frac{1}{2}\left(\rho\wt\epsilon_x^2-2\wt\epsilon_x\wt\epsilon_y+\rho\wt\epsilon_y^2\right).
\ese
\end{proof}

\begin{proof}[Proof of Theorem \ref{thm:rho2-if}]
By Taylor's mean value theorem, we have
\be\label{eq:rho2-decomp}
n_2^{1/2}(\wh\rho_2 - \rho) 
&=& n_2^{1/2} \frac{1}{\sigma_x\sigma_y}(\wh\sigma_{xy2}-\sigma_{xy})
- n_2^{1/2} \frac{\sigma_{xy}}{2\sigma_x^3\sigma_y}(\wh\sigma_{x2}^2-\sigma_x^2)\n\\
&&- n_2^{1/2} \frac{\sigma_{xy}}{2\sigma_x\sigma_y^3}(\wh\sigma_{y2}^2-\sigma_y^2)
+ R,
\ee
where
\be\label{eq:remainder}
R&=& - n_2^{1/2}\frac{1}{2\sigma_{x*}^3\sigma_{y*}}(\wh\sigma_{xy2}-\sigma_{xy})(\wh\sigma_{x2}^2-\sigma_x^2)
- n_2^{1/2}\frac{1}{2\sigma_{x*}\sigma_{y*}^3}(\wh\sigma_{xy2}-\sigma_{xy})(\wh\sigma_{y2}^2-\sigma_y^2) \n\\
&& + n_2^{1/2}\frac{\sigma_{xy*}}{4\sigma_{x*}^3\sigma_{y*}^3}(\wh\sigma_{x2}^2-\sigma_x^2)(\wh\sigma_{y2}^2-\sigma_y^2)\n\\
&& - n_2^{1/2}\frac{3\sigma_{xy*}}{8\sigma_{x*}^5\sigma_{y*}}(\wh\sigma_{x2}^2-\sigma_x^2)^2
- n_2^{1/2}\frac{3\sigma_{xy*}}{8\sigma_{x*}\sigma_{y*}^5}(\wh\sigma_{y2}^2-\sigma_y^2)^2,
\ee
for some $\sigma_{x*}^2$ between $\sigma_x^2$ and $\wh\sigma_{x2}^2$, $\sigma_{y*}^2$ between $\sigma_y^2$ and $\wh\sigma_{y2}^2$, and $\sigma_{xy*}$ between $\sigma_{xy}$ and $\wh\sigma_{xy2}$.

We decompose $\wh \sigma_{x2}^2$ into three parts as follows:
\be\label{eq:sigmax2-decomp}
\wh \sigma_{x2}^2 - \sigma_x^2
&=& n_2^{-1} \sumib\{X_i- \wh m_{x1}(\Z_i)\}^2 - \sigma_x^2 \n\\
&=& n_2^{-1} \sumib\{X_i- m_x(\Z_i)\}^2 
 + n_2^{-1} \sumib\{m_x(\Z_i) - \wh m_{x1}(\Z_i)\}^2 \n\\
&& + 2 n_2^{-1} \sumib\{m_x(\Z_i) - \wh m_{x1}(\Z_i)\}\{X_i- m_x(\Z_i)\}
 - \sigma_x^2\n\\
&=& n_2^{-1} \sumib \epsilon_{xi}^2 
 + n_2^{-1} \sumib\{m_x(\Z_i) - \wh m_{x1}(\Z_i)\}^2 \n\\
&& + 2 n_2^{-1} \sumib\{m_x(\Z_i) - \wh m_{x1}(\Z_i)\}\epsilon_{xi}
 - \sigma_x^2.
\ee
Under $\|\wh m_{x1} - m_x\|_2 = o_p(n_1^{-1/4})$, we have
\bse
E \left[ n_2^{-1} \sumib\{m_x(\Z_i) - \wh m_{x1}(\Z_i)\}^2 | \wh m_{x1} \right]
&=& E \left[ \{m_x(\Z) - \wh m_{x1}(\Z)\}^2 | \wh m_{x1} \right] \\
&=& \|\wh m_{x1} - m_x\|_2^2 \\
&=& o_p(n_1^{-1/2}),
\ese
and since the sum is nonnegative, by Markov's inequality, we have 
\bse
n_2^{-1} \sumib\{m_x(\Z_i) - \wh m_{x1}(\Z_i)\}^2 = o_p(n_1^{-1/2}).
\ese
Also, note that
\bse
n_2^{-1} \sumib \epsilon_{xi}^2 = \sigma_x^2 + o_p(1).
\ese
Furthermore, 
\bse
E \left[ \{m_x(\Z) - \wh m_{x1}(\Z)\}\epsilon_x | \wh m_{x1} \right] 
= E \left[ \{m_x(\Z) - \wh m_{x1}(\Z)\} | \wh m_{x1} \right] E(\epsilon_x) = 0,
\ese
and under $\|\wh m_{x1} - m_x\|_2 = o_p(n_1^{-1/4})$,
\bse
&& \var \left[ n_2^{-1} \sumib\{m_x(\Z_i) - \wh m_{x1}(\Z_i)\}\epsilon_{xi} | \wh m_{x1} \right]\\
&=& n_2^{-1} \var \left[ \{m_x(\Z) - \wh m_{x1}(\Z)\}\epsilon_x | \wh m_{x1} \right]\\
&=& n_2^{-1} E \left[\{m_x(\Z) - \wh m_{x1}(\Z)\}^2\epsilon_x^2  | \wh m_{x1} \right]\\
&=& n_2^{-1} E \left[\{m_x(\Z) - \wh m_{x1}(\Z)\}^2  | \wh m_{x1} \right] E(\epsilon_x^2)\\
&=& n_2^{-1} \|\wh m_{x1} - m_x\|_2^2 \sigma_x^2\\
&=& o_p(n_2^{-1}n_1^{-1/2}),
\ese
so by Chebyshev's inequality, we have 
\bse
n_2^{-1} \sumib\{m_x(\Z_i) - \wh m_{x1}(\Z_i)\}\epsilon_{xi} 
= o_p(n_2^{-1/2}n_1^{-1/4}).
\ese
Plugging back to \eqref{eq:sigmax2-decomp}, we have
\be\label{eq:sigmax2-asymp}
\wh \sigma_{x2}^2 - \sigma_x^2
= n_2^{-1} \sumib \epsilon_{xi}^2 - \sigma_x^2
+ o_p(n_1^{-1/2} + n_2^{-1/2}n_1^{-1/4}).
\ee
Same arguments lead to, under $\|\wh m_{y1} - m_y\|_2 = o_p(n_1^{-1/4})$,
\be\label{eq:sigmay2-asymp}
\wh \sigma_{y2}^2 - \sigma_y^2
= n_2^{-1} \sumib \epsilon_{yi}^2 - \sigma_y^2
+ o_p(n_1^{-1/2} + n_2^{-1/2}n_1^{-1/4}).
\ee

We also decompose $\wh \sigma_{xy}$ into three parts as follows:
\be\label{eq:sigmaxy-decomp}
\wh \sigma_{xy2} - \sigma_{xy}
&=& n_2^{-1} \sumib\{X_i- \wh m_{x1}(\Z_i)\}\{Y_i- \wh m_{y1}(\Z_i)\} - \sigma_{xy} \n\\
&=& n_2^{-1} \sumib\{X_i- m_x(\Z_i)\}\{Y_i- m_y(\Z_i)\} \n\\
&& + n_2^{-1} \sumib\{m_x(\Z_i) - \wh m_{x1}(\Z_i)\}\{m_y(\Z_i) - \wh m_{y1}(\Z_i)\} \n\\
&& + n_2^{-1} \sumib\{m_x(\Z_i) - \wh m_{x1}(\Z_i)\}\{Y_i- m_y(\Z_i)\} \n\\
&& + n_2^{-1} \sumib\{m_y(\Z_i) - \wh m_{y1}(\Z_i)\}\{X_i- m_x(\Z_i)\}
 - \sigma_{xy} \n\\
&=& n_2^{-1} \sumib \epsilon_{xi} \epsilon_{yi}
 + n_2^{-1} \sumib\{m_x(\Z_i) - \wh m_{x1}(\Z_i)\}\{m_y(\Z_i) - \wh m_{y1}(\Z_i)\} \n\\
&& + n_2^{-1} \sumib\{m_x(\Z_i) - \wh m_{x1}(\Z_i)\}\epsilon_{yi}
+ n_2^{-1} \sumib\{m_y(\Z_i) - \wh m_{y1}(\Z_i)\}\epsilon_{xi} \n\\
&& - \sigma_{xy}.
\ee
Under $\|\wh m_{x1} - m_x\|_2 = o_p(n_1^{-1/4})$ and $\|\wh m_{y1} - m_y\|_2 = o_p(n_1^{-1/4})$, we have
\bse
&&  E \left[ \left| n_2^{-1} \sumib\{m_x(\Z_i) - \wh m_{x1}(\Z_i)\}\{m_y(\Z_i) - \wh m_{y1}(\Z_i)\} \right| | \wh m_{x1}, \wh m_{y1} \right]  \\
&\le&  E \left[ \left|m_x(\Z) - \wh m_{x1}(\Z)\right| \left|m_y(\Z) - \wh m_{y1}(\Z) \right|| \wh m_{x1}, \wh m_{y1} \right]  \\
&=&  \left\langle \left| \wh m_{x1} - m_x\right|, \left|\wh m_{y1} - m_y\right| \right\rangle_2  \\
&\le& \|\wh m_{x1} - m_x\|_2 \|\wh m_{y1} - m_y\|_2 \\
&=& o_p(n_1^{-1/4}) o_p(n_1^{-1/4}) \\
&=& o_p(n_1^{-1/2}).
\ese
By Markov's inequality, we have 
\bse
n_2^{-1} \sumib\{m_x(\Z_i) - \wh m_{x1}(\Z_i)\}\{m_y(\Z_i) - \wh m_{y1}(\Z_i)\} = o_p(n_1^{-1/2}).
\ese
Furthermore, 
\bse
E \left[ \{m_x(\Z) - \wh m_{x1}(\Z)\}\epsilon_y | \wh m_{x1} \right] 
= E \left[ \{m_x(\Z) - \wh m_{x1}(\Z)\} | \wh m_{x1} \right] E(\epsilon_y) = 0,
\ese
and under $\|\wh m_{x1} - m_x\|_2 = o_p(n_1^{-1/4})$,
\bse
&& \var \left[ n_2^{-1} \sumib\{m_x(\Z_i) - \wh m_{x1}(\Z_i)\}\epsilon_{yi} | \wh m_{x1} \right]\\
&=& n_2^{-1} \var \left[ \{m_x(\Z) - \wh m_{x1}(\Z)\}\epsilon_y | \wh m_{x1} \right]\\
&=& n_2^{-1} E \left[\{m_x(\Z) - \wh m_{x1}(\Z)\}^2\epsilon_y^2  | \wh m_{x1} \right]\\
&=& n_2^{-1} E \left[\{m_x(\Z) - \wh m_{x1}(\Z)\}^2  | \wh m_{x1} \right] E(\epsilon_y^2)\\
&=& n_2^{-1} \|\wh m_{x1} - m_x\|_2^2 \sigma_y^2\\
&=& o_p(n_2^{-1}n_1^{-1/2}),
\ese
so by Chebyshev's inequality, we have 
\bse
n_2^{-1} \sumib\{m_x(\Z_i) - \wh m_{x1}(\Z_i)\}\epsilon_{yi} 
= o_p(n_2^{-1/2}n_1^{-1/4}).
\ese
Same arguments lead to, under $\|\wh m_{y1} - m_y\|_2 = o_p(n_1^{-1/4})$, 
\bse
n_2^{-1} \sumib\{m_y(\Z_i) - \wh m_{y1}(\Z_i)\}\epsilon_{xi} 
= o_p(n_2^{-1/2}n_1^{-1/4}).
\ese
Plugging back to \eqref{eq:sigmaxy-decomp}, we have
\be\label{eq:sigmaxy-asymp}
\wh \sigma_{xy2} - \sigma_{xy}
= n_2^{-1} \sumib \epsilon_{xi}\epsilon_{yi} - \sigma_{xy}
+ o_p(n_1^{-1/2} + n_2^{-1/2}n_1^{-1/4}).
\ee

Applying central limit theorem and Slutsky's theorem to \eqref{eq:sigmax2-asymp}, \eqref{eq:sigmay2-asymp} and \eqref{eq:sigmaxy-asymp}, if $n_1 \asymp n_2 \asymp n$, we have
\bse
\wh \sigma_{x2}^2 - \sigma_x^2 = O_p(n^{-1/2}), \quad
\wh \sigma_{y2}^2 - \sigma_y^2 = O_p(n^{-1/2}), \quad
\wh \sigma_{xy2} - \sigma_{xy} = O_p(n^{-1/2}),
\ese
and thus,
\bse
\wh \sigma_{x*}^2 - \sigma_x^2 = o_p(1), \quad
\wh \sigma_{y*}^2 - \sigma_y^2 = o_p(1), \quad
\wh \sigma_{xy*} - \sigma_{xy} = o_p(1).
\ese
Plugging back to \eqref{eq:remainder}, when $n_1 \asymp n_2 \asymp n$, we have
\bse
R
&=& - n_2^{1/2}\left\{\frac{1}{2\sigma_x^3\sigma_y}+o_p(1)\right\}O_p(n^{-1/2})O_p(n^{-1/2})\\
&&- n_2^{1/2}\left\{\frac{1}{2\sigma_x\sigma_y^3}+o_p(1)\right\}O_p(n^{-1/2})O_p(n^{-1/2}) \n\\
&& + n_2^{1/2}\left\{\frac{\sigma_{xy}}{4\sigma_x^3\sigma_y^3}+o_p(1)\right\}O_p(n^{-1/2})O_p(n^{-1/2})\n\\
&& - n_2^{1/2}\left\{\frac{3\sigma_{xy}}{8\sigma_x^5\sigma_y}+o_p(1)\right\}O_p(n^{-1/2})^2
- n_2^{1/2}\left\{\frac{3\sigma_{xy}}{8\sigma_x\sigma_y^5}+o_p(1)\right\}O_p(n^{-1/2})^2\\
&=& O_p(n^{-1/2}).
\ese
Also, when $n_1 \asymp n_2 \asymp n$, the remainder terms in \eqref{eq:sigmax2-asymp}, \eqref{eq:sigmay2-asymp} and \eqref{eq:sigmaxy-asymp} all become $o_p(n^{-1/2})$. Therefore, plugging back into \eqref{eq:rho2-decomp}, we have
\bse
&& n_2^{1/2}(\wh\rho_2 - \rho) \\
&=& \frac{1}{\sigma_x\sigma_y}n_2^{-1/2} \sumib (\epsilon_{xi}\epsilon_{yi} - \sigma_{xy})
- \frac{\sigma_{xy}}{2\sigma_x^3\sigma_y}n_2^{-1/2} \sumib (\epsilon_{xi}^2 - \sigma_x^2)\n\\
&&- \frac{\sigma_{xy}}{2\sigma_x\sigma_y^3}n_2^{-1/2} \sumib (\epsilon_{yi}^2 - \sigma_y^2)
+ O_p(n^{-1/2})\\
&=& n_2^{-1/2} \sumib \left(\frac{\epsilon_{xi}\epsilon_{yi}}{\sigma_x\sigma_y} - \frac{\sigma_{xy}}{\sigma_x\sigma_y} 
- \frac{\sigma_{xy}\epsilon_{xi}^2}{2\sigma_x^3\sigma_y} + \frac{\sigma_{xy}}{2\sigma_x\sigma_y}
- \frac{\sigma_{xy}\epsilon_{yi}^2}{2\sigma_x\sigma_y^3} + \frac{\sigma_{xy}}{2\sigma_x\sigma_y}\right)
+ O_p(n^{-1/2})\\
&=& n_2^{-1/2} \sumib \left(\wt\epsilon_{xi}\wt\epsilon_{yi}
- \frac{\rho}{2}\wt\epsilon_{xi}^2
- \frac{\rho}{2}\wt\epsilon_{yi}^2 \right)
+ O_p(n^{-1/2})\\
&=& n_2^{-1/2} \sumib -\frac{1}{2}\left(\rho\wt\epsilon_{xi}^2-2\wt\epsilon_{xi}\wt\epsilon_{yi}+\rho\wt\epsilon_{yi}^2\right)
+ O_p(n^{-1/2}).
\ese
Comparing with \eqref{eq:eif}, we have
\be\label{eq:rho2-if}
n_2^{1/2}(\wh\rho_2 - \rho)
= n_2^{-1/2} \sumib \phi\eff(X_i,Y_i,\Z_i)
+ O_p(n^{-1/2}).
\ee

Similarly, under 
$\|\wh m_{x2} - m_x\|_2 = o_p(n_2^{-1/4})$ and
$\|\wh m_{y2} - m_y\|_2 = o_p(n_2^{-1/4})$
with $n_1 \asymp n_2 \asymp n$, we also have
\be\label{eq:rho1-if}
n_1^{1/2}(\wh\rho_1 - \rho) 
= n_1^{-1/2} \sumia \phi\eff(X_i,Y_i,\Z_i)
+ O_p(n^{-1/2}).
\ee

Then, a direct application of \eqref{eq:rho2-if} and \eqref{eq:rho1-if} gives the asymptotic expansion for $\wh\rho$ as
\bse
n^{1/2}(\wh\rho - \rho) 
= n^{-1/2} \sumi \phi\eff(X_i,Y_i,\Z_i)
+ O_p(n^{-1/2}).
\ese
\end{proof}

\newpage

\section{Additional Simulation Results}

\begin{table}[htbp]
    \centering
    \footnotesize
\begin{tabular}{rrrrrrrrr}
\hline
 $n$ & RPCO & RPCS & RPCF & PaCo & RHSIC & RRIT & RCIT & RCoT\\
\hline
\hline
 100 & 0.036 & 0.074 & 0.062 & 0.072 & 0.052 & 0.044 & 1.000 & 1.000\\
 200 & 0.032 & 0.054 & 0.038 & 0.050 & 0.048 & 0.040 & 0.300 & 0.222\\
 500 & 0.068 & 0.068 & 0.072 & 0.072 & 0.052 & 0.044 & 0.094 & 0.096\\
 1000 & 0.058 & 0.068 & 0.058 & 0.056 & 0.034 & 0.040 & 0.058 & 0.060\\
 2000 & 0.046 & 0.046 & 0.042 & 0.044 & 0.050 & 0.076 & 0.074 & 0.050\\
 5000 & 0.042 & 0.050 & 0.040 & 0.042 & 0.058 & 0.046 & 0.054 & 0.056\\
\hline
\end{tabular}
\caption{Empirical levels of eight tests under Model 2 based on 500 experiments.}
\vspace{1cm}
    \label{tab:ep-level-2a}
\end{table}

\begin{figure}[htb]
    \centering
    \includegraphics[width=0.75\linewidth,page=2]{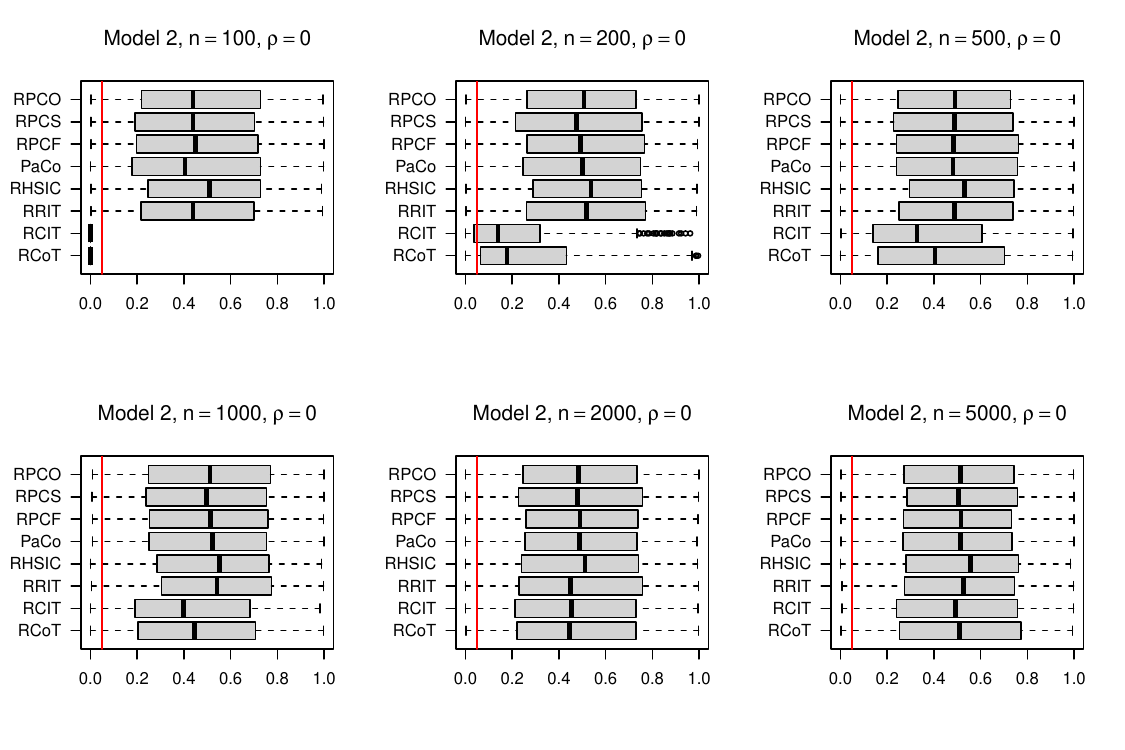}
    \caption{Boxplots of p-values of eight tests for Model 2 under the null hypothesis, with sample sizes $n=100,200,500,1000,2000,5000$. The red line represents 0.05.}
    \label{fig:boxplot-null-2a}
\end{figure}

\begin{table}[htbp]
    \centering
    \footnotesize
\begin{tabular}{rr|rrrrrrrr}
\hline
$\rho$ & $n$ & RPCO & RPCS & RPCF & PaCo & RHSIC & RRIT & RCIT & RCoT\\
\hline
\hline
-0.25 & 100 & 0.682 & 0.588 & 0.644 & 0.670 & 0.310 & 0.456 & 1.000 & 1.000\\
-0.25 & 200 & 0.962 & 0.930 & 0.950 & 0.952 & 0.622 & 0.778 & 0.714 & 0.750\\
-0.25 & 500 & 1.000 & 1.000 & 1.000 & 1.000 & 0.970 & 0.986 & 0.946 & 0.978\\
-0.25 & 1000 & 1.000 & 1.000 & 1.000 & 1.000 & 1.000 & 1.000 & 0.982 & 0.998\\
-0.25 & 2000 & 1.000 & 1.000 & 1.000 & 1.000 & 1.000 & 1.000 & 1.000 & 0.998\\
-0.25 & 5000 & 1.000 & 1.000 & 1.000 & 1.000 & 1.000 & 1.000 & 0.998 & 1.000\\
\hline
-0.50 & 100 & 1.000 & 0.996 & 1.000 & 1.000 & 0.948 & 0.970 & 1.000 & 1.000\\
-0.50 & 200 & 1.000 & 1.000 & 1.000 & 1.000 & 1.000 & 1.000 & 0.972 & 0.998\\
-0.50 & 500 & 1.000 & 1.000 & 1.000 & 1.000 & 1.000 & 1.000 & 1.000 & 1.000\\
-0.50 & 1000 & 1.000 & 1.000 & 1.000 & 1.000 & 1.000 & 1.000 & 1.000 & 1.000\\
-0.50 & 2000 & 1.000 & 1.000 & 1.000 & 1.000 & 1.000 & 1.000 & 1.000 & 1.000\\
-0.50 & 5000 & 1.000 & 1.000 & 1.000 & 1.000 & 1.000 & 1.000 & 1.000 & 1.000\\
\hline
-0.75 & 100 & 1.000 & 1.000 & 1.000 & 1.000 & 1.000 & 1.000 & 1.000 & 1.000\\
-0.75 & 200 & 1.000 & 1.000 & 1.000 & 1.000 & 1.000 & 1.000 & 1.000 & 1.000\\
-0.75 & 500 & 1.000 & 1.000 & 1.000 & 1.000 & 1.000 & 1.000 & 1.000 & 1.000\\
-0.75 & 1000 & 1.000 & 1.000 & 1.000 & 1.000 & 1.000 & 1.000 & 1.000 & 1.000\\
-0.75 & 2000 & 1.000 & 1.000 & 1.000 & 1.000 & 1.000 & 1.000 & 1.000 & 1.000\\
-0.75 & 5000 & 1.000 & 1.000 & 1.000 & 1.000 & 1.000 & 1.000 & 1.000 & 1.000\\
\hline
-1.00 & 100 & 1.000 & 1.000 & 1.000 & 1.000 & 1.000 & 1.000 & 1.000 & 1.000\\
-1.00 & 200 & 1.000 & 1.000 & 1.000 & 1.000 & 1.000 & 1.000 & 1.000 & 1.000\\
-1.00 & 500 & 1.000 & 1.000 & 1.000 & 1.000 & 1.000 & 1.000 & 1.000 & 1.000\\
-1.00 & 1000 & 1.000 & 1.000 & 1.000 & 1.000 & 1.000 & 1.000 & 1.000 & 1.000\\
-1.00 & 2000 & 1.000 & 1.000 & 1.000 & 1.000 & 1.000 & 1.000 & 1.000 & 1.000\\
-1.00 & 5000 & 1.000 & 1.000 & 1.000 & 1.000 & 1.000 & 1.000 & 1.000 & 1.000\\
\hline
\hline
-0.025 & 100 & 0.070 & 0.086 & 0.078 & 0.078 & 0.056 & 0.060 & 1.000 & 1.000\\
-0.025 & 200 & 0.056 & 0.064 & 0.048 & 0.054 & 0.044 & 0.052 & 0.186 & 0.204\\
-0.025 & 500 & 0.108 & 0.118 & 0.122 & 0.120 & 0.054 & 0.072 & 0.136 & 0.118\\
-0.025 & 1000 & 0.132 & 0.136 & 0.136 & 0.138 & 0.076 & 0.080 & 0.112 & 0.130\\
-0.025 & 2000 & 0.160 & 0.164 & 0.162 & 0.164 & 0.100 & 0.114 & 0.118 & 0.130\\
-0.025 & 5000 & 0.434 & 0.426 & 0.434 & 0.430 & 0.140 & 0.284 & 0.244 & 0.264\\
\hline
-0.050 & 100 & 0.056 & 0.066 & 0.056 & 0.070 & 0.060 & 0.044 & 1.000 & 1.000\\
-0.050 & 200 & 0.098 & 0.104 & 0.106 & 0.108 & 0.076 & 0.072 & 0.196 & 0.254\\
-0.050 & 500 & 0.196 & 0.216 & 0.206 & 0.204 & 0.094 & 0.160 & 0.196 & 0.188\\
-0.050 & 1000 & 0.364 & 0.350 & 0.370 & 0.372 & 0.168 & 0.244 & 0.220 & 0.264\\
-0.050 & 2000 & 0.606 & 0.598 & 0.610 & 0.598 & 0.276 & 0.418 & 0.362 & 0.422\\
-0.050 & 5000 & 0.956 & 0.952 & 0.954 & 0.950 & 0.608 & 0.818 & 0.714 & 0.780\\
\hline
-0.075 & 100 & 0.160 & 0.160 & 0.140 & 0.152 & 0.076 & 0.076 & 1.000 & 1.000\\
-0.075 & 200 & 0.192 & 0.210 & 0.188 & 0.200 & 0.086 & 0.116 & 0.294 & 0.274\\
-0.075 & 500 & 0.360 & 0.336 & 0.356 & 0.362 & 0.138 & 0.240 & 0.240 & 0.272\\
-0.075 & 1000 & 0.652 & 0.632 & 0.642 & 0.640 & 0.278 & 0.438 & 0.360 & 0.434\\
-0.075 & 2000 & 0.910 & 0.902 & 0.916 & 0.916 & 0.574 & 0.740 & 0.672 & 0.726\\
-0.075 & 5000 & 1.000 & 1.000 & 1.000 & 1.000 & 0.944 & 0.976 & 0.926 & 0.968\\
\hline
-0.100 & 100 & 0.144 & 0.164 & 0.140 & 0.160 & 0.086 & 0.102 & 1.000 & 1.000\\
-0.100 & 200 & 0.312 & 0.306 & 0.288 & 0.296 & 0.122 & 0.190 & 0.336 & 0.330\\
-0.100 & 500 & 0.624 & 0.620 & 0.624 & 0.624 & 0.236 & 0.416 & 0.416 & 0.426\\
-0.100 & 1000 & 0.904 & 0.882 & 0.896 & 0.900 & 0.502 & 0.718 & 0.656 & 0.706\\
-0.100 & 2000 & 0.994 & 0.996 & 0.992 & 0.992 & 0.820 & 0.962 & 0.888 & 0.910\\
-0.100 & 5000 & 1.000 & 1.000 & 1.000 & 1.000 & 0.996 & 0.992 & 0.986 & 0.998\\
\hline
\end{tabular}    
    \caption{Empirical powers of eight tests under Model 1 based on 500 experiments, where the alternative distributions include (1) $\rho=-0.25, -0.5, -0.75, -1$ and (2) $\rho= -0.025, -0.05, -0.075, -0.1$.}
    \label{tab:ep-1a-2}
\end{table}

\begin{figure}[htbp]
    \centering
    \includegraphics[width=\linewidth,page=2]{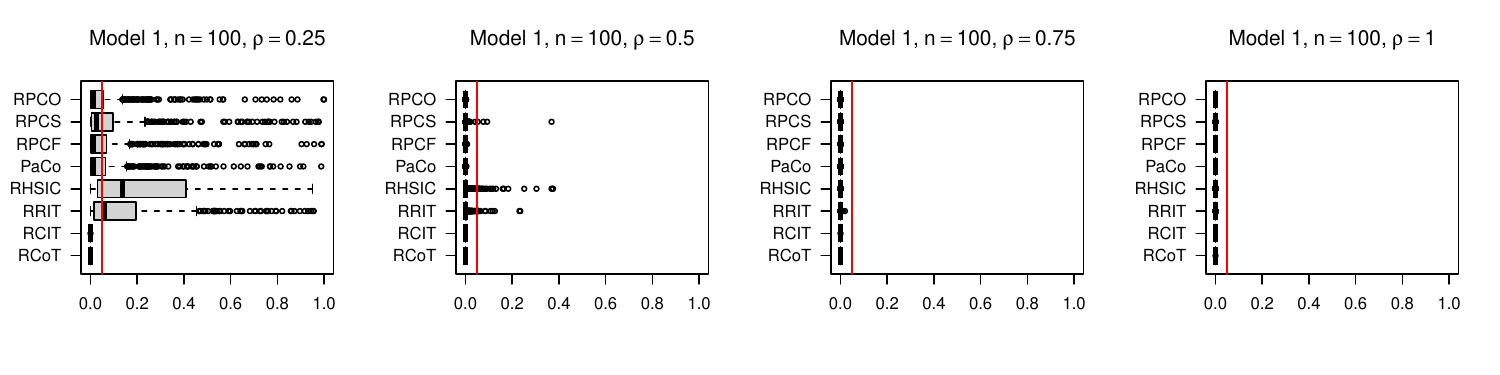}\vspace{-0.8cm}
    \includegraphics[width=\linewidth,page=4]{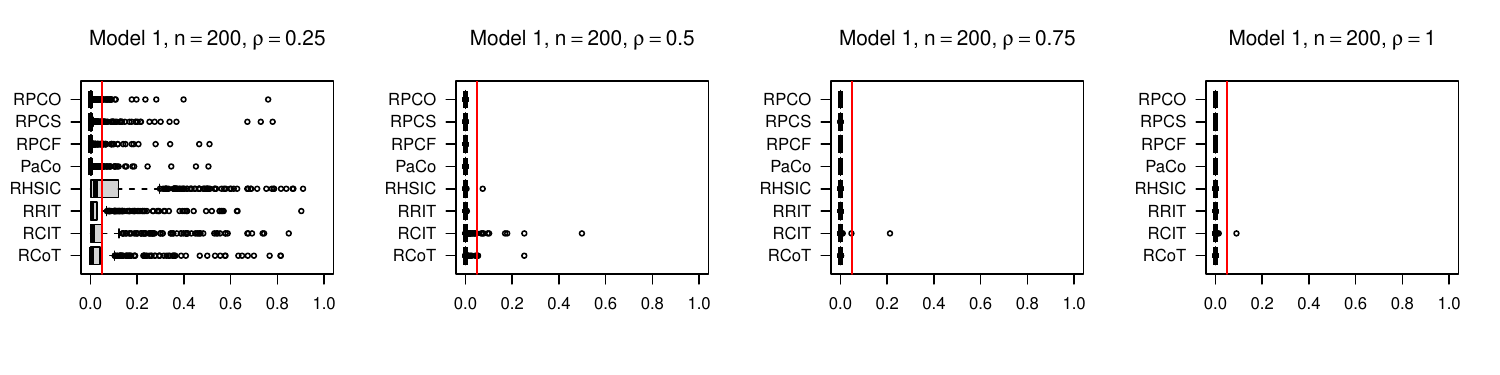}\vspace{-0.8cm}
    \includegraphics[width=\linewidth,page=6]{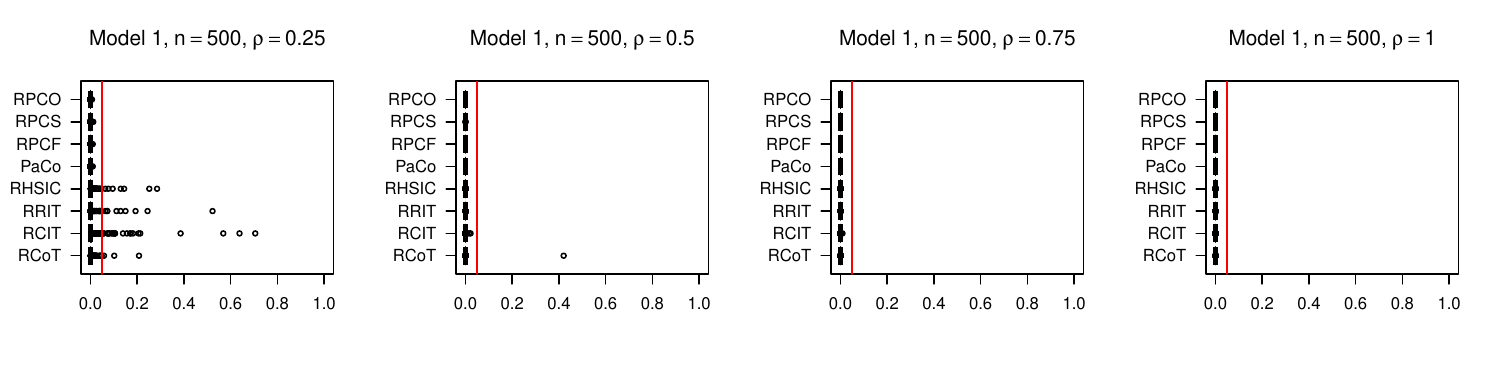}\vspace{-0.8cm}
    \includegraphics[width=\linewidth,page=8]{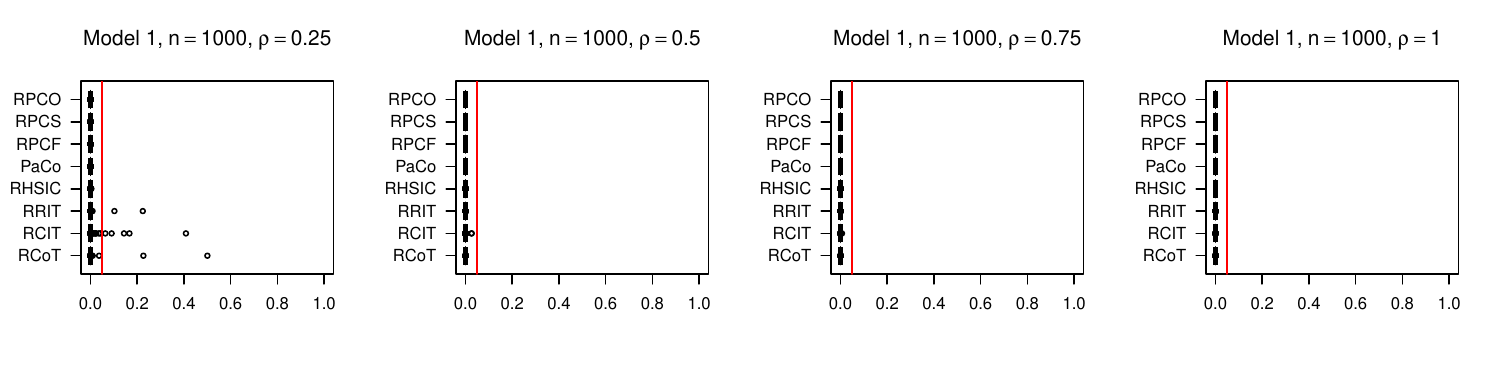}\vspace{-0.8cm}
    \includegraphics[width=\linewidth,page=10]{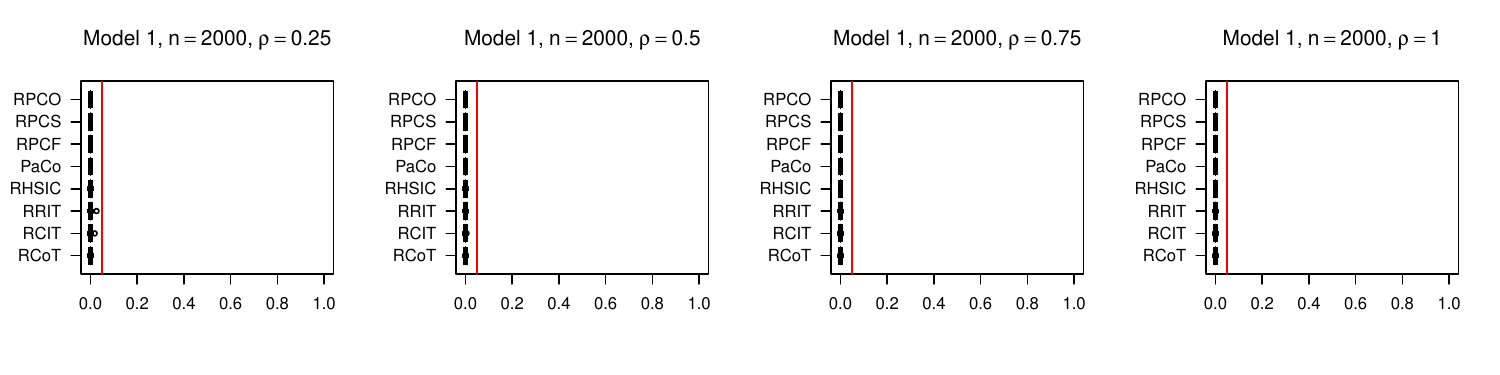}\vspace{-0.8cm}
    \includegraphics[width=\linewidth,page=12]{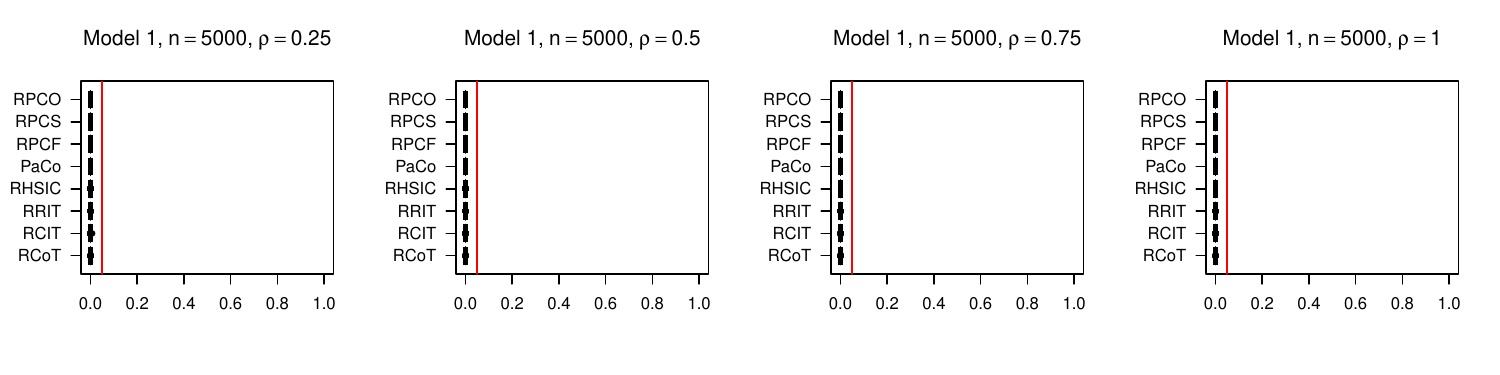}\vspace{-0.8cm}
\caption{Boxplots of p-values of eight tests for Model 1 under the alternative hypothesis, with sample sizes $n=100,200,500,1000,2000,5000$, where the alternative distributions include $\rho= 0.25,  0.5,  0.75,  1$. The red line represents 0.05.}
    \label{fig:boxplot-1a-1}
\end{figure}

\begin{figure}[htbp]
    \centering
    \includegraphics[width=\linewidth,page=2]{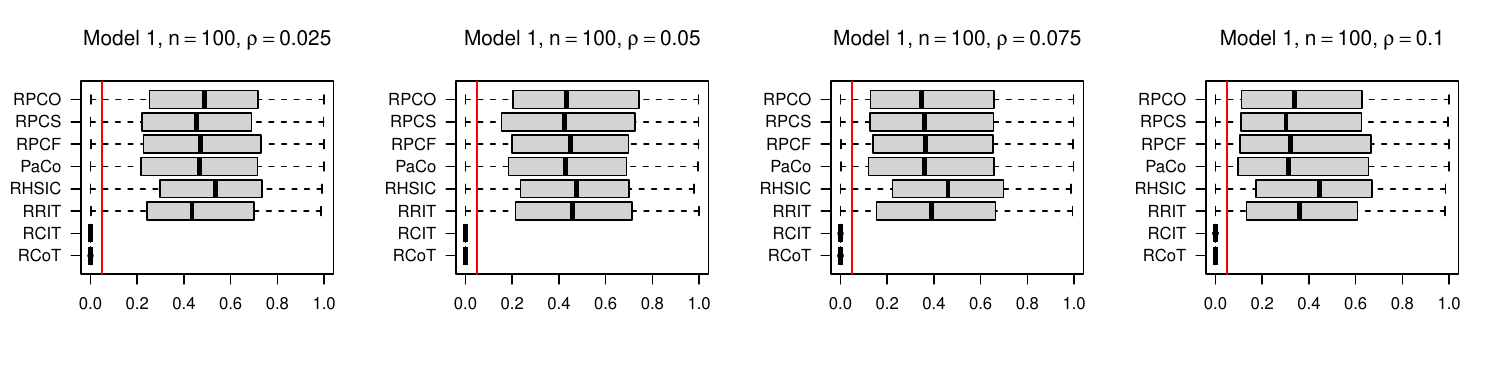}\vspace{-0.8cm}
    \includegraphics[width=\linewidth,page=4]{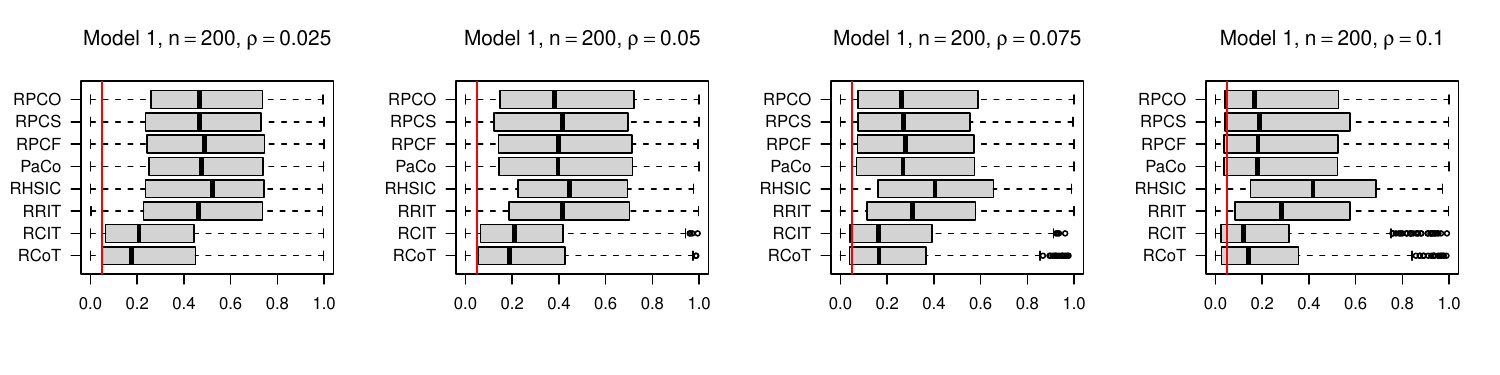}\vspace{-0.8cm}
    \includegraphics[width=\linewidth,page=6]{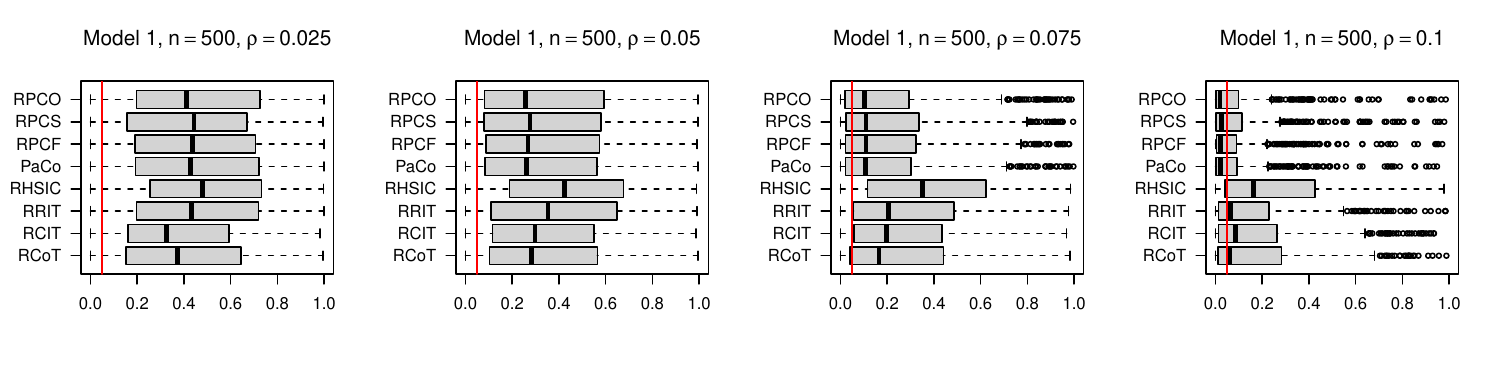}\vspace{-0.8cm}
    \includegraphics[width=\linewidth,page=8]{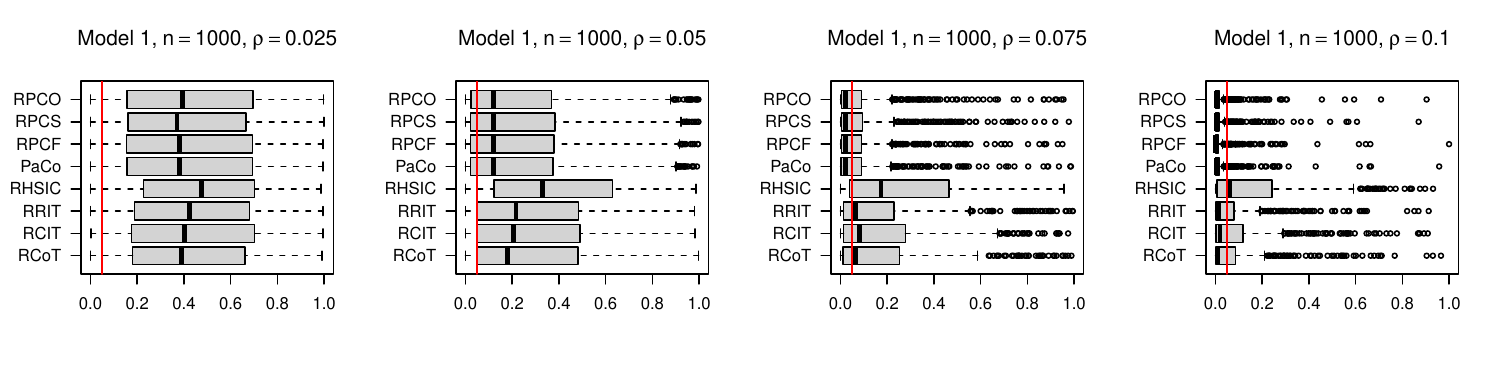}\vspace{-0.8cm}
    \includegraphics[width=\linewidth,page=10]{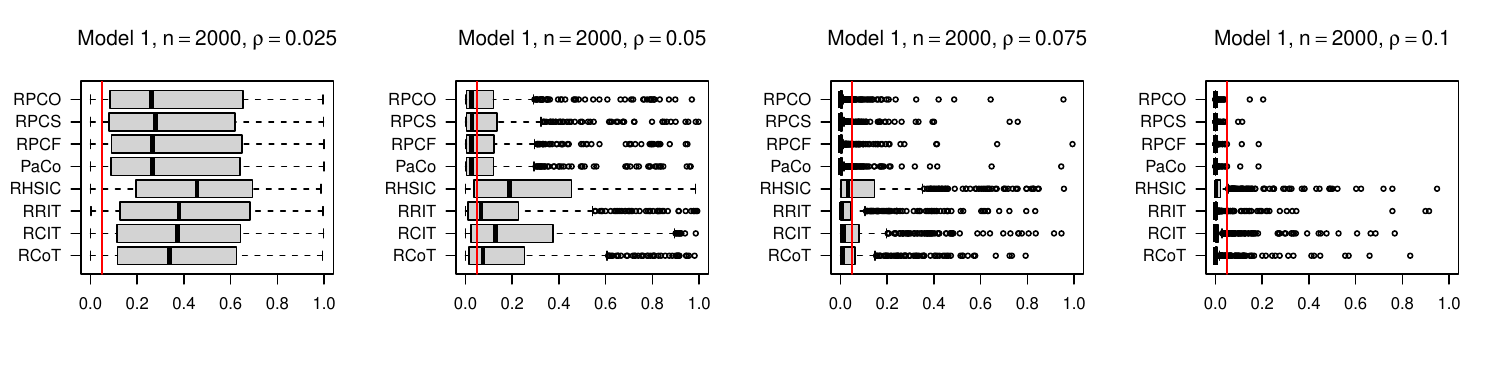}\vspace{-0.8cm}
    \includegraphics[width=\linewidth,page=12]{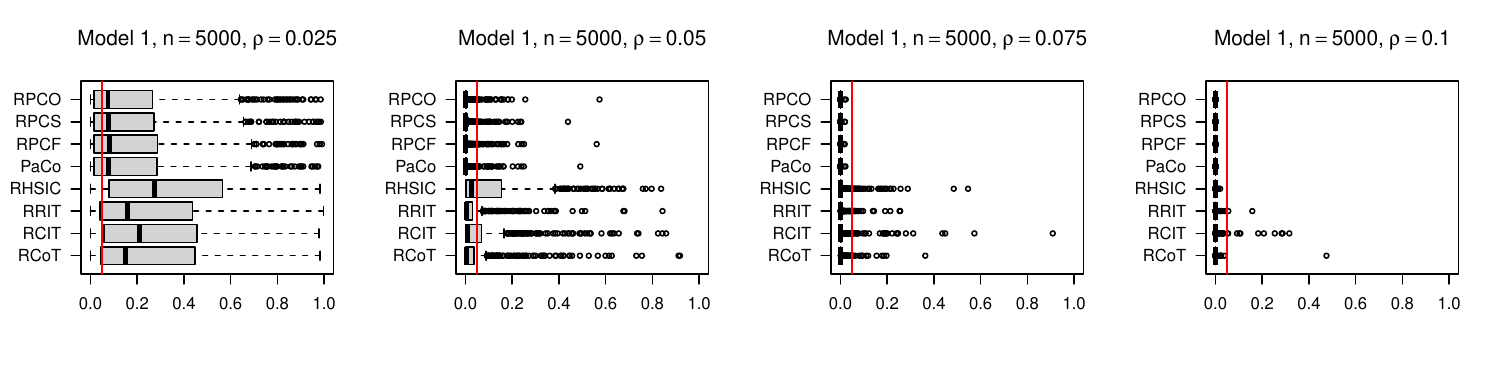}\vspace{-0.8cm}
\caption{Boxplots of p-values of eight tests for Model 1 under the alternative hypothesis, with sample sizes $n=100,200,500,1000,2000,5000$, where the alternative distributions include $\rho= 0.025,  0.05,  0.075,  0.1$. The red line represents 0.05.}
    \label{fig:boxplot-1a-2}
\end{figure}

\begin{figure}[htbp]
    \centering
    \includegraphics[width=\linewidth,page=1]{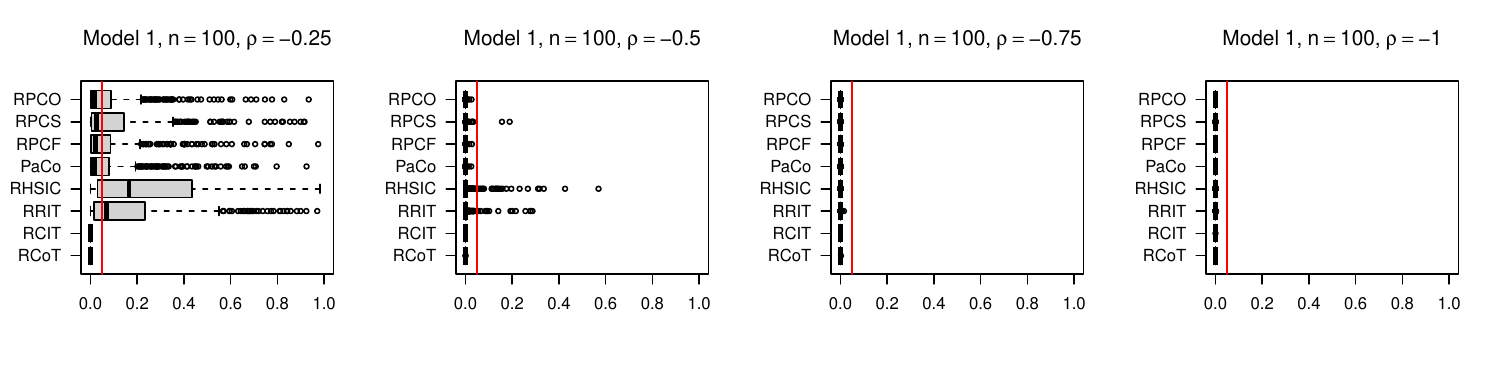}\vspace{-0.8cm}
    \includegraphics[width=\linewidth,page=3]{fig/model1-p1.pdf}\vspace{-0.8cm}
    \includegraphics[width=\linewidth,page=5]{fig/model1-p1.pdf}\vspace{-0.8cm}
    \includegraphics[width=\linewidth,page=7]{fig/model1-p1.pdf}\vspace{-0.8cm}
    \includegraphics[width=\linewidth,page=9]{fig/model1-p1.pdf}\vspace{-0.8cm}
    \includegraphics[width=\linewidth,page=11]{fig/model1-p1.pdf}\vspace{-0.8cm}
\caption{Boxplots of p-values of eight tests for Model 1 under the alternative hypothesis, with sample sizes $n=100,200,500,1000,2000,5000$, where the alternative distributions include $\rho= -0.25,  -0.5,  -0.75,  -1$. The red line represents 0.05.}
    \label{fig:boxplot-1a-3}
\end{figure}

\begin{figure}[htbp]
    \centering
    \includegraphics[width=\linewidth,page=1]{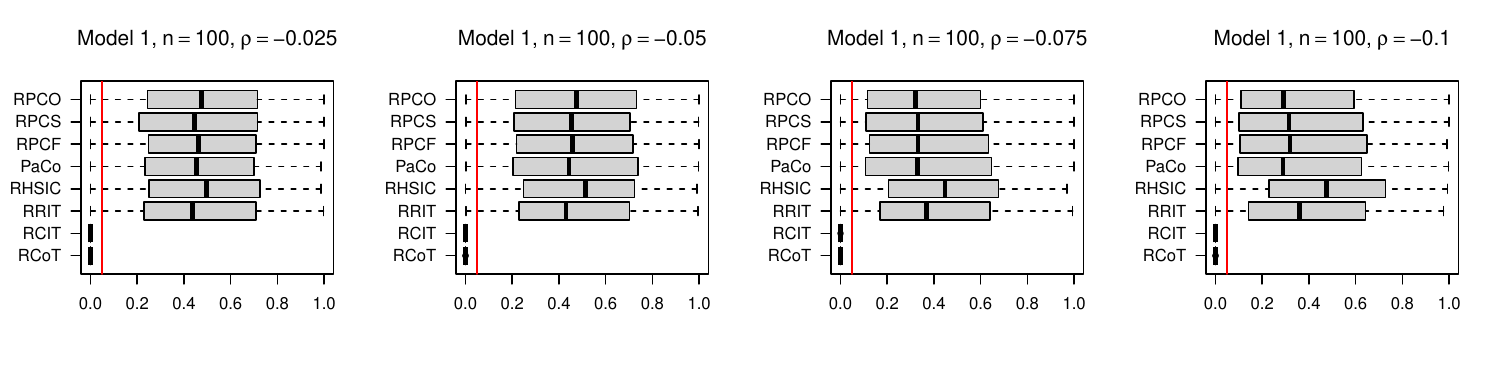}\vspace{-0.8cm}
    \includegraphics[width=\linewidth,page=3]{fig/model1-p2.pdf}\vspace{-0.8cm}
    \includegraphics[width=\linewidth,page=5]{fig/model1-p2.pdf}\vspace{-0.8cm}
    \includegraphics[width=\linewidth,page=7]{fig/model1-p2.pdf}\vspace{-0.8cm}
    \includegraphics[width=\linewidth,page=9]{fig/model1-p2.pdf}\vspace{-0.8cm}
    \includegraphics[width=\linewidth,page=11]{fig/model1-p2.pdf}\vspace{-0.8cm}
\caption{Boxplots of p-values of eight tests for Model 1 under the alternative hypothesis, with sample sizes $n=100,200,500,1000,2000,5000$, where the alternative distributions include $\rho= -0.025,  -0.05,  -0.075,  -0.1$. The red line represents 0.05.}
    \label{fig:boxplot-1a-4}
\end{figure}

\begin{table}[htbp]
    \centering
    \footnotesize
\begin{tabular}{rr|rrrrrrrr}
\hline
$\rho$ & $n$ & RPCO & RPCS & RPCF & PaCo & RHSIC & RRIT & RCIT & RCoT\\
\hline
\hline
-0.25 & 100 & 0.722 & 0.556 & 0.630 & 0.708 & 0.282 & 0.444 & 1.000 & 1.000\\
-0.25 & 200 & 0.922 & 0.860 & 0.902 & 0.916 & 0.576 & 0.742 & 0.716 & 0.776\\
-0.25 & 500 & 1.000 & 1.000 & 1.000 & 1.000 & 0.960 & 0.980 & 0.886 & 0.980\\
-0.25 & 1000 & 1.000 & 1.000 & 1.000 & 1.000 & 1.000 & 1.000 & 0.976 & 1.000\\
-0.25 & 2000 & 1.000 & 1.000 & 1.000 & 1.000 & 1.000 & 1.000 & 0.990 & 1.000\\
-0.25 & 5000 & 1.000 & 1.000 & 1.000 & 1.000 & 1.000 & 0.998 & 0.998 & 1.000\\
0.25 & 100 & 0.738 & 0.690 & 0.726 & 0.720 & 0.368 & 0.484 & 1.000 & 1.000\\
0.25 & 200 & 0.934 & 0.906 & 0.940 & 0.936 & 0.606 & 0.784 & 0.704 & 0.778\\
0.25 & 500 & 1.000 & 1.000 & 1.000 & 1.000 & 0.970 & 0.984 & 0.924 & 0.994\\
0.25 & 1000 & 1.000 & 1.000 & 1.000 & 1.000 & 1.000 & 1.000 & 0.972 & 0.998\\
0.25 & 2000 & 1.000 & 1.000 & 1.000 & 1.000 & 1.000 & 1.000 & 0.986 & 1.000\\
0.25 & 5000 & 1.000 & 1.000 & 1.000 & 1.000 & 1.000 & 1.000 & 1.000 & 1.000\\
\hline
-0.50 & 100 & 1.000 & 0.994 & 1.000 & 1.000 & 0.946 & 0.954 & 1.000 & 1.000\\
-0.50 & 200 & 1.000 & 1.000 & 1.000 & 1.000 & 1.000 & 1.000 & 0.974 & 1.000\\
-0.50 & 500 & 1.000 & 1.000 & 1.000 & 1.000 & 1.000 & 1.000 & 0.992 & 1.000\\
-0.50 & 1000 & 1.000 & 1.000 & 1.000 & 1.000 & 1.000 & 1.000 & 0.998 & 1.000\\
-0.50 & 2000 & 1.000 & 1.000 & 1.000 & 1.000 & 1.000 & 1.000 & 1.000 & 1.000\\
-0.50 & 5000 & 1.000 & 1.000 & 1.000 & 1.000 & 1.000 & 1.000 & 1.000 & 1.000\\
0.50 & 100 & 1.000 & 0.994 & 1.000 & 1.000 & 0.966 & 0.972 & 1.000 & 1.000\\
0.50 & 200 & 1.000 & 1.000 & 1.000 & 1.000 & 1.000 & 1.000 & 0.970 & 1.000\\
0.50 & 500 & 1.000 & 1.000 & 1.000 & 1.000 & 1.000 & 1.000 & 0.998 & 1.000\\
0.50 & 1000 & 1.000 & 1.000 & 1.000 & 1.000 & 1.000 & 1.000 & 1.000 & 1.000\\
0.50 & 2000 & 1.000 & 1.000 & 1.000 & 1.000 & 1.000 & 1.000 & 1.000 & 1.000\\
0.50 & 5000 & 1.000 & 1.000 & 1.000 & 1.000 & 1.000 & 1.000 & 1.000 & 1.000\\
\hline
-0.75 & 100 & 1.000 & 1.000 & 1.000 & 1.000 & 1.000 & 0.998 & 1.000 & 1.000\\
-0.75 & 200 & 1.000 & 1.000 & 1.000 & 1.000 & 1.000 & 1.000 & 0.992 & 1.000\\
-0.75 & 500 & 1.000 & 1.000 & 1.000 & 1.000 & 1.000 & 1.000 & 0.996 & 1.000\\
-0.75 & 1000 & 1.000 & 1.000 & 1.000 & 1.000 & 1.000 & 1.000 & 1.000 & 1.000\\
-0.75 & 2000 & 1.000 & 1.000 & 1.000 & 1.000 & 1.000 & 1.000 & 1.000 & 1.000\\
-0.75 & 5000 & 1.000 & 1.000 & 1.000 & 1.000 & 1.000 & 1.000 & 1.000 & 1.000\\
0.75 & 100 & 1.000 & 1.000 & 1.000 & 1.000 & 1.000 & 0.998 & 1.000 & 1.000\\
0.75 & 200 & 1.000 & 1.000 & 1.000 & 1.000 & 1.000 & 1.000 & 0.990 & 1.000\\
0.75 & 500 & 1.000 & 1.000 & 1.000 & 1.000 & 1.000 & 1.000 & 0.998 & 1.000\\
0.75 & 1000 & 1.000 & 1.000 & 1.000 & 1.000 & 1.000 & 1.000 & 0.998 & 1.000\\
0.75 & 2000 & 1.000 & 1.000 & 1.000 & 1.000 & 1.000 & 1.000 & 1.000 & 1.000\\
0.75 & 5000 & 1.000 & 1.000 & 1.000 & 1.000 & 1.000 & 1.000 & 1.000 & 1.000\\
\hline
-1.00 & 100 & 1.000 & 1.000 & 1.000 & 1.000 & 1.000 & 1.000 & 1.000 & 1.000\\
-1.00 & 200 & 1.000 & 1.000 & 1.000 & 1.000 & 1.000 & 1.000 & 0.998 & 1.000\\
-1.00 & 500 & 1.000 & 1.000 & 1.000 & 1.000 & 1.000 & 1.000 & 1.000 & 1.000\\
-1.00 & 1000 & 1.000 & 1.000 & 1.000 & 1.000 & 1.000 & 1.000 & 1.000 & 1.000\\
-1.00 & 2000 & 1.000 & 1.000 & 1.000 & 1.000 & 1.000 & 1.000 & 1.000 & 1.000\\
-1.00 & 5000 & 1.000 & 1.000 & 1.000 & 1.000 & 1.000 & 1.000 & 1.000 & 1.000\\
1.00 & 100 & 1.000 & 1.000 & 1.000 & 1.000 & 1.000 & 1.000 & 1.000 & 1.000\\
1.00 & 200 & 1.000 & 1.000 & 1.000 & 1.000 & 1.000 & 1.000 & 1.000 & 1.000\\
1.00 & 500 & 1.000 & 1.000 & 1.000 & 1.000 & 1.000 & 1.000 & 1.000 & 1.000\\
1.00 & 1000 & 1.000 & 1.000 & 1.000 & 1.000 & 1.000 & 1.000 & 1.000 & 1.000\\
1.00 & 2000 & 1.000 & 1.000 & 1.000 & 1.000 & 1.000 & 1.000 & 1.000 & 1.000\\
1.00 & 5000 & 1.000 & 1.000 & 1.000 & 1.000 & 1.000 & 1.000 & 1.000 & 1.000\\
\hline
\end{tabular}
    \caption{Empirical powers of eight tests under Model 2 based on 500 experiments, where the alternative distributions include $\rho=\pm 0.25, \pm 0.5, \pm 0.75, \pm 1$.}
    \label{tab:ep-2a-1}
\end{table}

\begin{table}[htbp]
    \centering
    \footnotesize
\begin{tabular}{rr|rrrrrrrr}
\hline
$\rho$ & $n$ & RPCO & RPCS & RPCF & PaCo & RHSIC & RRIT & RCIT & RCoT\\
\hline
\hline
-0.025 & 100 & 0.056 & 0.064 & 0.048 & 0.062 & 0.048 & 0.052 & 1.000 & 1.000\\
-0.025 & 200 & 0.048 & 0.056 & 0.044 & 0.058 & 0.064 & 0.060 & 0.290 & 0.246\\
-0.025 & 500 & 0.088 & 0.082 & 0.080 & 0.088 & 0.052 & 0.078 & 0.106 & 0.118\\
-0.025 & 1000 & 0.126 & 0.116 & 0.124 & 0.126 & 0.056 & 0.082 & 0.112 & 0.106\\
-0.025 & 2000 & 0.186 & 0.178 & 0.186 & 0.198 & 0.090 & 0.124 & 0.126 & 0.166\\
-0.025 & 5000 & 0.396 & 0.406 & 0.402 & 0.400 & 0.180 & 0.296 & 0.228 & 0.262\\
0.025 & 100 & 0.050 & 0.058 & 0.056 & 0.058 & 0.072 & 0.054 & 1.000 & 1.000\\
0.025 & 200 & 0.060 & 0.082 & 0.058 & 0.058 & 0.056 & 0.038 & 0.300 & 0.254\\
0.025 & 500 & 0.090 & 0.094 & 0.094 & 0.096 & 0.062 & 0.058 & 0.124 & 0.108\\
0.025 & 1000 & 0.102 & 0.118 & 0.098 & 0.100 & 0.068 & 0.094 & 0.110 & 0.110\\
0.025 & 2000 & 0.196 & 0.212 & 0.194 & 0.198 & 0.088 & 0.126 & 0.110 & 0.148\\
0.025 & 5000 & 0.406 & 0.408 & 0.400 & 0.402 & 0.148 & 0.282 & 0.184 & 0.248\\
\hline
-0.050 & 100 & 0.090 & 0.086 & 0.088 & 0.112 & 0.058 & 0.066 & 1.000 & 1.000\\
-0.050 & 200 & 0.112 & 0.098 & 0.086 & 0.110 & 0.082 & 0.076 & 0.302 & 0.242\\
-0.050 & 500 & 0.186 & 0.184 & 0.182 & 0.194 & 0.092 & 0.116 & 0.172 & 0.176\\
-0.050 & 1000 & 0.332 & 0.314 & 0.336 & 0.344 & 0.128 & 0.208 & 0.222 & 0.238\\
-0.050 & 2000 & 0.592 & 0.562 & 0.580 & 0.588 & 0.232 & 0.402 & 0.266 & 0.374\\
-0.050 & 5000 & 0.944 & 0.946 & 0.942 & 0.946 & 0.618 & 0.816 & 0.606 & 0.790\\
0.050 & 100 & 0.076 & 0.096 & 0.094 & 0.102 & 0.060 & 0.072 & 1.000 & 1.000\\
0.050 & 200 & 0.112 & 0.144 & 0.130 & 0.134 & 0.074 & 0.092 & 0.344 & 0.268\\
0.050 & 500 & 0.220 & 0.238 & 0.238 & 0.230 & 0.114 & 0.144 & 0.170 & 0.182\\
0.050 & 1000 & 0.370 & 0.400 & 0.384 & 0.384 & 0.162 & 0.252 & 0.216 & 0.268\\
0.050 & 2000 & 0.636 & 0.644 & 0.638 & 0.634 & 0.274 & 0.472 & 0.312 & 0.450\\
0.050 & 5000 & 0.972 & 0.974 & 0.968 & 0.970 & 0.600 & 0.810 & 0.644 & 0.794\\
\hline
-0.075 & 100 & 0.124 & 0.104 & 0.090 & 0.122 & 0.064 & 0.074 & 1.000 & 1.000\\
-0.075 & 200 & 0.166 & 0.128 & 0.142 & 0.176 & 0.078 & 0.106 & 0.352 & 0.268\\
-0.075 & 500 & 0.398 & 0.338 & 0.364 & 0.386 & 0.158 & 0.280 & 0.248 & 0.326\\
-0.075 & 1000 & 0.668 & 0.626 & 0.642 & 0.664 & 0.298 & 0.446 & 0.366 & 0.490\\
-0.075 & 2000 & 0.942 & 0.926 & 0.940 & 0.942 & 0.528 & 0.744 & 0.590 & 0.730\\
-0.075 & 5000 & 1.000 & 1.000 & 1.000 & 1.000 & 0.950 & 0.980 & 0.906 & 0.972\\
0.075 & 100 & 0.108 & 0.140 & 0.126 & 0.124 & 0.086 & 0.100 & 1.000 & 1.000\\
0.075 & 200 & 0.190 & 0.194 & 0.204 & 0.186 & 0.100 & 0.106 & 0.370 & 0.308\\
0.075 & 500 & 0.392 & 0.392 & 0.404 & 0.392 & 0.160 & 0.256 & 0.260 & 0.304\\
0.075 & 1000 & 0.636 & 0.634 & 0.640 & 0.630 & 0.304 & 0.448 & 0.338 & 0.446\\
0.075 & 2000 & 0.926 & 0.920 & 0.924 & 0.918 & 0.528 & 0.752 & 0.606 & 0.758\\
0.075 & 5000 & 0.998 & 0.998 & 0.998 & 0.998 & 0.924 & 0.962 & 0.876 & 0.972\\
\hline
-0.100 & 100 & 0.178 & 0.122 & 0.144 & 0.190 & 0.078 & 0.076 & 1.000 & 1.000\\
-0.100 & 200 & 0.280 & 0.212 & 0.252 & 0.282 & 0.106 & 0.170 & 0.396 & 0.370\\
-0.100 & 500 & 0.614 & 0.566 & 0.578 & 0.612 & 0.216 & 0.412 & 0.356 & 0.434\\
-0.100 & 1000 & 0.886 & 0.852 & 0.880 & 0.888 & 0.506 & 0.700 & 0.508 & 0.668\\
-0.100 & 2000 & 0.998 & 0.994 & 0.998 & 0.998 & 0.830 & 0.924 & 0.820 & 0.924\\
-0.100 & 5000 & 1.000 & 1.000 & 1.000 & 1.000 & 1.000 & 1.000 & 0.950 & 0.998\\
0.100 & 100 & 0.166 & 0.190 & 0.182 & 0.178 & 0.084 & 0.132 & 1.000 & 1.000\\
0.100 & 200 & 0.314 & 0.318 & 0.312 & 0.306 & 0.156 & 0.204 & 0.398 & 0.386\\
0.100 & 500 & 0.608 & 0.610 & 0.630 & 0.618 & 0.242 & 0.430 & 0.324 & 0.448\\
0.100 & 1000 & 0.900 & 0.894 & 0.894 & 0.890 & 0.486 & 0.736 & 0.508 & 0.694\\
0.100 & 2000 & 0.992 & 0.994 & 0.994 & 0.992 & 0.810 & 0.932 & 0.780 & 0.912\\
0.100 & 5000 & 1.000 & 1.000 & 1.000 & 1.000 & 1.000 & 0.998 & 0.946 & 1.000\\
\hline
\end{tabular}
    \caption{Empirical powers of eight tests under Model 2 based on 500 experiments, where the alternative distributions include $\rho=\pm 0.025, \pm 0.05, \pm 0.075, \pm 0.1$.}
    \label{tab:ep-2a-2}
\end{table}

\begin{figure}[htbp]
    \centering
    \includegraphics[width=\linewidth,page=2]{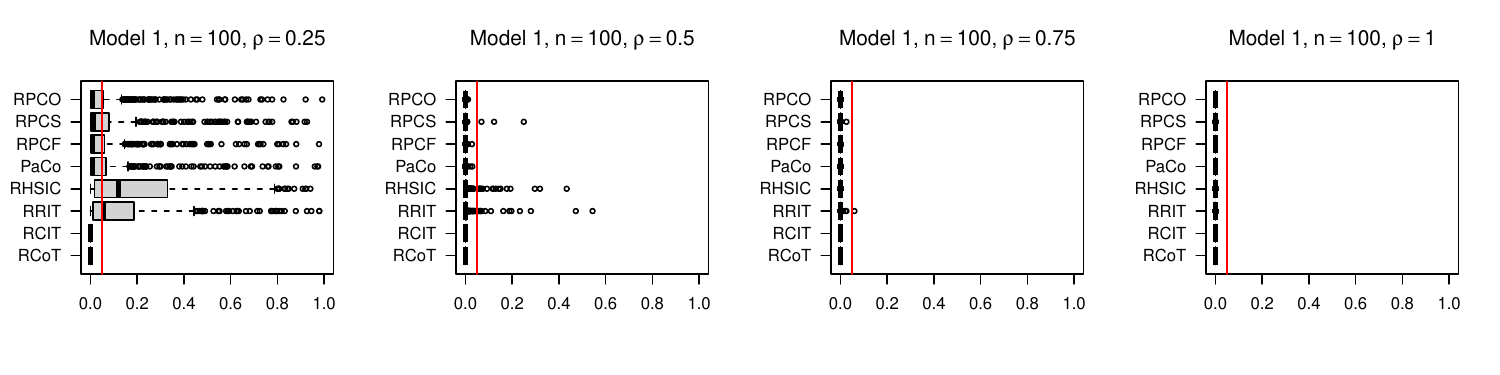}\vspace{-0.8cm}
    \includegraphics[width=\linewidth,page=4]{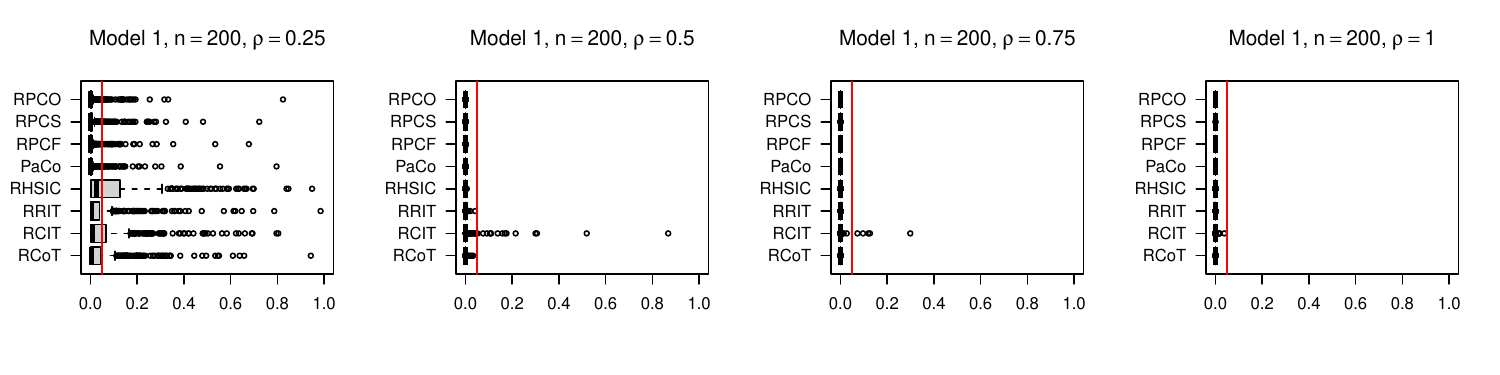}\vspace{-0.8cm}
    \includegraphics[width=\linewidth,page=6]{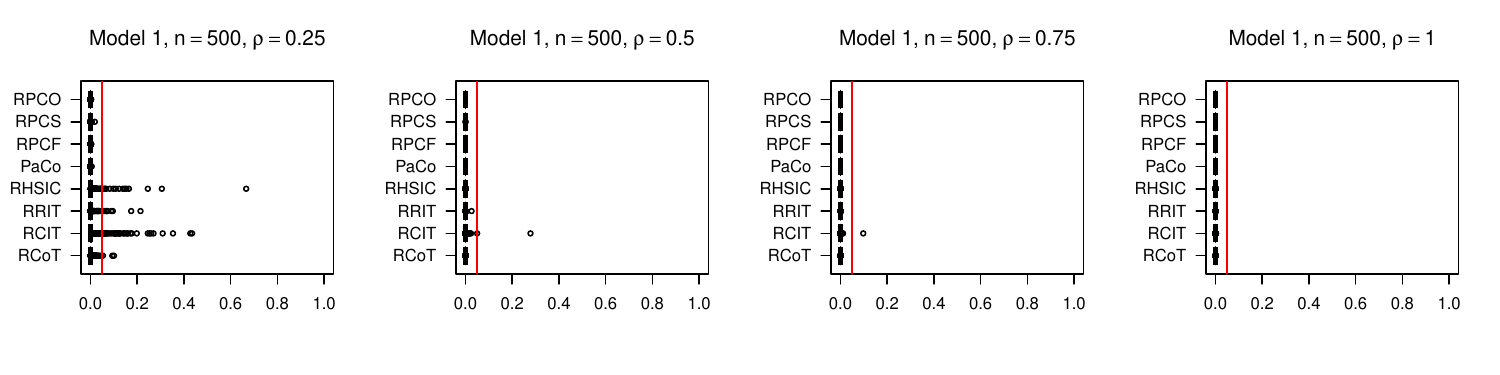}\vspace{-0.8cm}
    \includegraphics[width=\linewidth,page=8]{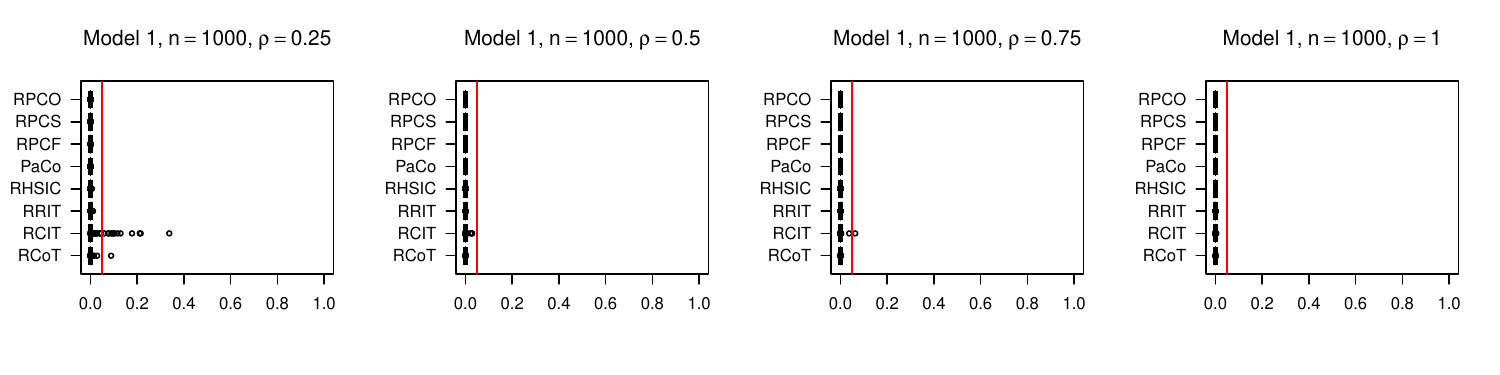}\vspace{-0.8cm}
    \includegraphics[width=\linewidth,page=10]{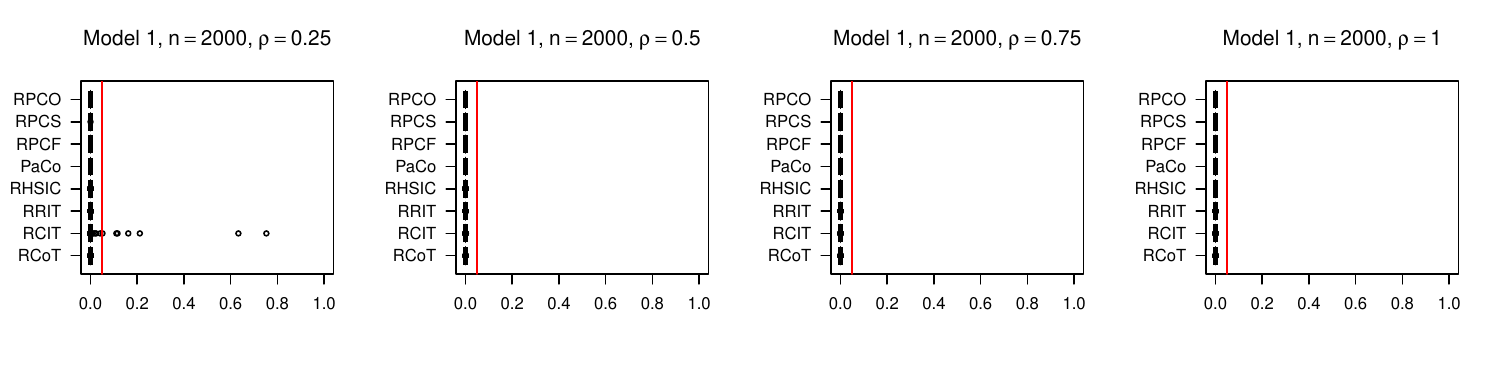}\vspace{-0.8cm}
    \includegraphics[width=\linewidth,page=12]{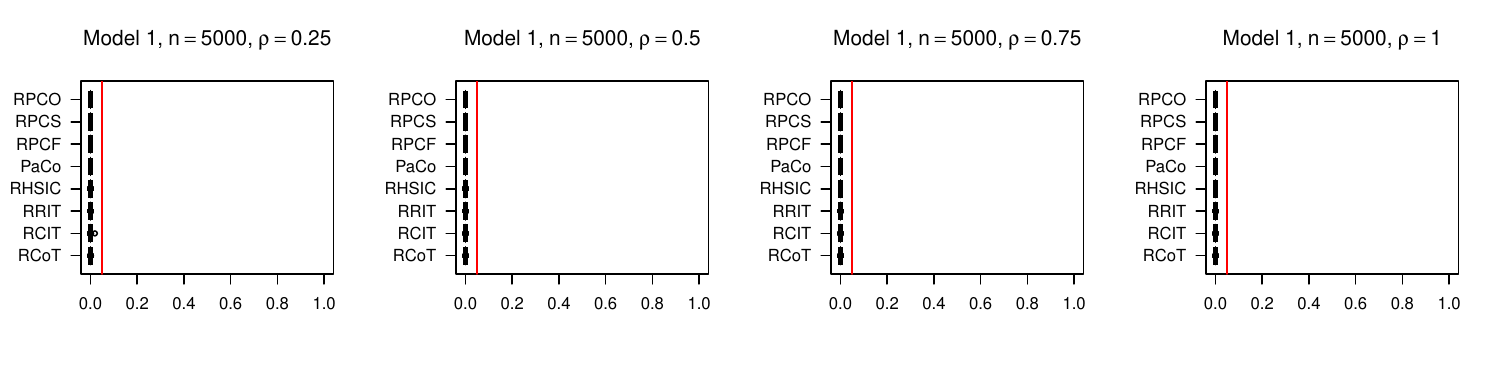}\vspace{-0.8cm}
\caption{Boxplots of p-values of eight tests for Model 2 under the alternative hypothesis, with sample sizes $n=100,200,500,1000,2000,5000$, where the alternative distributions include $\rho= 0.25,  0.5,  0.75,  1$. The red line represents 0.05.}
    \label{fig:boxplot-2a-1}
\end{figure}

\begin{figure}[htbp]
    \centering
    \includegraphics[width=\linewidth,page=2]{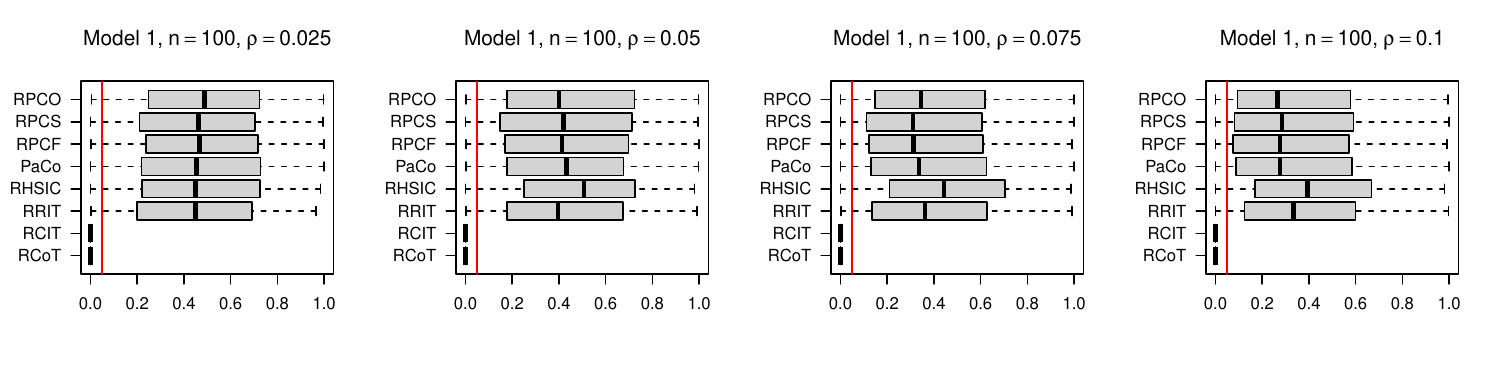}\vspace{-0.8cm}
    \includegraphics[width=\linewidth,page=4]{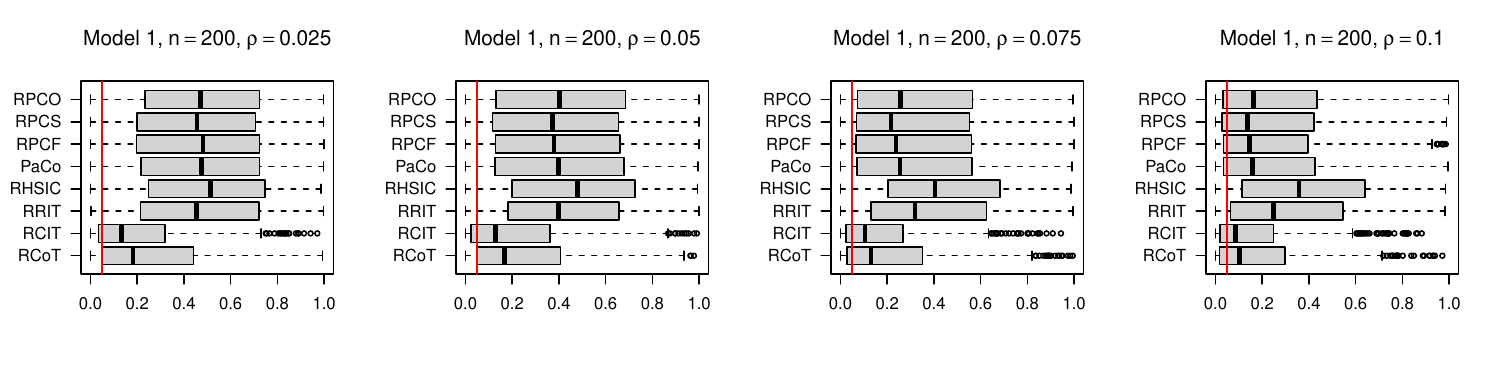}\vspace{-0.8cm}
    \includegraphics[width=\linewidth,page=6]{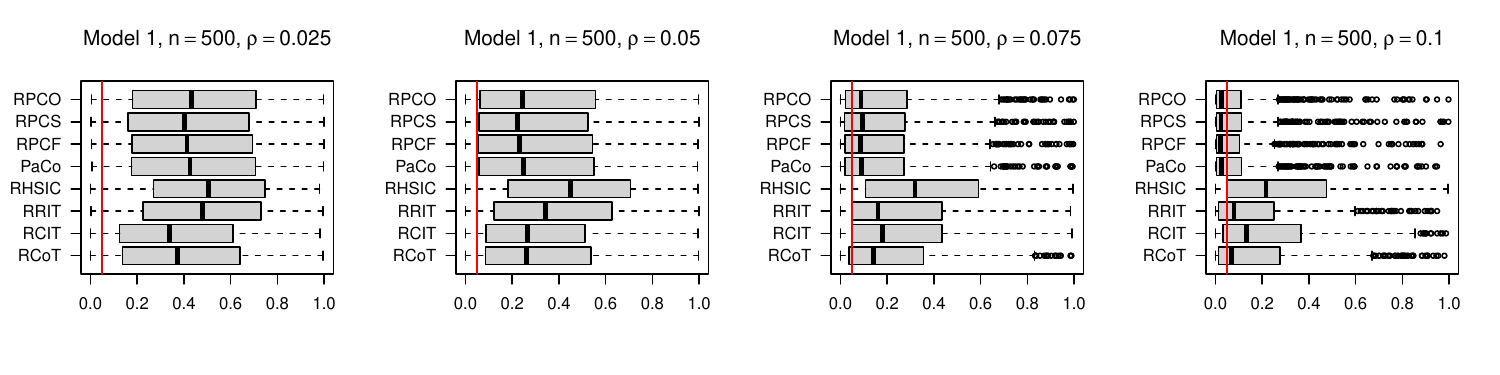}\vspace{-0.8cm}
    \includegraphics[width=\linewidth,page=8]{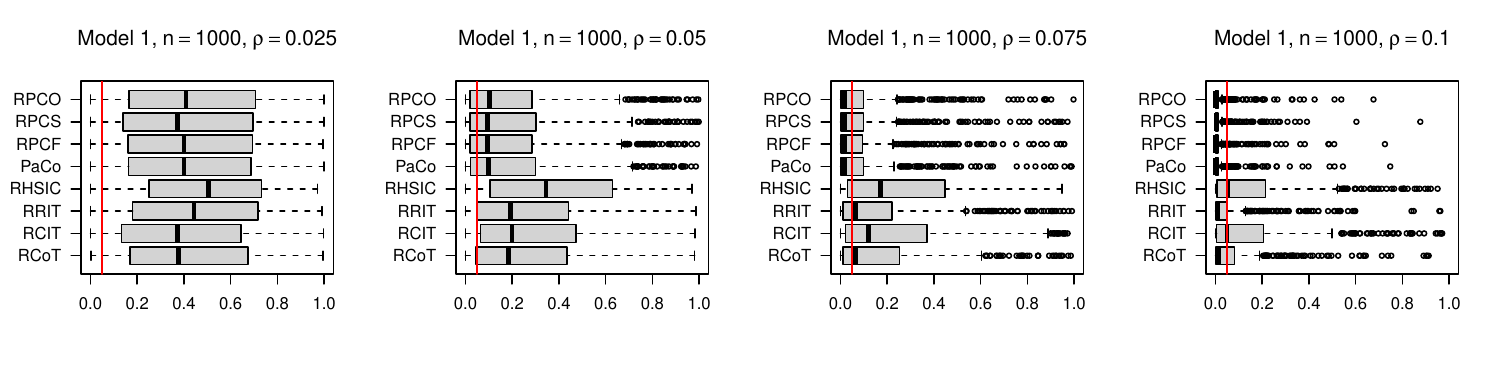}\vspace{-0.8cm}
    \includegraphics[width=\linewidth,page=10]{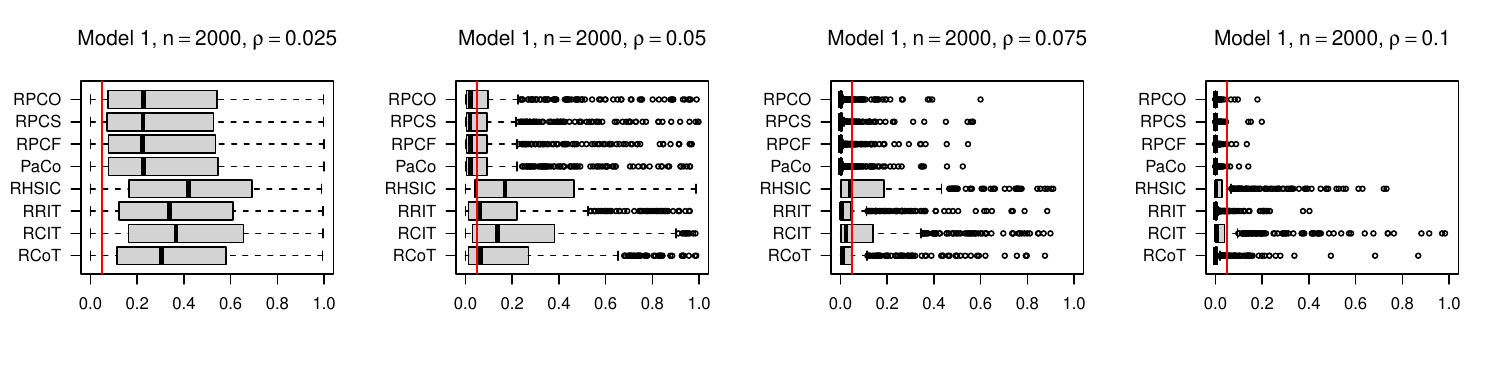}\vspace{-0.8cm}
    \includegraphics[width=\linewidth,page=12]{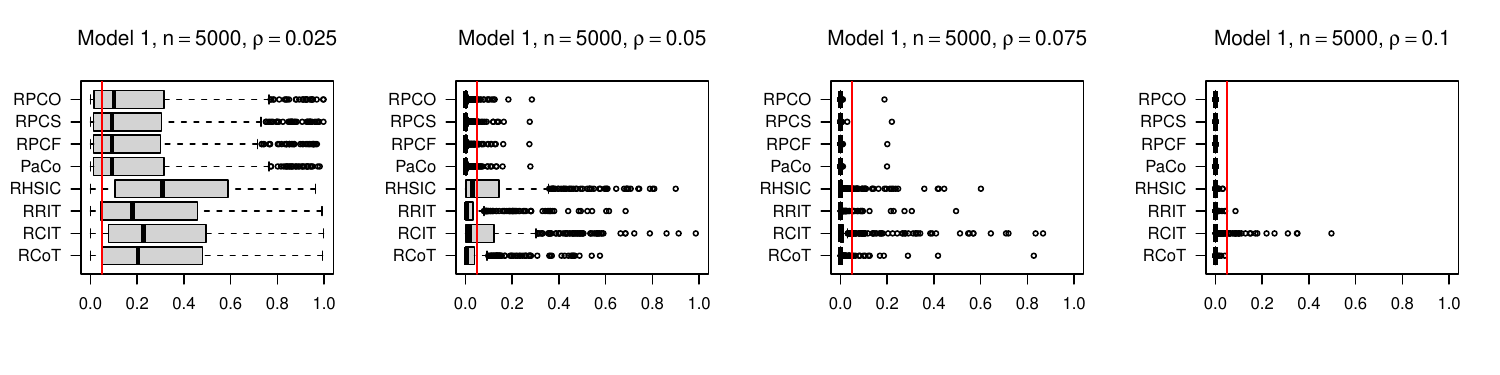}\vspace{-0.8cm}
\caption{Boxplots of p-values of eight tests for Model 2 under the alternative hypothesis, with sample sizes $n=100,200,500,1000,2000,5000$, where the alternative distributions include $\rho= 0.025,  0.05,  0.075,  0.1$. The red line represents 0.05.}
    \label{fig:boxplot-2a-2}
\end{figure}

\begin{figure}[htbp]
    \centering
    \includegraphics[width=\linewidth,page=1]{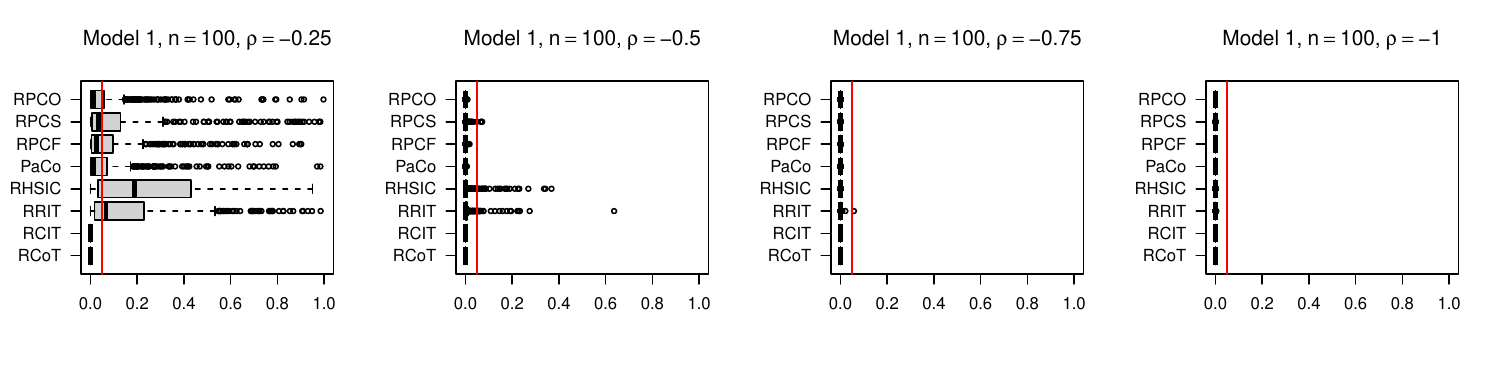}\vspace{-0.8cm}
    \includegraphics[width=\linewidth,page=3]{fig/model2-p1.pdf}\vspace{-0.8cm}
    \includegraphics[width=\linewidth,page=5]{fig/model2-p1.pdf}\vspace{-0.8cm}
    \includegraphics[width=\linewidth,page=7]{fig/model2-p1.pdf}\vspace{-0.8cm}
    \includegraphics[width=\linewidth,page=9]{fig/model2-p1.pdf}\vspace{-0.8cm}
    \includegraphics[width=\linewidth,page=11]{fig/model2-p1.pdf}\vspace{-0.8cm}
\caption{Boxplots of p-values of eight tests for Model 2 under the alternative hypothesis, with sample sizes $n=100,200,500,1000,2000,5000$, where the alternative distributions include $\rho= -0.25,  -0.5,  -0.75,  -1$. The red line represents 0.05.}
    \label{fig:boxplot-2a-3}
\end{figure}

\begin{figure}[htbp]
    \centering
    \includegraphics[width=\linewidth,page=1]{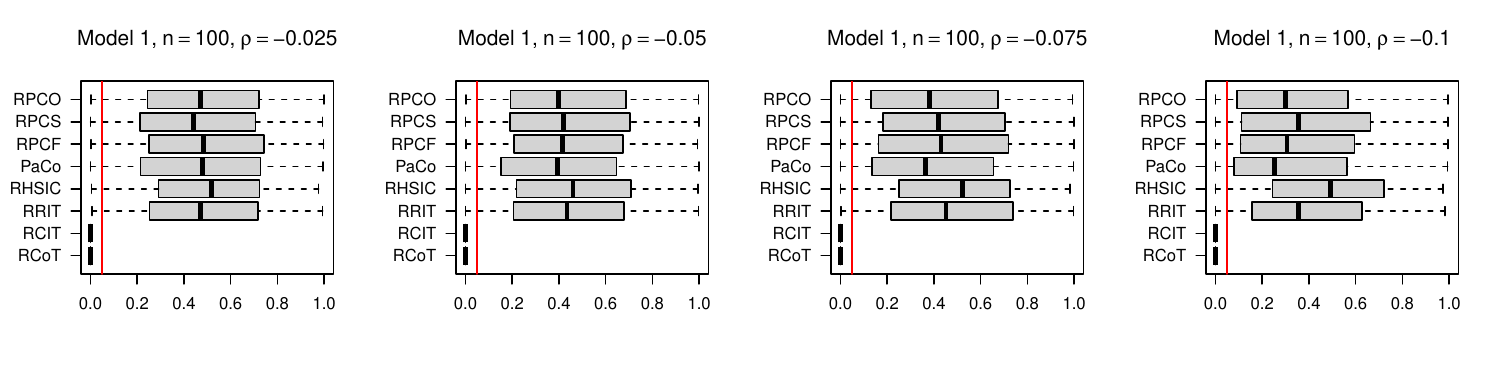}\vspace{-0.8cm}
    \includegraphics[width=\linewidth,page=3]{fig/model2-p2.pdf}\vspace{-0.8cm}
    \includegraphics[width=\linewidth,page=5]{fig/model2-p2.pdf}\vspace{-0.8cm}
    \includegraphics[width=\linewidth,page=7]{fig/model2-p2.pdf}\vspace{-0.8cm}
    \includegraphics[width=\linewidth,page=9]{fig/model2-p2.pdf}\vspace{-0.8cm}
    \includegraphics[width=\linewidth,page=11]{fig/model2-p2.pdf}\vspace{-0.8cm}
\caption{Boxplots of p-values of eight tests for Model 2 under the alternative hypothesis, with sample sizes $n=100,200,500,1000,2000,5000$, where the alternative distributions include $\rho= -0.025,  -0.05,  -0.075,  -0.1$. The red line represents 0.05.}
    \label{fig:boxplot-2a-4}
\end{figure}

\begin{table}[htbp]
    \centering
    \footnotesize
\begin{tabular}{rr|rrrrr|rrrrr}
\hline
 & & \multicolumn{5}{c|}{RPCS} & \multicolumn{5}{c}{RPCF} \\
\hline
$\rho$ & $n$ & Bias & SD & $\wh {\rm SD}$ & RMSE & 95\% cvg & Bias & SD & $\wh {\rm SD}$ & RMSE & 95\% cvg\\
\hline
\hline
-0.25 & 100 & 0.028 & 0.099 & 0.094 & 0.103 & 0.940 & 0.016 & 0.092 & 0.094 & 0.093 & 0.956\\
-0.25 & 200 & 0.008 & 0.070 & 0.066 & 0.070 & 0.930 & 0.002 & 0.065 & 0.066 & 0.065 & 0.944\\
-0.25 & 500 & 0.001 & 0.047 & 0.042 & 0.047 & 0.916 & -0.003 & 0.044 & 0.042 & 0.044 & 0.930\\
-0.25 & 1000 & 0.004 & 0.030 & 0.030 & 0.031 & 0.942 & 0.002 & 0.030 & 0.030 & 0.030 & 0.958\\
-0.25 & 2000 & 0.001 & 0.021 & 0.021 & 0.021 & 0.942 & 0.001 & 0.021 & 0.021 & 0.021 & 0.942\\
-0.25 & 5000 & 0.001 & 0.013 & 0.013 & 0.013 & 0.946 & 0.001 & 0.013 & 0.013 & 0.013 & 0.954\\
\hline
-0.50 & 100 & 0.049 & 0.089 & 0.079 & 0.101 & 0.884 & 0.024 & 0.075 & 0.077 & 0.079 & 0.954\\
-0.50 & 200 & 0.026 & 0.059 & 0.055 & 0.064 & 0.918 & 0.009 & 0.053 & 0.053 & 0.054 & 0.952\\
-0.50 & 500 & 0.010 & 0.038 & 0.034 & 0.039 & 0.932 & 0.004 & 0.032 & 0.034 & 0.032 & 0.968\\
-0.50 & 1000 & 0.006 & 0.028 & 0.024 & 0.028 & 0.928 & 0.003 & 0.025 & 0.024 & 0.025 & 0.956\\
-0.50 & 2000 & 0.003 & 0.017 & 0.017 & 0.017 & 0.944 & 0.002 & 0.016 & 0.017 & 0.016 & 0.954\\
-0.50 & 5000 & 0.002 & 0.011 & 0.011 & 0.011 & 0.944 & 0.001 & 0.011 & 0.011 & 0.011 & 0.940\\
\hline
-0.75 & 100 & 0.062 & 0.067 & 0.052 & 0.091 & 0.824 & 0.027 & 0.046 & 0.048 & 0.053 & 0.964\\
-0.75 & 200 & 0.032 & 0.044 & 0.034 & 0.054 & 0.870 & 0.012 & 0.031 & 0.032 & 0.033 & 0.960\\
-0.75 & 500 & 0.013 & 0.028 & 0.020 & 0.031 & 0.894 & 0.005 & 0.021 & 0.020 & 0.021 & 0.938\\
-0.75 & 1000 & 0.007 & 0.021 & 0.014 & 0.022 & 0.920 & 0.002 & 0.014 & 0.014 & 0.014 & 0.944\\
-0.75 & 2000 & 0.003 & 0.011 & 0.010 & 0.012 & 0.936 & 0.002 & 0.010 & 0.010 & 0.010 & 0.940\\
-0.75 & 5000 & 0.003 & 0.006 & 0.006 & 0.007 & 0.938 & 0.002 & 0.006 & 0.006 & 0.006 & 0.952\\
\bottomrule
\hline
\end{tabular}

\vspace{0.3cm}
\begin{tabular}{rr|rrrrr|rrrrr}
\hline
 & & \multicolumn{5}{c|}{RPCO} & \multicolumn{5}{c}{PaCo} \\
\hline
$\rho$ & $n$ & Bias & SD & $\wh {\rm SD}$ & RMSE & 95\% cvg & Bias & SD & $\wh {\rm SD}$ & RMSE & 95\% cvg\\
\hline
\hline
-0.25 & 100 & 0.010 & 0.092 & 0.093 & 0.092 & 0.956 & 0.010 & 0.094 & 0.093 & 0.095 & 0.954\\
-0.25 & 200 & -0.002 & 0.064 & 0.066 & 0.064 & 0.954 & -0.001 & 0.065 & 0.066 & 0.065 & 0.942\\
-0.25 & 500 & -0.004 & 0.044 & 0.042 & 0.044 & 0.938 & -0.003 & 0.044 & 0.042 & 0.044 & 0.930\\
-0.25 & 1000 & 0.001 & 0.030 & 0.030 & 0.030 & 0.948 & 0.002 & 0.030 & 0.030 & 0.030 & 0.952\\
-0.25 & 2000 & 0.000 & 0.021 & 0.021 & 0.021 & 0.946 & 0.000 & 0.021 & 0.021 & 0.021 & 0.942\\
-0.25 & 5000 & 0.000 & 0.013 & 0.013 & 0.013 & 0.952 & 0.001 & 0.013 & 0.013 & 0.013 & 0.952\\
\hline
-0.50 & 100 & 0.006 & 0.075 & 0.075 & 0.075 & 0.934 & 0.007 & 0.076 & 0.075 & 0.076 & 0.950\\
-0.50 & 200 & 0.002 & 0.053 & 0.053 & 0.053 & 0.948 & 0.004 & 0.053 & 0.053 & 0.053 & 0.948\\
-0.50 & 500 & 0.001 & 0.032 & 0.034 & 0.032 & 0.966 & 0.002 & 0.032 & 0.034 & 0.032 & 0.964\\
-0.50 & 1000 & 0.001 & 0.025 & 0.024 & 0.025 & 0.954 & 0.002 & 0.025 & 0.024 & 0.025 & 0.952\\
-0.50 & 2000 & 0.000 & 0.016 & 0.017 & 0.016 & 0.954 & 0.002 & 0.016 & 0.017 & 0.016 & 0.952\\
-0.50 & 5000 & 0.000 & 0.011 & 0.011 & 0.011 & 0.942 & 0.001 & 0.011 & 0.011 & 0.011 & 0.942\\
\hline
-0.75 & 100 & 0.002 & 0.042 & 0.044 & 0.042 & 0.954 & 0.003 & 0.044 & 0.044 & 0.044 & 0.948\\
-0.75 & 200 & 0.002 & 0.031 & 0.031 & 0.031 & 0.942 & 0.004 & 0.031 & 0.031 & 0.031 & 0.950\\
-0.75 & 500 & 0.001 & 0.020 & 0.020 & 0.020 & 0.938 & 0.003 & 0.021 & 0.020 & 0.021 & 0.932\\
-0.75 & 1000 & -0.001 & 0.014 & 0.014 & 0.014 & 0.938 & 0.002 & 0.014 & 0.014 & 0.014 & 0.948\\
-0.75 & 2000 & -0.001 & 0.009 & 0.010 & 0.009 & 0.950 & 0.002 & 0.010 & 0.010 & 0.010 & 0.950\\
-0.75 & 5000 & 0.000 & 0.006 & 0.006 & 0.006 & 0.966 & 0.002 & 0.006 & 0.006 & 0.006 & 0.954\\
\hline
\end{tabular}
    \caption{Results based on 500 estimates of $\rho$ under Model 1, with $\rho = -0.25,-0.5,-0.75$. }
    \label{tab:est-model1-p2}
\end{table}

\begin{table}[htbp]
    \centering
    \footnotesize
\begin{tabular}{rr|rrrrr|rrrrr}
\hline
 & & \multicolumn{5}{c|}{RPCS} & \multicolumn{5}{c}{RPCF} \\
\hline
$\rho$ & $n$ & Bias & SD & $\wh {\rm SD}$ & RMSE & 95\% cvg & Bias & SD & $\wh {\rm SD}$ & RMSE & 95\% cvg\\
\hline
\hline
0.00 & 100 & 0.018 & 0.111 & 0.099 & 0.112 & 0.912 & 0.016 & 0.105 & 0.099 & 0.106 & 0.924\\
0.00 & 200 & 0.006 & 0.072 & 0.070 & 0.072 & 0.934 & 0.004 & 0.069 & 0.070 & 0.069 & 0.954\\
0.00 & 500 & 0.004 & 0.048 & 0.045 & 0.048 & 0.926 & 0.002 & 0.048 & 0.045 & 0.048 & 0.928\\
0.00 & 1000 & 0.003 & 0.032 & 0.032 & 0.032 & 0.930 & 0.002 & 0.031 & 0.032 & 0.031 & 0.942\\
0.00 & 2000 & 0.002 & 0.023 & 0.022 & 0.023 & 0.954 & 0.001 & 0.022 & 0.022 & 0.022 & 0.958\\
0.00 & 5000 & 0.001 & 0.014 & 0.014 & 0.014 & 0.950 & 0.001 & 0.014 & 0.014 & 0.014 & 0.960\\
\hline
0.25 & 100 & -0.007 & 0.104 & 0.093 & 0.104 & 0.916 & 0.003 & 0.096 & 0.093 & 0.096 & 0.934\\
0.25 & 200 & -0.009 & 0.074 & 0.066 & 0.074 & 0.918 & -0.003 & 0.069 & 0.066 & 0.069 & 0.936\\
0.25 & 500 & 0.000 & 0.045 & 0.042 & 0.045 & 0.930 & 0.003 & 0.043 & 0.042 & 0.044 & 0.944\\
0.25 & 1000 & -0.005 & 0.032 & 0.030 & 0.032 & 0.920 & -0.003 & 0.029 & 0.030 & 0.030 & 0.942\\
0.25 & 2000 & -0.003 & 0.021 & 0.021 & 0.021 & 0.944 & -0.002 & 0.020 & 0.021 & 0.020 & 0.948\\
0.25 & 5000 & -0.002 & 0.014 & 0.013 & 0.014 & 0.950 & -0.001 & 0.013 & 0.013 & 0.013 & 0.946\\
\hline
0.50 & 100 & -0.027 & 0.086 & 0.077 & 0.090 & 0.918 & -0.011 & 0.076 & 0.075 & 0.076 & 0.954\\
0.50 & 200 & -0.021 & 0.062 & 0.054 & 0.065 & 0.920 & -0.007 & 0.053 & 0.053 & 0.053 & 0.950\\
0.50 & 500 & -0.008 & 0.039 & 0.034 & 0.040 & 0.892 & -0.002 & 0.035 & 0.034 & 0.035 & 0.936\\
0.50 & 1000 & -0.005 & 0.030 & 0.024 & 0.030 & 0.912 & 0.000 & 0.024 & 0.024 & 0.024 & 0.942\\
0.50 & 2000 & -0.003 & 0.022 & 0.017 & 0.022 & 0.910 & -0.001 & 0.017 & 0.017 & 0.017 & 0.938\\
0.50 & 5000 & -0.002 & 0.013 & 0.011 & 0.014 & 0.924 & 0.000 & 0.011 & 0.011 & 0.011 & 0.948\\
\hline
0.75 & 100 & -0.049 & 0.070 & 0.050 & 0.085 & 0.842 & -0.019 & 0.049 & 0.046 & 0.053 & 0.946\\
0.75 & 200 & -0.029 & 0.045 & 0.034 & 0.054 & 0.888 & -0.010 & 0.030 & 0.032 & 0.032 & 0.962\\
0.75 & 500 & -0.019 & 0.037 & 0.021 & 0.041 & 0.884 & -0.006 & 0.019 & 0.020 & 0.020 & 0.962\\
0.75 & 1000 & -0.010 & 0.027 & 0.014 & 0.028 & 0.882 & -0.003 & 0.014 & 0.014 & 0.015 & 0.948\\
0.75 & 2000 & -0.005 & 0.016 & 0.010 & 0.017 & 0.928 & -0.002 & 0.010 & 0.010 & 0.010 & 0.942\\
0.75 & 5000 & -0.003 & 0.012 & 0.006 & 0.013 & 0.918 & -0.001 & 0.006 & 0.006 & 0.006 & 0.952\\
\hline
\end{tabular}

\vspace{0.3cm}
\begin{tabular}{rr|rrrrr|rrrrr}
\hline
 & & \multicolumn{5}{c|}{RPCO} & \multicolumn{5}{c}{PaCo} \\
\hline
$\rho$ & $n$ & Bias & SD & $\wh {\rm SD}$ & RMSE & 95\% cvg & Bias & SD & $\wh {\rm SD}$ & RMSE & 95\% cvg\\
\hline
\hline
0.00 & 100 & 0.005 & 0.105 & 0.099 & 0.105 & 0.948 & 0.004 & 0.112 & 0.099 & 0.112 & 0.910\\
0.00 & 200 & -0.002 & 0.068 & 0.070 & 0.068 & 0.962 & -0.003 & 0.071 & 0.070 & 0.071 & 0.950\\
0.00 & 500 & 0.000 & 0.047 & 0.045 & 0.047 & 0.932 & 0.000 & 0.048 & 0.045 & 0.048 & 0.924\\
0.00 & 1000 & 0.001 & 0.031 & 0.032 & 0.031 & 0.942 & 0.001 & 0.031 & 0.032 & 0.031 & 0.942\\
0.00 & 2000 & 0.001 & 0.022 & 0.022 & 0.022 & 0.954 & 0.001 & 0.022 & 0.022 & 0.022 & 0.956\\
0.00 & 5000 & 0.001 & 0.014 & 0.014 & 0.014 & 0.956 & 0.001 & 0.014 & 0.014 & 0.014 & 0.958\\
\hline
0.25 & 100 & 0.003 & 0.096 & 0.093 & 0.096 & 0.942 & 0.003 & 0.100 & 0.093 & 0.100 & 0.922\\
0.25 & 200 & -0.003 & 0.068 & 0.066 & 0.068 & 0.932 & -0.004 & 0.070 & 0.066 & 0.070 & 0.926\\
0.25 & 500 & 0.003 & 0.043 & 0.042 & 0.043 & 0.948 & 0.003 & 0.044 & 0.042 & 0.044 & 0.940\\
0.25 & 1000 & -0.002 & 0.029 & 0.030 & 0.029 & 0.946 & -0.003 & 0.029 & 0.030 & 0.029 & 0.940\\
0.25 & 2000 & -0.001 & 0.020 & 0.021 & 0.020 & 0.956 & -0.002 & 0.020 & 0.021 & 0.020 & 0.950\\
0.25 & 5000 & -0.001 & 0.013 & 0.013 & 0.013 & 0.952 & -0.001 & 0.013 & 0.013 & 0.013 & 0.946\\
\hline
0.50 & 100 & 0.001 & 0.074 & 0.074 & 0.074 & 0.952 & -0.002 & 0.080 & 0.075 & 0.080 & 0.938\\
0.50 & 200 & 0.000 & 0.053 & 0.053 & 0.053 & 0.948 & -0.002 & 0.054 & 0.053 & 0.054 & 0.952\\
0.50 & 500 & 0.001 & 0.034 & 0.033 & 0.034 & 0.938 & 0.000 & 0.035 & 0.034 & 0.035 & 0.930\\
0.50 & 1000 & 0.002 & 0.024 & 0.024 & 0.024 & 0.934 & 0.001 & 0.025 & 0.024 & 0.025 & 0.936\\
0.50 & 2000 & 0.001 & 0.017 & 0.017 & 0.017 & 0.932 & 0.000 & 0.017 & 0.017 & 0.017 & 0.936\\
0.50 & 5000 & 0.001 & 0.010 & 0.011 & 0.010 & 0.952 & 0.000 & 0.011 & 0.011 & 0.011 & 0.946\\
\hline
0.75 & 100 & 0.001 & 0.046 & 0.043 & 0.046 & 0.932 & 0.001 & 0.048 & 0.043 & 0.048 & 0.912\\
0.75 & 200 & 0.001 & 0.030 & 0.031 & 0.030 & 0.942 & -0.001 & 0.031 & 0.031 & 0.031 & 0.950\\
0.75 & 500 & -0.001 & 0.019 & 0.020 & 0.019 & 0.960 & -0.002 & 0.019 & 0.020 & 0.019 & 0.958\\
0.75 & 1000 & 0.000 & 0.014 & 0.014 & 0.014 & 0.944 & -0.002 & 0.014 & 0.014 & 0.015 & 0.946\\
0.75 & 2000 & 0.000 & 0.010 & 0.010 & 0.010 & 0.944 & -0.002 & 0.010 & 0.010 & 0.010 & 0.950\\
0.75 & 5000 & 0.000 & 0.006 & 0.006 & 0.006 & 0.952 & -0.001 & 0.006 & 0.006 & 0.006 & 0.954\\
\hline
\end{tabular}
    \caption{Results based on 500 estimates of $\rho$ under Model 2, with $\rho = 0,0.25,0.5,0.75$. }
    \label{tab:est-model2-p1}
\end{table}

\begin{table}[htbp]
    \centering
    \footnotesize
\begin{tabular}{rr|rrrrr|rrrrr}
\hline
 & & \multicolumn{5}{c|}{RPCS} & \multicolumn{5}{c}{RPCF} \\
\hline
$\rho$ & $n$ & Bias & SD & $\wh {\rm SD}$ & RMSE & 95\% cvg & Bias & SD & $\wh {\rm SD}$ & RMSE & 95\% cvg\\
\hline
\hline
-0.25 & 100 & 0.037 & 0.100 & 0.094 & 0.107 & 0.902 & 0.024 & 0.095 & 0.094 & 0.098 & 0.942\\
-0.25 & 200 & 0.028 & 0.074 & 0.067 & 0.080 & 0.904 & 0.015 & 0.072 & 0.066 & 0.073 & 0.920\\
-0.25 & 500 & 0.015 & 0.042 & 0.042 & 0.044 & 0.946 & 0.009 & 0.040 & 0.042 & 0.041 & 0.968\\
-0.25 & 1000 & 0.006 & 0.032 & 0.030 & 0.033 & 0.928 & 0.002 & 0.030 & 0.030 & 0.030 & 0.950\\
-0.25 & 2000 & 0.002 & 0.020 & 0.021 & 0.021 & 0.966 & 0.000 & 0.020 & 0.021 & 0.020 & 0.960\\
-0.25 & 5000 & 0.001 & 0.013 & 0.013 & 0.013 & 0.954 & 0.000 & 0.012 & 0.013 & 0.012 & 0.958\\
\hline
-0.50 & 100 & 0.061 & 0.084 & 0.080 & 0.104 & 0.894 & 0.036 & 0.075 & 0.078 & 0.083 & 0.952\\
-0.50 & 200 & 0.040 & 0.062 & 0.055 & 0.073 & 0.908 & 0.019 & 0.054 & 0.054 & 0.057 & 0.952\\
-0.50 & 500 & 0.021 & 0.043 & 0.034 & 0.048 & 0.888 & 0.008 & 0.035 & 0.034 & 0.036 & 0.924\\
-0.50 & 1000 & 0.013 & 0.031 & 0.024 & 0.034 & 0.904 & 0.005 & 0.024 & 0.024 & 0.025 & 0.948\\
-0.50 & 2000 & 0.005 & 0.019 & 0.017 & 0.019 & 0.942 & 0.002 & 0.017 & 0.017 & 0.017 & 0.952\\
-0.50 & 5000 & 0.003 & 0.011 & 0.011 & 0.011 & 0.926 & 0.002 & 0.011 & 0.011 & 0.011 & 0.940\\
\hline
-0.75 & 100 & 0.075 & 0.067 & 0.054 & 0.100 & 0.782 & 0.043 & 0.051 & 0.050 & 0.067 & 0.916\\
-0.75 & 200 & 0.049 & 0.047 & 0.036 & 0.068 & 0.760 & 0.023 & 0.034 & 0.033 & 0.041 & 0.932\\
-0.75 & 500 & 0.026 & 0.038 & 0.021 & 0.046 & 0.780 & 0.010 & 0.020 & 0.020 & 0.022 & 0.952\\
-0.75 & 1000 & 0.013 & 0.025 & 0.014 & 0.029 & 0.860 & 0.004 & 0.014 & 0.014 & 0.015 & 0.946\\
-0.75 & 2000 & 0.006 & 0.018 & 0.010 & 0.019 & 0.902 & 0.003 & 0.010 & 0.010 & 0.011 & 0.930\\
-0.75 & 5000 & 0.004 & 0.010 & 0.006 & 0.010 & 0.908 & 0.002 & 0.006 & 0.006 & 0.007 & 0.938\\
\hline
\end{tabular}

\vspace{0.3cm}
\begin{tabular}{rr|rrrrr|rrrrr}
\hline
 & & \multicolumn{5}{c|}{RPCO} & \multicolumn{5}{c}{PaCo} \\
\hline
$\rho$ & $n$ & Bias & SD & $\wh {\rm SD}$ & RMSE & 95\% cvg & Bias & SD & $\wh {\rm SD}$ & RMSE & 95\% cvg\\
\hline
\hline
-0.25 & 100 & 0.002 & 0.094 & 0.093 & 0.094 & 0.930 & 0.003 & 0.097 & 0.093 & 0.097 & 0.928\\
-0.25 & 200 & 0.004 & 0.070 & 0.066 & 0.071 & 0.918 & 0.005 & 0.073 & 0.066 & 0.073 & 0.908\\
-0.25 & 500 & 0.004 & 0.040 & 0.042 & 0.040 & 0.962 & 0.005 & 0.041 & 0.042 & 0.041 & 0.958\\
-0.25 & 1000 & 0.000 & 0.030 & 0.030 & 0.030 & 0.948 & 0.000 & 0.030 & 0.030 & 0.030 & 0.944\\
-0.25 & 2000 & -0.001 & 0.020 & 0.021 & 0.020 & 0.960 & 0.000 & 0.020 & 0.021 & 0.020 & 0.964\\
-0.25 & 5000 & -0.001 & 0.012 & 0.013 & 0.012 & 0.956 & 0.000 & 0.012 & 0.013 & 0.012 & 0.958\\
\hline
-0.50 & 100 & 0.002 & 0.073 & 0.075 & 0.073 & 0.944 & 0.005 & 0.077 & 0.075 & 0.077 & 0.946\\
-0.50 & 200 & 0.002 & 0.055 & 0.053 & 0.055 & 0.948 & 0.003 & 0.056 & 0.053 & 0.056 & 0.946\\
-0.50 & 500 & 0.002 & 0.035 & 0.034 & 0.035 & 0.938 & 0.002 & 0.035 & 0.034 & 0.035 & 0.932\\
-0.50 & 1000 & 0.002 & 0.024 & 0.024 & 0.024 & 0.952 & 0.003 & 0.024 & 0.024 & 0.024 & 0.950\\
-0.50 & 2000 & 0.000 & 0.017 & 0.017 & 0.017 & 0.956 & 0.001 & 0.017 & 0.017 & 0.017 & 0.956\\
-0.50 & 5000 & 0.000 & 0.011 & 0.011 & 0.011 & 0.952 & 0.001 & 0.011 & 0.011 & 0.011 & 0.942\\
\hline
-0.75 & 100 & -0.001 & 0.044 & 0.043 & 0.044 & 0.940 & 0.001 & 0.048 & 0.044 & 0.047 & 0.926\\
-0.75 & 200 & 0.001 & 0.033 & 0.031 & 0.033 & 0.924 & 0.003 & 0.034 & 0.031 & 0.034 & 0.930\\
-0.75 & 500 & 0.001 & 0.019 & 0.020 & 0.019 & 0.960 & 0.002 & 0.020 & 0.020 & 0.020 & 0.956\\
-0.75 & 1000 & 0.000 & 0.014 & 0.014 & 0.014 & 0.952 & 0.001 & 0.014 & 0.014 & 0.014 & 0.944\\
-0.75 & 2000 & 0.000 & 0.010 & 0.010 & 0.010 & 0.952 & 0.001 & 0.010 & 0.010 & 0.010 & 0.942\\
-0.75 & 5000 & 0.000 & 0.006 & 0.006 & 0.006 & 0.960 & 0.001 & 0.006 & 0.006 & 0.006 & 0.948\\
\hline
\end{tabular}
    \caption{Results based on 500 estimates of $\rho$ under Model 2, with $\rho = -0.25,-0.5,-0.75$. }
    \label{tab:est-model2-p2}
\end{table}

\begin{table}[htbp]
    \centering
    \footnotesize
\begin{tabular}{rr|rrrrr|rrrrr}
\hline
 & & \multicolumn{5}{c|}{RPCS} & \multicolumn{5}{c}{RPCF} \\
\hline
$\rho$ & $n$ & Bias & SD & $\wh {\rm SD}$ & RMSE & 95\% cvg & Bias & SD & $\wh {\rm SD}$ & RMSE & 95\% cvg\\
\hline
\hline
0 & 100 & 0.033 & 0.117 & 0.099 & 0.122 & 0.874 & 0.032 & 0.103 & 0.099 & 0.108 & 0.930\\
0 & 200 & 0.021 & 0.080 & 0.070 & 0.083 & 0.906 & 0.023 & 0.071 & 0.070 & 0.075 & 0.926\\
0 & 500 & 0.015 & 0.048 & 0.045 & 0.051 & 0.938 & 0.013 & 0.045 & 0.045 & 0.047 & 0.944\\
0 & 1000 & 0.013 & 0.038 & 0.032 & 0.040 & 0.886 & 0.010 & 0.034 & 0.032 & 0.035 & 0.918\\
0 & 2000 & 0.008 & 0.026 & 0.022 & 0.027 & 0.892 & 0.007 & 0.024 & 0.022 & 0.025 & 0.914\\
0 & 5000 & 0.003 & 0.016 & 0.014 & 0.016 & 0.894 & 0.002 & 0.015 & 0.014 & 0.015 & 0.938\\
\hline
\end{tabular}

\vspace{0.3cm}
\begin{tabular}{rr|rrrrr|rrrrr}
\hline
 & & \multicolumn{5}{c|}{RPCO} & \multicolumn{5}{c}{PaCo} \\
\hline
$\rho$ & $n$ & Bias & SD & $\wh {\rm SD}$ & RMSE & 95\% cvg & Bias & SD & $\wh {\rm SD}$ & RMSE & 95\% cvg\\
\hline
\hline
0 & 100 & -0.003 & 0.102 & 0.099 & 0.102 & 0.944 & 0.144 & 0.104 & 0.097 & 0.178 & 0.650\\
0 & 200 & 0.002 & 0.070 & 0.070 & 0.070 & 0.958 & 0.135 & 0.068 & 0.069 & 0.151 & 0.536\\
0 & 500 & -0.001 & 0.043 & 0.045 & 0.042 & 0.956 & 0.142 & 0.043 & 0.044 & 0.148 & 0.096\\
0 & 1000 & 0.002 & 0.032 & 0.032 & 0.032 & 0.950 & 0.144 & 0.031 & 0.031 & 0.147 & 0.000\\
0 & 2000 & 0.000 & 0.022 & 0.022 & 0.022 & 0.940 & 0.144 & 0.022 & 0.022 & 0.146 & 0.000\\
0 & 5000 & 0.000 & 0.014 & 0.014 & 0.014 & 0.956 & 0.144 & 0.013 & 0.014 & 0.144 & 0.000\\
\hline
\end{tabular}
    \caption{Results based on 500 estimates of $\rho$ under Model 3, with $\rho = 0$. }
    \label{tab:est-model3}
\end{table}

\end{document}